\documentclass[11pt]{amsart}

\usepackage{amsmath,amsthm,amssymb,amscd}
\usepackage[margin=1in]{geometry}
\usepackage[bookmarks]{hyperref}
\usepackage{verbatim}
\usepackage{xcolor}

\newtheorem{theorem}{Theorem}[subsection]
\newtheorem{proposition}[theorem]{Proposition}

\newtheorem{corollary}[theorem]{Corollary}

\theoremstyle{definition}

\newtheorem{definition}[theorem]{Definition}
\newtheorem{example}[theorem]{Example}

\numberwithin{equation}{subsection}
\numberwithin{table}{subsection}

\newcommand{\C}{{\mathbb C}}
\newcommand{\Z}{{\mathbb Z}}

\newcommand{\Ext}{\operatorname{Ext}}

\newcommand{\0}{\bar 0}

\newcommand{\1}{\bar 1}

\newcommand{\g}{\ensuremath{\mathfrak{g}}}

\newcommand{\e}{\ensuremath{\mathfrak{e}}}
\newcommand{\f}{\ensuremath{\mathfrak{f}}}

\newcommand{\res}{\ensuremath{\operatorname{res}}}

\begin{document}

\title[On BBW Parabolics for Simple Classical Lie Superalgebras]
{\bf On BBW Parabolics for Simple Classical Lie Superalgebras}

\author{\sc Dimitar Grantcharov}
\address
{University of Texas at Arlington  \\
Department of Mathematics               \\
Arlington\\ TX~76019}
\thanks{Research of the first author was supported in part by Simons Collaboration Grant 358245.}
\email{grandim@uta.edu}

\author{\sc Nikolay Grantcharov}
\address
{Department of Mathematics \\
University of Chicago   \\
Chicago \\ IL~60637}
\thanks{Research of the second author was supported in part by NSF Grant DGE-1746045.}

\email{nikolayg@uchicago.edu}

\author{\sc Daniel K. Nakano}
\address
{Department of Mathematics\\ University of Georgia \\
Athens\\ GA~30602, USA}
\thanks{Research of the third author was supported in part by NSF
grant  DMS-1701768. The material is based upon work supported by NSF grant 
DMS-1440140 while the third author was in residence at MSRI during Spring 2018.} 

\email{nakano@uga.edu}

\author{\sc Jerry Wu}
\address 
{Department of Mathematics \\
University of California, Berkeley   \\
Berkeley \\ CA~94720-3840  }

\email{jerryxzwu@gmail.com}


\date\today

\maketitle

\begin{abstract} In this paper the authors introduce a class of parabolic subalgebras for classical simple Lie superalgebras associated to the detecting subalgebras 
introduced by Boe, Kujawa and Nakano. These parabolic subalgebras are shown to have good cohomological properties governed 
by the Bott-Borel-Weil theorem involving the zero component of the Lie superalgebra in conjunction with the odd roots. These results are later used to 
verify an open conjecture given by Boe, Kujawa and Nakano pertaining to the equality of various support varieties. 
\end{abstract} 
\vskip 1cm 

\section{Introduction}

\subsection{} Let ${\mathfrak g}$ be a classical simple Lie superalgebra over ${\mathbb C}$ and $G$ be the corresponding supergroup (scheme) with 
$\text{Lie }G={\mathfrak g}$. Given a parabolic subgroup scheme $P$, a major open question has been to compute the higher 
sheaf cohomology group $R^{j}\text{ind}_{P}^{G} N$ for $j\geq 0$ where $N$ is a finite-dimensional $P$-module. General theory on this topic can be found in \cite{Zubkov}, and some computations for Lie superalgebras 
such as $\mathfrak{gl}(m|n)$, $\mathfrak{osp}(m|2n)$, and $\mathfrak{q}(n)$ are presented in \cite{Zubkov,GrS1,GrS2,P,PS,S1,S2}. For reductive algebraic groups, when 
$P$ is a Borel subgroup and $N$ is a one-dimensional module, the answer is given by the classical Bott-Borel-Weil (BBW) theorem. 

In this paper we introduce parabolic subsupergroups $P=B$ such that the higher sheaf cohomology $R^{j}\text{ind}_{B}^{G} (-)$ can be computed 
using data from the BBW theorem. These subgroups are obtained by using the detecting subalgebras via the 
stable action of $G_{\0}$ on ${\mathfrak g}_{\1}$. The striking features about these subalgebras is the interplay between the even roots and 
the odd roots with their associated finite reflection groups, and the fact that our approach allows for a uniform treatment of all classical simple Lie superalgebras. 
In particular as a byproduct of our work, we obtain an important computation of the higher sheaf cohomology groups of $G/B$ for the trivial line bundle: ${\mathcal H}^{j}(G/B,{\mathcal L}(0)):=
R^{j}\text{ind}_{B}^{G} {\mathbb C}$ for $j\geq 0$ (cf. Theorem~\ref{T:Poincareseriesequal}). For classical simple Lie superalgebras other than ${\mathfrak p}(n)$, it is shown that the polynomial 
$p_{G,B}(t)=\sum_{i=0}^{\infty} \dim R^{i}\text{ind}_{B}^{G} {\mathbb C}\ t^{i}$ is equal to a Poincar\'e polynomial for a finite reflection group $W_{\1}$ specialized at a power of $t$. 
This indicates that the combinatorics of the length function on $W_{\1}$ plays in important role in this setting, and opens the possibilities for developing a 
general theory involving these parabolic subsupergroups. 

\subsection{} For finite groups it is well-known that the cohomology is detected on the collection of elementary abelian $p$-subgroups. 
Moreover, Quillen \cite{quillen1,quillen2} showed that these subgroups can be used to describe the spectrum of the cohomology ring. Later 
Avrunin and Scott \cite{AS} demonstrated that the support varieties for finite groups consist of taking unions of support 
varieties for elementary abelian subgroup whose varieties can be described using rank varieties. 

In the study of classical simple Lie superalgebras, Boe, Kujawa and Nakano \cite{BKN1} used invariant theory for reductive groups 
to show that there are natural classes of ``subalgebras'' that detect the cohomology. These subalgebras come in one of two families: ${\mathfrak f}$ (when 
$\g$ is stable) and ${\mathfrak e}$ (when $\g$ is polar). In all cases, $\g$ admits a stable action and in most cases $\g$ admits a polar action 
(cf \cite[Table 5]{BKN1}). 

In this situation, the restriction maps induce isomorphisms:\footnote{There are some errors in the statements in \cite{BKN1} and \cite{LNZ}. In these papers 
``$\text{H}^{\bullet}(\f,\f_{\0},{\mathbb C})^{N/N_{\0}}$" should be replaced with ``$\text{H}^{\bullet}(\f,\f_{\0},{\mathbb C})^{N}$'' and 
``${\mathcal V}_{(\f,\f_{\0})}(M)/(N/N_{\0})$" should be replaced with ``${\mathcal V}_{(\f,\f_{\0})}(M)/N$". }
$$\text{H}^{\bullet}(\g,\g_{\0},{\mathbb C})\cong
\text{H}^{\bullet}(\f,\f_{\0},{\mathbb C})^{N}\cong
\text{H}^{\bullet}(\e,\e_{\0},{\mathbb C})^{W_{\e}}$$
where $N$ is a reductive group and $W_{\e}$ is a finite pseudoreflection group. These
relative cohomology rings may be identified with the invariant ring
$S^{\bullet}({\mathfrak g}_{\1}^{*})^{G_{\0}}$, where $S^{\bullet}$ denotes the symmetric
algebra, and so are finitely generated. This property was used to construct support varieties for modules in the 
category ${\mathcal F}_{(\g,\g_{\0})}$ (i.e., finite-dimensional ${\mathfrak g}$-modules that are completely reducible over ${\mathfrak g}_{\0}$). 

The main application of the existence and properties of the BBW type parabolic subalgebras is our verification of the following theorem. 

\begin{theorem}\label{T:isosupports}  Let $\g$ be a simple classical Lie superalgebra and let $M$ be in ${\mathcal F_{(\g,\g_{\0})}}$.
\begin{itemize}
\item[(a)] If $\g$ is stable then the map on support varieties
$$\operatorname{res}^{*}:{\mathcal V}_{(\f,\f_{\0})}(M)/N
\rightarrow {\mathcal V}_{(\g,\g_{\0})}(M)$$ is an isomorphism.
\item[(b)] If $\g$ is stable and polar then the maps on support varieties
$$\operatorname{res}^{*}:
{\mathcal V}_{(\e,\e_{\0})}(M)/W_{\e}\rightarrow {\mathcal V}_{(\f,\f_{\0})}(M)/N
\rightarrow {\mathcal V}_{(\g,\g_{\0})}(M)$$
are isomorphisms, where $W_{\e}$ is a pseudoreflection group.
\end{itemize}
\end{theorem} 

The aforementioned theorem has been a conjecture that was first introduced in \cite{BKN1}. In that paper, the equality of the varieties in Theorem~\ref{T:isosupports} 
was shown to hold on the complement of the discriminant locus (i.e., an open dense set). This provided strong evidence for the validity of the conjecture. Later, Lehrer, Nakano and Zhang \cite{LNZ} proved the conjecture for the general linear Lie superalgebra and more generally type I classical simple Lie superalgebra via a cohomological embedding theorem. 

Kac and Wakimoto defined a combinatorial invariant called the atypicality of a weight $\lambda$ when ${\mathfrak g}$ is a basic classical simple Lie superalgebra. 
The support varieties in Theorem~\ref{T:isosupports} play a prominent role in the theory because they provide a geometric interpretation of this combinatorial invariant. 
It is conjectured that for the basic simple Lie superalgebras, the dimension of the support variety ${\mathcal V}_{(\g,\g_{\0})}(L(\lambda))$ equals the atypicality of the finite-dimensional irreducible representation $L(\lambda)$. This has been verified in a 
number of cases including $\mathfrak{gl}(m|n)$ \cite{BKN2} and $\mathfrak{osp}(m|2n)$ \cite{K}.  

\subsection{} For the detecting subalgebra $\e$ one has a realization of the support variety ${\mathcal V}_{(\e,\e_{\0})}(M)$ as 
a rank variety: 
$${\mathcal V}_{(\e,\e_{\0})}(M)\cong {\mathcal V}^{\operatorname{rank}}_{(\e,\e_{\0})}(M):=\{x\in \e_{\1}:\ M|_{U(\langle x \rangle)}\ \text{is not projective}\}\cup \{0\}.$$ 
The establishment of Theorem~\ref{T:isosupports} along with  this rank variety description (i) provides a concrete realization of ${\mathcal V}_{(\g,\g_{\0})}(M)$ and 
(ii) shows that the assignment $(-)\rightarrow {\mathcal V}_{(\g,\g_{\0})}(-)$ satisfies the properties as stated in \cite{Bal} for a support datum. These important properties are stated in the 
following corollary. 

\begin{corollary} \label{T:supportprop} Let ${\mathfrak g}$ be a simple classical Lie superalgebra 
which is both stable and
polar, and let $M_{1}$, $M_{2}$ and $M$ be in ${\mathcal F}_{(\g,\g_{\0})}$. \begin{itemize}
\item[(a)] ${\mathcal V}_{(\g,\g_{\0})}(M)\cong {\mathcal V}^{\operatorname{rank}}_{(\e,\e_{\0})}(M)/W_{\e}$;
\item[(b)] ${\mathcal V}_{(\g,\g_{\0})}(M_{1}\otimes M_{2})=
{\mathcal V}_{(\g,\g_{\0})}(M_{1})\cap {\mathcal V}_{(\g,\g_{\0})}(M_{2})$.
\item[(c)] Let $X$ be a conical subvariety of ${\mathcal V}_{(\g,\g_{\0})}({\mathbb C})$. Then
there exists $L$ in ${\mathcal F}_{(\g,\g_{\0})}$ with $X={\mathcal V}_{(\g,\g_{\0})}(L)$.
\item[(d)] If $M$ is indecomposable then $\operatorname{Proj}({\mathcal V}_{(\g,\g_{\0})}(M))$
is connected.
\end{itemize}
\end{corollary}

Note that the verification of the corollary above follows by the same line of reasoning as given in \cite[Theorem 5.2.1]{LNZ}.

\subsection{} The paper is organized as follows. In the next section, Section~\ref{S:prelim}, the structure theory for the detecting subalgebras 
and their relationship to various support variety theories is reviewed. Given a classical simple Lie superalgebra, ${\mathfrak g}$, the construction of 
a parabolic subalgebra, ${\mathfrak b}$, that is generated by the negative Borel subalgebra for ${\mathfrak g}_{\0}$ and the detecting subalgebra 
${\mathfrak f}$ is presented in Section~\ref{S:parabolicconstruction}. These parabolics are defined via hyperplanes in the span of the roots in a Euclidean space. 
A comparison theorem is proved between the relative cohomology for $({\mathfrak b},{\mathfrak b}_{\0})$ and $({\mathfrak f},{\mathfrak f}_{\0})$ (cf. Theorem~\ref{t:b-tcoho}) 
and between the relative cohomology for $({\mathfrak g},{\mathfrak g}_{\0})$ and $({\mathfrak b},{\mathfrak b}_{\0})$ 
(cf. Theorem~\ref{T:paraproperties}). The latter relationship involves a natural grading on the group algebra of a finite reflection group $W_{\1}$. 

In Section 4, we investigate sheaf cohomology for $G/B$ where ${\mathfrak g}=\text{Lie }G$ and ${\mathfrak b}=\text{Lie }B$. 
In particular, we consider the Poincar\'e series, $p_{G,B}(t)=\sum_{i=0}^{\infty} \dim R^{i}\text{ind}_{B}^{G} {\mathbb C}\ t^{i}$ and give a complete computation 
for all Lie superalgebras except when ${\mathfrak g}={\mathfrak p}(n)$. It is shown that $p_{G,B}(t)$ is directly related to the standard Poincar\'e 
polynomial of $W_{\1}$ via the natural length function on the  finite reflection group $W_{\1}$ (cf. Table~\ref{T:dimensions}). Our calculations use an intricate 
and detailed analysis of the 
(odd) dot action of $W_{\1}$ on a natural subset, $\Phi_{\1}$, of odd roots. Section 5 is devoted to investigating the situation for ${\mathfrak g}={\mathfrak p}(n)$. 
For ${\mathfrak p}(2)$ and ${\mathfrak p}(3)$ it is shown that $p_{G,B}(t)$ is governed by the BBW theorem. However, for ${\mathfrak p}(4)$ 
this is not the case and open questions are presented at the end of this section. 

Finally, in Section 6, we indicate how our computation fit into a more functorial setting involving natural spectral sequences (see Theorem~\ref{T:sscollapse} and 
Theorem~\ref{T:gbsupports}). For all classical Lie superalgebras with the possible exception of ${\mathfrak g}={\mathfrak p}(n)$, it is shown that the 
spectral sequence in Theorem~\ref{T:sscollapse} collapses. This result enables us to prove the conjecture involving the equality of 
supports stated as Theorem~\ref{T:isosupports}. 

\subsection{Acknowledgements} The third author would like to acknowledge the support and hospitality of the Mathematical Sciences Research 
Institute (MSRI) during his stay as a General Member in Spring 2018. Many of the results in the paper were obtained during this time with 
weekly meetings after Wednesday Tea with the other coauthors. We also thank Matthew Douglass, Chun-Ju Lai and the referee for their comments and suggestions on an earlier version of this manuscript. 

\section{Preliminaries}\label{S:prelim}

\subsection{Notation: } We will use and summarize the conventions developed in
\cite{BKN1, BKN2, BKN3}. For more details we refer the reader to \cite[Section 2]{BKN1}.

Throughout this paper, let ${\mathfrak a}$ be a Lie superalgebra over the complex numbers ${\mathbb C}$.
In particular, ${\mathfrak a}={\mathfrak a}_{\0}\oplus {\mathfrak a}_{\1}$ is a $\Z_{2}$-graded vector space
with a supercommutator $[\;,\;]:{\mathfrak a}\otimes {\mathfrak a} \to {\mathfrak a}$. A finite-dimensional Lie superalgebra
${\mathfrak a}$ is called \emph{classical} if there is a connected reductive algebraic group $A_{\0}$ such
that $\operatorname{Lie}(A_{\0})={\mathfrak a}_{\0},$ and the action of $A_{\0}$ on ${\mathfrak a}_{\1}$ differentiates to
the adjoint action of ${\mathfrak a}_{\0}$ on ${\mathfrak a}_{\1}.$  The Lie superalgebra ${\mathfrak a}$ is  \emph{basic classical}
if it is a classical Lie superalgebra with a nondegenerate invariant supersymmetric even bilinear form. In this paper our main focus will be on 
classical ``simple'' Lie superalgebras. The algebras of interest are listed in Table~\ref{T:dimensions}. Although some of these Lie superalgebras 
are not simple in the true sense, they are close enough to being simple and are ones of general interest. With a slight abuse of notation we will let  $A(m|n)$ denote 
the Lie superalgebras $\mathfrak{gl}(m|n)$ and $\mathfrak{sl}(m|n)$ for $m\neq n$ and $\mathfrak{sl}(n|n)$ and $\mathfrak{psl}(n|n)$ for $m=n$. For the Lie superalgebras of type Q we use the notation of \cite{PS}. Namely, $\mathfrak{q}(n)$ will be the Lie superalgebra with even and odd parts $\mathfrak{gl}_n$, while $\mathfrak{psq}(n)$ is the corresponding simple subquotient of $\mathfrak{q}(n)$. The Lie superalgebras 
that fall into the family of type P will be denoted by $P(n)$. These algebras include ${\mathfrak p}(n)$ and its enlargement $\widetilde{{\mathfrak p}}(n)$.

Let $U({\mathfrak a})$ be the universal enveloping superalgebra of ${\mathfrak a}$. Supermodules 
are $\Z_{2}$-graded left $U({\mathfrak a})$-modules. If $M$ and $N$ are ${\mathfrak a}$-supermodules
one can use the antipode and coproduct of $U({\mathfrak a})$ to define a ${\mathfrak a}$-supermodule
structure on the dual $M^{*}$ and the tensor product $M\otimes N$. For the remainder of the
paper the term ${\mathfrak a}$-module will mean a ${\mathfrak a}$-supermodule. 

Let ${\mathfrak a}$ be an arbitrary Lie superalgebra (not necessary classical).  In order to apply homological algebra techniques, we will restrict ourselves to the
\emph{underlying even category}, consisting of ${\mathfrak a}$-modules with the degree preserving morphisms.
In this paper we will study homological properties of the category of ${\mathfrak a}$-modules where the 
projective objects are relatively projective $U({\mathfrak a}_{\0})$-modules. Given ${\mathfrak a}$-modules, 
$M, N$, let $\text{Ext}^{n}_{({\mathfrak a},{\mathfrak a}_{\0})}(M,N)$ denote the $n$-extension group 
defined by using a relatively projective $U({\mathfrak a}_{\0})$-resolution for $M$. Under the conditions that 
either ${\mathfrak a}_{\1}$ is finitely semisimple over ${\mathfrak a}_{\0}$ or ${\mathfrak a}={\mathfrak a}_{\0}\oplus {\mathfrak a}_{\1}$ is a direct sum of ${\mathfrak a}_{\0}$-modules (cf. 
 \cite[3.1.8 Corollary, 3.1.15 Remark]{Kum}),  there is a concrete realization for these
extension groups via the relative Lie superalgebra cohomology for the pair
$({\mathfrak a},{\mathfrak a}_{\0})$:
$$\text{Ext}^{n}_{({\mathfrak a},{\mathfrak a}_{\0})}(M,N)\cong \text{H}^{n}({\mathfrak a},{\mathfrak a}_{\0}; M^{*}\otimes N).$$
The later cohomology group can be computed using an explicit complex. For a detailed discussion about the complex to compute relative Lie superalgebra 
cohomology the reader is referred to \cite[Section 2.3]{BKN1}. Set 
\begin{equation} 
p_{\mathfrak a}(t)=\sum_{i=0}^{\infty} \dim \operatorname{H}^{i}({\mathfrak a},{\mathfrak a}_{\0},{\mathbb C}) t^{i}.
\end{equation}

When ${\mathfrak a}$ is a classical Lie superalgebra, let ${\mathcal F}_{({\mathfrak a},{\mathfrak a}_{\0})}$ be  the full subcategory of
finite-dimensional ${\mathfrak a}$-modules which are finitely semisimple over ${\mathfrak a}_{\0}$ (a
${\mathfrak a}_{\0}$-module is \emph{finitely semisimple} if it decomposes into a direct sum of
finite-dimensional simple ${\mathfrak a}_{\0}$-modules). The projectives in the category ${\mathcal F}:={\mathcal F}_{({\mathfrak a},{\mathfrak a}_{\0})}$ 
are the finite-dimensional relatively projective $U({\mathfrak a}_{\0})$-modules. Moreover, ${\mathcal F}_{({\mathfrak a},{\mathfrak a}_{\0})}$
is a Frobenius category (i.e., where injectivity is equivalent to projectivity) \cite{BKN3}.
Given $M, N$ in ${\mathcal F}$, 
$\Ext_{\mathcal{F}}^{n}(M,N)\cong \text{Ext}^{n}_{({\mathfrak a},{\mathfrak a}_{\0})}(M,N)$. 
Let $R$ be the cohomology ring
$$
\text{H}^{\bullet}({\mathfrak a},{\mathfrak a}_{\0}; {\mathbb C})=
S^{\bullet}({\mathfrak a}_{\1}^*)^{{\mathfrak a}_{\0}}\cong S^{\bullet}({\mathfrak a}_{\1}^*)^{A_{\0}}.
$$
The last isomorphism holds because $A_{\0}$ is reductive and acts semisimply on the symmetric algebra. 
Moreover, since $A_{\0}$ is reductive it follows that $R$ is finitely generated. 

\subsection{Support varieties: } We recall the definition of the support variety
of a finite-dimensional ${\mathfrak a}$-supermodule $M$ (cf. \cite[Section 6.1]{BKN1}). 
Let ${\mathfrak a}$ be a classical Lie superalgebra,
$R:=\operatorname{H}^{\bullet}({\mathfrak a}, {\mathfrak a}_{\0};{\mathbb C})$, and $M_{1}$, $M_{2}$
be in ${\mathcal F}:={\mathcal F}_{({\mathfrak a},{\mathfrak a}_{\0})}$. According to \cite[Theorem 2.5.3]{BKN1},
$\Ext_{\mathcal{F}}^{\bullet}(M_{1},M_{2})$ is a finitely generated $R$-module.
Set $J_{({\mathfrak a},{\mathfrak a}_{\0})}(M_{1},M_{2})=
\operatorname{Ann}_{R}(\Ext_{\mathcal{F}}^{\bullet}(M_{1},M_{2}))$
(i.e., the annihilator ideal of this module).  The \emph{relative support variety of the pair $(M,N)$} is

\begin{equation}
\mathcal{V}_{({\mathfrak a},{\mathfrak a}_{\0})}(M,N)=
\operatorname{MaxSpec}(R/J_{({\mathfrak a},{\mathfrak a}_{\0})}(M,N))
\end{equation}

In the case when $M=M_{1}=M_{2}$, set $J_{({\mathfrak a},{\mathfrak a}_{\0})}(M)=J_{({\mathfrak a},{\mathfrak a}_{\0})}(M,M)$, and
$$\mathcal{V}_{({\mathfrak a},{\mathfrak a}_{\0})}(M):=\mathcal{V}_{({\mathfrak a},{\mathfrak a}_{\0})}(M,M).$$
The variety $\mathcal{V}_{(\mathfrak{a},\mathfrak{a}_{\0})}(M)$ is called the \emph{support variety} of $M$.
In this situation, $J_{({\mathfrak a},{\mathfrak a}_{\0})}(M)=\text{Ann}_{R} \ \text{Id}$ where $\text{Id}$ is the identity
morphism in $\text{Ext}^{0}_{\mathcal F}(M,M)$.

\subsection{Structure theory for the detecting subalgebras} The main ideas used in constructing 
the detecting subalgebras ${\mathfrak f}$ and ${\mathfrak e}$ for classical simple Lie superalgebras are summarized below. 

Let ${\mathfrak g}$ be a classical simple Lie superalgebra as described in \cite[Section 8]{BKN1}. 
It was shown that the action of $G_{\0}$ on ${\mathfrak g}_{\1}$ admits a {\em stable} action. The reader is 
referred to \cite[Section 3.2]{BKN1} for a detailed exposition on stable actions. 

Fix a generic element $x_{0} \in \g_{\1}$ 
(cf. \cite[Section 8.9]{BKN1} for an explicit construction). Set 
$$H=\text{Stab}_{G_{\0}} x_{0}:=G_{\0,x_{0}}.$$ 
and 
$$\f_{\1}= \g_{\1}^{H}=\{z \in \g_{\1}:\ h.z=z \text{ for all $h \in H$} \}.$$ 
Note that the roots of $\f_{\1}$ are listed in Table \ref{T:detectingsub}. One can construct the detecting subalgebra ${\mathfrak f}$ by letting ${\mathfrak f}_{\0}=
[{\mathfrak f}_{\1},{\mathfrak f}_{\1}]$ with ${\mathfrak f}:={\mathfrak f}_{\0}\oplus 
{\mathfrak f}_{\1}$.

Now let $N=N_{G_{\0}}(H)$ and $N_{\0}$ be the connected component of the identity. 
Since $x_{0}$ is semisimple, $H$ is reductive as well as $N$. Set 
$$W_{\1}=W_{\mathfrak f}:=N_{G_{\0}}(H)/N_{\0}.$$
The finite group $W_{\1}$ is a pseudo-reflection group. 

The action of $G_{\0}$ on ${\mathfrak g}_{\1}$ is a {\em polar} representation 
(as in \cite{dadokkac}). In particular, 
$$\dim {\mathfrak e}_{x_{0}}=\text{Kr. dim }S^{\bullet}({\mathfrak g}_{\1}^{*})^{G_{\0}}$$ 
where 
$${\mathfrak e}_{x_{0}}:=\{x\in {\mathfrak g}_{\1}:\ {\mathfrak g}_{\0}x\subseteq {\mathfrak g}_{\0}x_{0}\}.$$ 
Set ${\mathfrak e}_{\1}={\mathfrak e}_{x_{0}}$, ${\mathfrak e}={\mathfrak e}_{\0}\oplus {\mathfrak e}_{\1}$ 
with ${\mathfrak e}_{\0}=[{\mathfrak e}_{\1},{\mathfrak e}_{\1}]$. 

One can obtain a finite reflection group $W_{\e}$ by setting 
$$W_{\e}=N_{G_{\0}}({\mathfrak e}_{\1})/\text{Stab}_{G_{\0}}({\mathfrak e}_{\1}).$$

\subsection{} In this section, we compare the support varieties for the classical Lie superalgebras ${\mathfrak g}$,
${\mathfrak f}$, and ${\mathfrak e}$ under the restriction maps. Assume that $\g$ is both stable and polar. Without
the assumption that $\g$ is polar, the statements concerning cohomology and
support varieties for ${\mathfrak g}$ and ${\mathfrak f}$ remain true. We recall the exposition given in \cite[Section 6.1]{BKN1}. 

First there are natural maps of rings given by restriction,
$$
\text{res}:\operatorname{H}^{\bullet}(\g,\g_{\0}; \C) \rightarrow \operatorname{H}^{\bullet}(\f,\f_{\0}; \C)
\rightarrow \operatorname{H}^{\bullet}(\e,\e_{\0},\C), 
$$
which induce isomorphisms
\begin{equation}\label{eq:quotient} 
\text{res}:\operatorname{H}^{\bullet}(\g,\g_{\0}; \C) \rightarrow 
\operatorname{H}^{\bullet}(\f,\f_{\0}; \C)^{N}
\rightarrow \operatorname{H}^{\bullet}(\e,\e_{\0},\C)^{W_{\e}}.
\end{equation}
The map on cohomology above induces a morphism of varieties:
$$
\text{res}^{*}:\mathcal{V}_{(\e,\e_{\0})}(\C)\longrightarrow
\mathcal{V}_{(\f,\f_{\0})}(\C)\rightarrow \mathcal{V}_{(\g,\g_{\0})}(\C)
$$
and isomorphisms (by passing to quotient spaces)
\begin{equation}\label{E:isovarieties}
\text{res}^{*}:\mathcal{V}_{(\e,\e_{\0})}(\C)/W_{\e}\rightarrow
\mathcal{V}_{(\f,\f_{\0})}(\C)/N\xrightarrow{} \mathcal{V}_{(\g,\g_{\0})}(\C).
\end{equation}

Let $M$ be a finite-dimensional $\g$-module. Then $\res^{*}$ induces maps
between support varieties:
$$\mathcal{V}_{(\e,\e_{\0})}(M) \rightarrow \mathcal{V}_{(\f,\f_{\0})}(M) \rightarrow
\mathcal{V}_{(\g,\g_{\0})}(M).
$$
Since $M$ is a ${\mathfrak g}_{\0}$-module, the first two varieties are stable under the action of
$W_{\e}$ and $N$ respectively. Consequently, we obtain the following induced maps of
varieties using (\ref{E:isovarieties}):
$$\mathcal{V}_{(\e,\e_{\0})}(M)/W_{\e} \hookrightarrow \mathcal{V}_{(\f,\f_{\0})}(M)/N
\hookrightarrow \mathcal{V}_{(\g,\g_{\0})}(M).
$$
These maps are embeddings because if $x\in R$ annihilates the identity in 
$\text{H}^{0}(\g,\g_{\0},M^{*}\otimes M)$ then
it must annihilate the identity elements in $\text{H}^{0}(\f,\f_{\0},M^{*}\otimes M)$ and
$\text{H}^{0}(\e,\e_{\0},M^{*}\otimes M)$, and the restriction maps induce isomorphisms on the cohomology given in 
(\ref{eq:quotient}).

\subsection{Support varieties for stable and polar detecting subalgebras} We record 
the result proved in \cite[Theorem 4.5.1]{LNZ} that shows that the support varieties for $\e$ and $\f$ coincide 
after taking the geometric quotient. 

\begin{theorem}\label{T:isosupportsfe}  Let $\g$ be a classical simple Lie superalgebra which is
stable and polar. If $M\in {\mathcal F}_{(\f,\f_{\0})}$ then
we have the following isomorphism of varieties:
$$\operatorname{res}^{*}:{\mathcal V}_{(\e,\e_{\0})}(M)/W_{\e}\rightarrow 
{\mathcal V}_{(\f,\f_{\0})}(M)/N.$$
\end{theorem}

\section{Construction of ${\mathfrak b}$}\label{S:parabolicconstruction}

\subsection{Generalities on parabolic subalgebras}\label{SS:parabolicsub} 
Let $\mathfrak{g}$ be a classical simple Lie superalgebra with a fixed Cartan subalgebra $\mathfrak{h}$ and root system $\Phi = \Phi(\mathfrak{g}, \mathfrak{h})$.  For the definitions of $\epsilon_i, \delta_j,\epsilon$ we follow the convention of \cite{Kac} with the exception of the superalgebras $D(2,1,\alpha)$, $F(4)$, and $G(3)$.  For the latter  we use the following notation:  
$(\epsilon, 0,0)$, $(0,\epsilon, 0)$, $(0,0,\epsilon)$ for $\epsilon_1,\epsilon_2,\epsilon_3$, respectively,  if $\mathfrak{g} = D(2,1,\alpha)$; $(0,\epsilon)$ for $\delta$  if $\mathfrak{g} = G(3)$ or $\mathfrak{g} = F(4)$.

In what follows we use the terminology and setting of \cite{GY}. A \emph{parabolic subalgebra} of $\mathfrak{g}$ is a subalgebra that contains a Borel subalgebra of $\mathfrak{g}$. We will consider  only parabolic subalgebras that contain $\mathfrak{h}$. Every such parabolic subalgebra corresponds to a parabolic set of roots, as explained below.

Assume first that $\Phi$ is symmetric, i.e., $\Phi= - \Phi$. This is true for all classical Lie superalgebras $\mathfrak g$ except for those of type $P$.  We call a proper subset $S$ of $\Phi$ 
a \emph{parabolic set in $\Phi$} if
\[
\Phi= S  \cup(-S), \; \; \; \mbox{and} \; \; 
\alpha, \beta \in S \; {\text { with }} \; \alpha + \beta \in \Phi
\; \; 
{\text { implies }} \; \; \alpha + \beta \in S.
\]
In the case when  $\Phi \neq - \Phi$, we call $S \subsetneq \Phi$ a \emph{parabolic subset} if $S = \widetilde{S} \cap \Phi$  for some parabolic subset $\widetilde{S}$ of $\Phi \cup (- \Phi )$.

To assign a parabolic set of roots to a parabolic subalgebra $\mathfrak{p}$ of $\mathfrak{g}$, we use the correspondence $\mathfrak{p} \mapsto \Phi_{\mathfrak{p}}$, where $\Phi_{\mathfrak{p}}$ are the roots of $\mathfrak{p}$ relative to $(\mathfrak{g}, \mathfrak{h})$. For the reverse direction we proceed as follows.

For a parabolic subset of roots 
$S$, we call $S^0 := S \cap (-S)$ 
the {\em{Levi component}} of $S$, $S^-:= S \backslash (-S)$ the 
{\em{nilpotent component}} of $S$, and $S = S^0 \sqcup S^-$ the \emph{Levi decomposition 
of $S$}. Then $\mathfrak{p}_S = \mathfrak{h} \oplus \left( \bigoplus_{\mu \in S } \mathfrak{g}^{\mu} \right)$ is a parabolic subalgebra of $\mathfrak{g}$ containing $\mathfrak{h}$, and $\mathfrak{l}_S = \mathfrak{h} \oplus \left( \bigoplus_{\mu \in S^0 } \mathfrak{g}^{\mu}\right)$ and  $\mathfrak{n}_S^- =   \bigoplus_{\mu \in S^- } \mathfrak{g}^{\mu} $ are called the \emph{Levi subalgebra}, and the \emph{nilradical} of $\mathfrak{p}_S$, respectively.

Let $V_{\Phi}$ be a real vector space such that $\Phi \subset V_{\Phi} \setminus \{ 0\}$. An element  $\mathcal H$ in $V_{\Phi}^*$ defines a parabolic subset of roots $S=S(\mathcal H)$ as follows. We define $S^0$ (respectively, $S^-$) to be the subset of $\Phi$ consisting of all roots $\alpha$ such that $\alpha (h) = 0$ (respectively, $\alpha(h) <0$) for all $h \in \mathcal H$. Note that we identify the elements of $(V_{\Phi}^*)^*$ and $V_{\Phi}$. A parabolic subset of roots $S$ that is of the form $S(\mathcal H)$ for some $\mathcal H$
 is called \emph{principal parabolic subset}. Note that $\ker \mathcal H$ is a  hyperplane in $V_{\Phi}$, and the roots in $S^0$ (respectively, $S^-$) can be treated as those that are on (respectively, ``below'') the hyperplane $\ker \mathcal H$.

\subsection{} \label{SS:BBWparabolics} A parabolic subalgebra, ${\mathfrak b}$, that arises from taking a principal parabolic subset $S = S(\mathcal H) = S^0 \sqcup S^-$, where $\mathcal H$ is listed in Table \ref{T:BBW-hyperplanes}, will be called 
a {\em BBW parabolic subalgebra}. Later, in Theorem~\ref{T:Poincareseriesequal}, it will be shown that these subalgebras have very special cohomological properties 
involving equality of various Poincar\'e series. There exists a natural triangular decomposition of ${\mathfrak g}={\mathfrak u}^{+}\oplus {\mathfrak f}\oplus {\mathfrak u}$ where the roots in ${\mathfrak u}_{\1}^{+}$ (resp. ${\mathfrak u}$) coincide with $-(S^{-})$ (resp. $S^{-}$). The BBW parabolic subalgebra identifies with $\mathfrak{b} = \mathfrak{f} \oplus \mathfrak{u}$. Even though ${\mathfrak b}$ is a parabolic subalgebra and technically is not a Borel subalgebra, we will view ${\mathfrak b}$ as being analogous to a Borel subalgebra for a complex simple Lie algebra, and the 
detecting subalgebra ${\mathfrak f}$ like a maximal torus. In the cases when ${\mathfrak g}=\mathfrak{gl}(n|n)$ or ${\mathfrak q}(n)$, ${\mathfrak b}$ can be realized 
as matrices of the form: 
 \[ 
{\mathfrak b}=\left\{  \left[
\begin{array}{c|c}
A  & B \\ \hline
 C & D
\end{array}\right] \in {\mathfrak g}:  A,\ B,\ C,\ D\in L_{n}({\mathbb C}) \right\}
 \]
where $L_{n}({\mathbb C})$ are the set of $n\times n$ lower triangular matrices. 
We add that there exists a supergroup scheme $B$ with $\text{Lie }B={\mathfrak b}$ that corresponds to the the  (super) Hopf algebra $U({\mathfrak b})\cong \text{Dist}(B)$. 

For this paper, let $\Phi_{\1}^{-}$ (resp. $\Phi_{\1}^{+}$) corresponds with the roots in ${\mathfrak u}_{\1}$ (resp. ${\mathfrak u}_{\1}^{+}$). One has $\Phi_{\1}=\Phi_{1}^{+}\cup 
\Phi_{\1}^{-}$ and in the case when ${\mathfrak g}\neq {\mathfrak p}(n)$, $\Phi_{\1}^{-}=-(\Phi_{\1}^{+})$. In particular, we will take the liberty of calling  $\Phi_{\1}^{-}$ the negative roots of $\mathfrak{f}$. The authors realize that this convention is not the standard practice 
in the literature. However, in Section 4.3, we will demonstrate that the dot action of $W_{\1}$ on $\Phi^{+}_{\1}$ is compatible with the dot action of the Weyl group 
of $G_{\0}$ on $\Phi^{+}_{\0}$. This key observation entailing the compatibility of these even and odd roots allows us to successfully complete the computations in the paper. 

In Table \ref{T:BBW-roots} we describe the odd negative roots of the principal parabolic subsets $S = S^0 \sqcup S^-$ corresponding to the parabolic subalgebras $\mathfrak{b} = \mathfrak{f} \oplus \mathfrak{u}$. The elements $\mathcal H$ defining $P$ are listed in Table \ref{T:BBW-hyperplanes}.  For  ${\mathfrak g} = \mathfrak{gl} (m|n), \mathfrak{sl} (m|n), \mathfrak{osp} (2m|2n), \mathfrak{osp} (2m+1|2n)$, we let $V_{\Phi} = \mbox{Span}\,\{\epsilon_i, \delta_j\; | \; 1\leq i \leq m, 1\leq j \leq n\}$ and fix $E_i$ and $D_j$ to be the basis vectors of $V_{\Phi}^*$ that are dual to  $\epsilon_i$ and $\delta_j$, respectively. Also, for these superalgebras, we let $E_i =0$ and $D_j=0$ whenever $i>m$ and $j>n$. For all exceptional Lie superalgebras we choose $ V_{\Phi} = {\mathbb R} \otimes_{\mathbb Z} ({\mathbb Z} \Phi)$. For ${\mathfrak g} = D(2,1,\alpha)$ we let $E_1, E_2,E_3$ to be the dual to $(\epsilon, 0,0)$, $(0,\epsilon, 0)$, $(0,0,\epsilon)$, respectively. Lastly, if  ${\mathfrak g} = G(3), F(4)$ we use $L_i$ for the vectors in $V_{\Phi}^*$  dual to the fundamental weights $\omega_i$ of $G_2$ ($i=1,2$), $\mathfrak{so}(7)$ ($i=1,2,3$), respectively, and $E$ for the dual of $(0,\epsilon)$.

Note that $x_i$ are arbitrary real numbers subject to the conditions listed in the table. In all cases $\Phi_{\1}^-$ corresponds to the odd part of $S^-$.

\subsection{}\label{SS:parabolicdef} For each classical simple Lie superalgebra ${\mathfrak g}$ we can define a parabolic 
subalgebra ${\mathfrak b}$ via the decomposition of odd roots given in Table~\ref{T:BBW-roots} and in Section~\ref{SS:BBWp(n)} 
for ${\mathfrak g}={\mathfrak p}(n)$ that satisfies the following properties:  

\begin{itemize} 
\item[(a)] ${\mathfrak b}={\mathfrak b}_{\0}\oplus {\mathfrak b}_{\1}$ where 
${\mathfrak b}_{\0}$ is a (negative) Borel subalgebra of ${\mathfrak g}_{\0}$ with maximal torus ${\mathfrak t}_{\0}$. 
\item[(b)] ${\mathfrak t}={\mathfrak t}_{\0}\oplus {\mathfrak t}_{\1}$ where ${\mathfrak t}_{\1}={\mathfrak f}_{\1}$ where ${\mathfrak f}$ is the (stable) detecting subalgebra. 
\item[(c)] ${\mathfrak f}$ is a subalgebra of ${\mathfrak t}$. 
\item[(d)] ${\mathfrak f}_{\1}$ is $T_{\0}$-stable where $\text{Lie }T_{\0}={\mathfrak t}_{\0}$. 
\item[(e)] ${\mathfrak b}={\mathfrak t}\oplus {\mathfrak u}$ where ${\mathfrak u}$ is a nilpotent Lie superalgebra.
\item[(f)] ${\mathfrak u}={\mathfrak u}_{\0}\oplus {\mathfrak u}_{\1}$ where ${\mathfrak u}_{\0}$ is the unipotent radical of ${\mathfrak b}_{\0}$.
\end{itemize} 
In this setting one has a weight space decomposition ${\mathfrak u}_{\1}=\oplus_{\lambda\in {\mathfrak t}_{\0}^{*}} ({\mathfrak u}_{\1})_{\lambda}$ where 
$({\mathfrak u}_{\1})_{\lambda}$ is a ${\mathfrak t}_{\0}$-module with composition factors of the form $\lambda$.
 
\subsection{Comparison of cohomology} We first compare the relative cohomology for 
$({\mathfrak b},{\mathfrak b}_{\0})$ and $({\mathfrak f},{\mathfrak f}_{\0})$. 

\begin{theorem} \label{t:b-tcoho} Let ${\mathfrak b}={\mathfrak t}\oplus {\mathfrak u}$ be the parabolic subalgebra as defined in Section~\ref{SS:parabolicdef}. 
Then 
\begin{itemize} 
\item[(a)] $\operatorname{H}^{\bullet}({\mathfrak f},{\mathfrak f}_{\0},{\mathbb C})\cong S^{\bullet}({\mathfrak f}^{*}_{\1})$. 
\item[(b)] The restriction map 
$$\operatorname{H}^{\bullet}({\mathfrak b},{\mathfrak b}_{\0},{\mathbb C})\rightarrow 
\operatorname{H}^{\bullet}({\mathfrak f},{\mathfrak f}_{\0},{\mathbb C})^{T_{\0}}$$ 
is an isomorphism. Moreover, $\operatorname{H}^{\bullet}({\mathfrak f},{\mathfrak f}_{\0},{\mathbb C})^{T_{\0}}\cong \operatorname{H}^{\bullet}({\mathfrak t},{\mathfrak t}_{\0},{\mathbb C})$.
\end{itemize} 
\end{theorem} 

\begin{proof} (a) Since $[{\mathfrak f}_{\0},{\mathfrak f}_{\1}]=0$ it follows that 
\begin{equation} 
\operatorname{H}^{\bullet}({\mathfrak f},{\mathfrak f}_{\0},{\mathbb C})\cong S^{\bullet}({\mathfrak f}^{*}_{\1})^{F_{\0}}\cong S^{\bullet}({\mathfrak f}^{*}_{\1}).
\end{equation} 

(b) Next observe that 
$$S^{n}({\mathfrak b}^{*}_{\1})^{F_{\0}} \cong S^{n}({\mathfrak f}^{*}_{\1}\oplus {\mathfrak u}^{*}_{\1})^{F_{\0}} \cong 
 \bigoplus_{i+j=n} S^{i}({\mathfrak f}^{*}_{\1})\otimes S^{j}({\mathfrak u}^{*}_{\1})^{F_{\0}} \cong  S^{n}({\mathfrak f}^{*}_{\1}).$$ 
The last isomorphism holds since (i) $S^{\bullet}({\mathfrak u}^{*}_{\1})^{F_{\0}}\subseteq S^{\bullet}({\mathfrak u}^{*}_{\1})^{T_{\0}}$ and 
(ii) the duals of roots in ${\mathfrak u}_{\1}^{*}$ under the $T_{\0}$-grading are positive (see Section~\ref{SS:parabolicsub}). It follows that 
$$S^{n}({\mathfrak b}^{*}_{\1})^{T_{\0}} \cong S^{n}({\mathfrak f}^{*}_{\1})^{T_{\0}}$$ 
and $\dim S^{n}({\mathfrak b}^{*}_{\1})^{B_{\0}} \leq \dim S^{n}({\mathfrak f}^{*}_{\1})^{T_{\0}}$ for $n\geq 0$. 

Since $\operatorname{H}^{\bullet}({\mathfrak b},{\mathfrak b}_{\0},{\mathbb C})\cong S^{\bullet}({\mathfrak b}^{*}_{\1})^{B_{\0}}$, 
the restriction map $\operatorname{H}^{\bullet}({\mathfrak b},{\mathfrak b}_{\0},{\mathbb C})\rightarrow 
\operatorname{H}^{\bullet}({\mathfrak t},{\mathfrak t}_{\0},{\mathbb C})$ is given by the restriction map on functions: 
\begin{equation}
S^{\bullet}({\mathfrak b}^{*}_{\1})^{B_{\0}}\rightarrow S^{\bullet}({\mathfrak t}^{*}_{\1})^{T_{\0}}. 
\end{equation}
Finally, observe that as $B_{\0}$-module, one has a short exact sequence 
$$0\rightarrow {\mathfrak u}_{\1} \rightarrow {\mathfrak b}_{\1} \rightarrow {\mathfrak t}_{\1} \rightarrow 0.$$ 
Therefore, 
$$0\rightarrow {\mathfrak t}_{\1}^{*} \rightarrow {\mathfrak b}_{\1}^{*} \rightarrow {\mathfrak u}_{\1}^{*} \rightarrow 0$$ 
with $B_{\0}$-acting trivially on ${\mathfrak t}_{\1}^{*}$. This shows there exists a subring $S\subseteq  S^{\bullet}({\mathfrak b}^{*}_{\1})^{B_{\0}}$ 
such that the restriction map induces an isomorphism of $S\cong S^{\bullet}({\mathfrak t}^{*}_{\1})^{T_{\0}}=S^{\bullet}({\mathfrak f}^{*}_{\1})^{T_{\0}}$. The statement of 
(b) now follows because $\dim S^{n}({\mathfrak b}^{*}_{\1})^{B_{\0}} \leq \dim S^{n}({\mathfrak t}^{*}_{\1})^{T_{\0}}$ for $n\geq 0$. 
\end{proof} 

\subsection{} We can now demonstrate how the relative cohomology for ${\mathfrak b}$ is related to the relative cohomology for ${\mathfrak g}$ and the
dual of the group algebra of $W_{\1}$. One can view this result as a functorial interpretation of the harmonic decomposition for $S^{\bullet}({\mathfrak f}_{\1}^{*})$. 

\begin{theorem} \label{T:paraproperties} Let ${\mathfrak g}$ be a classical simple Lie superalgebra. There exists a detecting subalgebra ${\mathfrak f}={\mathfrak f}_{\0}\oplus {\mathfrak f}_{\1}$ obtained 
by using the stable action of $G_{\0}$ on ${\mathfrak g}_{\1}$ and a proper parabolic subalgebra ${\mathfrak b}$ with the following properties 
\begin{itemize} 
\item[(a)] ${\mathfrak b}={\mathfrak b}_{\0}\oplus {\mathfrak b}_{\1}$ where ${\mathfrak b}_{\1}\cong {\mathfrak f}_{\1} 
\oplus {\mathfrak u}_{\1}$ and ${\mathfrak b}_{\0}$ is a Borel subalgebra for ${\mathfrak g}_{\0}$. 
\item[(b)] There exists a finite reflection group $W_{\1}$ isomorphic to $N/N_{\0}$ and a grading on the group algebra, 
${\mathbb C}[W_{\1}]$, such that as graded vector spaces, 
$$\operatorname{H}^{\bullet}({\mathfrak b},{\mathfrak b}_{\0},{\mathbb C})\cong 
\operatorname{H}^{\bullet}({\mathfrak g},{\mathfrak g}_{\0},{\mathbb C})\otimes {\mathbb C}[W_{\1}]_{\bullet}.$$ 
\end{itemize} 
\end{theorem} 

\begin{proof} Let ${\mathfrak b}$ be as in Section~\ref{SS:parabolicdef}. One has the harmonic decomposition (cf. \cite[Theorem 3.5]{BKN1}): 
\begin{equation} 
S^{\bullet}({\mathfrak f}_{\1}^{*})\cong S^{\bullet}({\mathfrak f}_{\1}^{*})^{N}\otimes [\text{ind}_{H}^{N}{\mathbb C}]_{\bullet}.
\end{equation} 
as graded $S^{\bullet}({\mathfrak f}_{\1}^{*})^{N}$-modules. Applying $T_{\0}$ fixed points and using the fact that $T_{\0}\leq N$ one can has 
\begin{equation} \label{e:invarianteq}
S^{\bullet}({\mathfrak f}_{\1}^{*})^{T_{\0}}\cong S^{\bullet}({\mathfrak f}_{\1}^{*})^{N}\otimes [\text{ind}_{H}^{N}{\mathbb C}]^{T_{\0}}_{\bullet}.
\end{equation} 
From the definition of the induced module, one has 
\begin{eqnarray*} 
\text{ind}_{H}^{N} {\mathbb C}&\cong& [\C[N]\otimes {\mathbb C}]^{H}\\
&\cong& \text{Hom}_{H}({\mathbb C},{\mathbb C}[N]) 
\end{eqnarray*} 
Now by applying $T_{\0}$ fixed points and using the fact that $N_{0}$ is generated by $T_{\0}$ and $H$: 
$$[\text{ind}_{H}^{N} {\mathbb C}]^{T_{\0}}\cong  [\text{Hom}_{H}({\mathbb C},{\mathbb C}[N])]^{T_{\0}}\cong  \text{Hom}_{N_{0}}({\mathbb C},{\mathbb C}[N]) \cong {\mathbb C}[W_{\1}].$$ 
Here ${\mathbb C}[W_{\1}]$ is the coordinate algebra of $W_{\1}$ which is dual to the group algebra of $W_{\1}$. 

Next one can use the isomorphisms: $\operatorname{H}^{\bullet}({\mathfrak b},{\mathfrak b}_{\0},{\mathbb C})\cong S^{\bullet}({\mathfrak f}_{\1}^{*})^{T_{\0}}$ 
by Theorem~\ref{t:b-tcoho}(b), and $\operatorname{H}^{\bullet}({\mathfrak g},{\mathfrak g}_{\0},{\mathbb C})\cong S^{\bullet}({\mathfrak f}_{\1}^{*})^{N}$ 
\cite[Theorem 4.1]{BKN1}. One can now reinterpret (\ref{e:invarianteq}) as 
\begin{equation} \label{e:cohocompare}
\operatorname{H}^{\bullet}({\mathfrak b},{\mathfrak b}_{\0},{\mathbb C}) \cong \operatorname{H}^{\bullet}({\mathfrak g},{\mathfrak g}_{\0},{\mathbb C}) \otimes {\mathbb C}[W_{\1}]_{\bullet}.
\end{equation} 
\end{proof} 

The reader should be made aware that the grading on ${\mathbb C}[W_{\1}]_{\bullet}$ is not always given by the Poincar\'e series for the finite reflection group $W_{\1}$. We will explore this important 
issue in the upcoming sections. 

\subsection{} Let $W$ be a finite reflection group and consider the Poincar\'e polynomial (cf. \cite[Section 1.11]{Hum})
\begin{equation} 
p_{W}(t)=\sum_{w\in W}t^{l(w)}
\end{equation}
Note that the coefficient of $t^{j}$ is precisely $|\{w\in W:\ l(w)=j \}|$. In general one has the identity
$$p_{W}(t) = \prod_{i=1}^{n}(1+t+\dots+t^{e_i}),$$
where $e_i$  are the exponents of $W$.
Set 
\begin{equation} 
z_{{\mathfrak b},{\mathfrak g}}(t)=p_{\mathfrak b}(t)/p_{\mathfrak g}(t)
\end{equation} 

We now provide some examples that show how to compute $z_{{\mathfrak b},{\mathfrak g}}(t)$. 

\begin{example} [${\mathfrak g}={\mathfrak q}(n)$ and $\mathfrak{gl}(m|n)$] Assume that $m\geq n$. 
One has $\operatorname{H}^{\bullet}({\mathfrak b},{\mathfrak b}_{\0},{\mathbb C})\cong S^{\bullet}({\mathfrak f}_{\1}^{*})^{T_{\0}}$. 
This implies that 
$$
\operatorname{H}^{\bullet}({\mathfrak b},{\mathfrak b}_{\0},{\mathbb C})\cong {\mathbb C}[z_{1},z_{2},\dots,z_{n}]
$$ 
where the degree of $z_{j}$ ($j=1,2,\dots,n$) is $1$ for ${\mathfrak q}(n)$ and $2$ for $\mathfrak{gl}(m|n)$. 
Furthermore, by \cite[Table 1]{BKN1}, 
$$
\operatorname{H}^{\bullet}({\mathfrak g},{\mathfrak g}_{\0},{\mathbb C})\cong {\mathbb C}[z_{1},z_{2},\dots,z_{n}]^{\Sigma_{n}}.
$$ 
Hence, $\operatorname{H}^{\bullet}({\mathfrak g},{\mathfrak g}_{\0},{\mathbb C})$ is a polynomial algebra generated in 
degrees $1,2,\dots,n$. Therefore, 
$z_{{\mathfrak b},{\mathfrak g}}(t)=p_{\Sigma_{n}}(t^{r})$ 
where $r=1$ for ${\mathfrak q}(n)$ and $r=2$ for $\mathfrak{gl}(m|n)$. 
\end{example}

\begin{example}[${\mathfrak g}=D(2,1,\alpha),\ G(3),\ F(4)$] A direct computation shows that 
$\operatorname{H}^{\bullet}({\mathfrak b},{\mathfrak b}_{\0},{\mathbb C})\cong {\mathbb C}[z]$ 
where $z$ is of degree $2$. From \cite[Table 1]{BKN1}, $\operatorname{H}^{\bullet}({\mathfrak g},{\mathfrak g}_{\0},{\mathbb C})$ 
is a polynomial algebra generated in degree $4$. Therefore,
$$z_{{\mathfrak b},{\mathfrak g}}(t)=\frac{1-t^{4}}{1-t^{2}}=1+t^{2}=p_{\Sigma_{2}}(t^{2}).$$
\end{example} 

One can compute $z_{{\mathfrak b},{\mathfrak g}}(t)$ for the other classical simple Lie superalgebras by using the 
ideas presented in the preceding examples. Table \ref{T:Poincareseries} provides the relationship between $z_{{\mathfrak b},{\mathfrak g}}(t)$ and 
the Poincar\'e polynomial for $W_{\1}$ for other classical simple Lie superalgebras. Note that the $x$'s, $y$'s, and $z$'s have degree one. 
We can summarize these results in the following theorem. 

\begin{theorem}\label{T:zcompute} Let ${\mathfrak g}$ be a classical simple Lie superalgebra. Assume that ${\mathfrak g}$ 
is not isomorphic to $P(n)$. 
There exists a detecting subalgebra ${\mathfrak f}={\mathfrak f}_{\0}\oplus {\mathfrak f}_{\1}$ obtained 
by using the stable action of $G_{\0}$ on ${\mathfrak g}_{\1}$ and a parabolic subalgebra ${\mathfrak b}$ such that 
$z_{{\mathfrak b},{\mathfrak g}}(t)=p_{W_{\1}}(s)$, where $s=t$ for Lie algebras ${\mathfrak g}$ of type $Q$, and $s=t^{2}$ otherwise. 
\end{theorem} 

\section{Connections with the geometry of $G/B$} 

\subsection{Supergroups and the Induction Functor} Let $G$ be an affine supergroup scheme over ${\mathbb C}$ and $\text{Mod}(G)$ 
be the category of rational modules for $G$. For a general overview and details about supergroup schemes, the reader is referred to work 
of Brundan and Kleshchev \cite[Sections 2,4,5]{BruKl} \cite[Section 2]{Bru}. 

In the case when ${\mathfrak g}$ is a classical Lie superalgebra and ${\mathfrak g}=\text{Lie }G$, the category $\text{Mod}(G)$ is equivalent to locally finite integral 
modules for $\text{Dist}(G)=U({\mathfrak g})$ (cf. \cite[Corollary 5.7]{BruKl}). In particular, if 
${\mathfrak g}$ is a classical Lie superalgebra, then $\text{Mod}(G)$ is equivalent to ${\mathcal C}_{({\mathfrak g},{\mathfrak g}_{\0})}$ 
(i.e., the category of ${\mathfrak g}$-supermodules that are completely reducible over ${\mathfrak g}_{\0}$). 

Let $H$ be a closed subgroup scheme of $G$ and $R^{j}\text{ind}_{H}^{G}(-)$ be the higher right derived functors of the induction functor $\text{ind}_{H}^{G}(-)$. 
In the case when ${\mathfrak g}=\text{Lie }G$ is a classical simple Lie superalgebra and $H=P$ where $P$ is a parabolic subgroup, the following two propositions provide information about $R^{\bullet}\operatorname{ind}_{P}^{G} M$ when restricted to $G_{\0}$. 

\begin{proposition} \label{P:Gzeroiso} Let ${\mathfrak g}=\operatorname{Lie }G$ be a classical simple Lie superalgebra and $P$ be a parabolic 
subgroup with $M$ a $P$-module. 
\begin{itemize} 
\item[(a)] Assume that 
$R^{n}\operatorname{ind}_{P_{\0}}^{G_{\0}} [M \otimes \Lambda^{i}(({\mathfrak g}_{\1}/{\mathfrak p}_{\1})^{*})]=0$ for $n$-odd, and $n$-even when $n\neq i$. 
Then 
$$(R^{n}\operatorname{ind}_{P}^{G} M)|_{G_{\0}}\cong R^{n}\operatorname{ind}_{P_{\0}}^{G_{\0}} [M \otimes \Lambda^{\bullet}(({\mathfrak g}_{\1}/{\mathfrak p}_{\1})^{*})]$$
for $n\geq 0$. 
\item[(b)] Assume that $M\cong {\mathbb C}$ and 
$R^{n}\operatorname{ind}_{P_{\0}}^{G_{\0}} [\Lambda^{i}(({\mathfrak g}_{\1}/{\mathfrak p}_{\1})^{*})]=0$ for $i\neq n$. Then 
$$(R^{n}\operatorname{ind}_{P}^{G} {\mathbb C})|_{G_{\0}}\cong R^{n}\operatorname{ind}_{P_{\0}}^{G_{\0}} [\Lambda^{\bullet}(({\mathfrak g}_{\1}/{\mathfrak p}_{\1})^{*})]$$
for $n\geq 0$. 
\end{itemize} 
\end{proposition} 

\begin{proof} We will employ results provided in the exposition given in \cite[Section 2]{Bru}. Let $X=G/P$ and $M$ be a $P$-module, with ${\mathcal L}(M)$ being the 
associated quasi-coherent ${\mathcal O}_{X}G$-(super)module. First note that $H^{n}(G/P,{\mathcal L}(M))\cong R^{n}\text{ind}_{P}^{G}M$ for all $n\geq 0$. Now according to 
\cite[(6)]{Bru}, one has 
\begin{equation} 
R^{n}\text{ind}_{P}^{G}M|_{G_{\0}} \cong H^{n}(G/P,{\mathcal L}(M))|_{G_{\0}}\cong H^{n}(G/P,\text{res}^{G}_{G_{\0}}({\mathcal L}(M))) 
\end{equation} 
for all $n\geq 0$. 
Next observe that by \cite[(2), Theorem 2.7]{Bru} there exists a canonical filtration of ${\mathcal L}(M)$: 
$${\mathcal J}^{0}={\mathcal L}(M)\supseteq {\mathcal J}^{1} \supseteq {\mathcal J}^{2} \supseteq \dots {\mathcal J}^{t-1}\supseteq {\mathcal J}^{t}=\{0\}$$ 
with 
\begin{equation} 
{\mathcal J}^{i}/{\mathcal J}^{i+1}\cong {\mathcal L}_{ev}(M \otimes \Lambda^{i}(({\mathfrak g}_{\1}/{\mathfrak p}_{\1})^{*}))
\end{equation} 
Finally,  by \cite[equation after (2)]{Bru} one has the following isomorphisms: 
\begin{equation} 
H^{n}(G/P,{\mathcal J}^{i}/{\mathcal J}^{i+1})\cong H^{j}(G_{\0}/P_{\0}, {\mathcal L}_{ev}(M \otimes \Lambda^{i}(({\mathfrak g}_{\1}/{\mathfrak p}_{\1})^{*}))\cong 
R^{n}\text{ind}_{P_{\0}}^{G_{\0}} [M \otimes \Lambda^{i}(({\mathfrak g}_{\1}/{\mathfrak p}_{\1})^{*})].
\end{equation} 

(a) The filtration described above yields the short exact sequence:
$$0\rightarrow {\mathcal J}^{i+1} \rightarrow {\mathcal J}^{i} \rightarrow {\mathcal J}^{i}/{\mathcal J}^{i+1} \rightarrow 0.$$
Next apply the long exact sequence in cohomology, and use the fact that $H^{n}(G/P,{\mathcal J}^{i}/{\mathcal J}^{i+1})=0$ for $n$-odd and $i\geq 0$ to 
obtains a five term exact sequences for $(n-1)$-even: 
\begin{eqnarray*} \label{eq:fiveterm}
0 &\rightarrow& H^{n-1}(G/P, {\mathcal J}^{i+1}) \rightarrow H^{n-1}(G/P, {\mathcal J}^{i}) \rightarrow H^{n-1}(G/P, {\mathcal J}^{i}/{\mathcal J}^{i+1}) \\
   &\rightarrow& H^{n}(G/P, {\mathcal J}^{i+1}) \rightarrow H^{n}(G/P, {\mathcal J}^{i}) \rightarrow 0.
\end{eqnarray*} 

Using these five term sequences, we first show that $H^{n}(G/P,{\mathcal J}^{i})=0$ for $n$-odd and all $i\geq 0$. First consider the case when $i+1=t$. Then 
$H^{n}(G/P, {\mathcal J}^{i+1})=0$ for all $n\geq 0$ and from the sequences above $H^{n}(G/P, {\mathcal J}^{t-1})=0$ for $n$-odd. Now apply the process 
again for $i+1=t-1$, and the prior result to show that $H^{n}(G/P, {\mathcal J}^{t-2})=0$ for $n$-odd. Continuing this process proves that 
$H^{n}(G/P,{\mathcal J}^{i})=0$ for $n$-odd and $i\geq 0$ and the statement of part (a) of the theorem in the case when $n$ is odd. 

Next we finish off the statement of part (a) when $n$ is even. From the results in the prior paragraph, the five term exact sequences become short exact sequences of the form: 
\begin{equation} \label{eq:ses}
0\rightarrow H^{n}(G/P, {\mathcal J}^{i+1}) \rightarrow H^{n}(G/P, {\mathcal J}^{i}) \rightarrow H^{n}(G/P, {\mathcal J}^{i}/{\mathcal J}^{i+1}) \rightarrow 0.
\end{equation} 
for $n\geq 0$ (here $n$ can be either even or odd).  
Using these short exact sequences, we can conclude that for $i\leq n$ 
\begin{equation} 
H^{n}(G/P,{\mathcal J}^{0})=H^{n}(G/P,{\mathcal J}^{i}) 
\end{equation} 
and for $n<i$, 
\begin{equation} 
H^{n}(G/P,{\mathcal J}^{i})=0.
\end{equation} 
Combining these equations and using the assumptions in the theorem, it follows that as $G_{\0}$-module, 
$$H^{n}(G/P,{\mathcal J}^{0})=H^{n}(G/P,{\mathcal J}^{n}/{\mathcal J}^{n+1})=H^{n}(G/P,\oplus_{i\geq 0}{\mathcal J}^{i}/{\mathcal J}^{i+1}).$$ 
The result now follows by applying the identifications provided in the first paragraph. 

(b) In order to prove the statement we need to consider the $\pi$-graded 
category of $G$-(super)modules where $\pi={\mathbb Z}_{2}$. In this category, the simple modules consist of the simple $G$-modules with their images under the 
parity change functor $\Pi$. In particular, one has the trivial module ${\mathbb C}$ with trivial $\pi$-action and the module 
$\Pi{\mathbb C}$ with trivial $G$-action and the non-trivial element in $\pi$ acting as $(-1)$. Denote the graded category by 
$\pi$-$({\mathfrak g},{\mathfrak g}_{\0})$. 

The short exact sequence 
$$0\rightarrow {\mathcal J}^{i+1} \rightarrow {\mathcal J}^{i} \rightarrow {\mathcal J}^{i}/{\mathcal J}^{i+1} \rightarrow 0.$$
along with the the long exact sequence in cohomology and the fact that $H^{k}(G/P,{\mathcal J}^{i}/{\mathcal J}^{i+1})=0$ for $i\neq k$ yields the 
the five term exact sequence: 
\begin{eqnarray*} \label{eq:fiveterm}
0 &\rightarrow& H^{i}(G/P, {\mathcal J}^{i+1}) \rightarrow H^{i}(G/P, {\mathcal J}^{i}) \rightarrow H^{i}(G/P, {\mathcal J}^{i}/{\mathcal J}^{i+1}) \\
   &\rightarrow& H^{i+1}(G/P, {\mathcal J}^{i+1}) \rightarrow H^{i+1}(G/P, {\mathcal J}^{i}) \rightarrow 0.
\end{eqnarray*} 
Moreover, one obtains the following isomorphisms: 
\begin{equation} \label{eq:isos} 
H^{k}(G/P,{\mathcal J}^{i+1})\cong H^{k}(G/P,{\mathcal J}^{i})\ \ \ \text{for $k<i$ and $k>i+1$}. 
\end{equation} 
Now fix $n\geq 0$. From the isomorphisms in (\ref{eq:isos}), 
\begin{equation} \label{eq:iso1}
H^{n}(G/P,{\mathcal J}^{0})\cong H^{n}(G/P,{\mathcal J}^{1})\cong \dots \cong H^{n}(G/P,{\mathcal J}^{n-1})\cong H^{n}(G/P,{\mathcal J}^{n}),
\end{equation} 
\begin{equation} \label{eq:iso2} 
0=H^{n}(G/P,{\mathcal J}^{t})\cong H^{n}(G/P,{\mathcal J}^{t-1})\cong \dots \cong H^{n}(G/P,{\mathcal J}^{n+2})\cong H^{n}(G/P,{\mathcal J}^{n+1}),
\end{equation} 
One can use the five term sequence above along with (\ref{eq:iso1}) and (\ref{eq:iso2}) to obtain a four term exact sequnce 
\begin{equation*}
0 \rightarrow  H^{n}(G/P, {\mathcal J}^{0}) \rightarrow H^{n}(G/P, {\mathcal J}^{n}/{\mathcal J}^{n+1}) 
\rightarrow  H^{n+1}(G/P, {\mathcal J}^{n+1}) \rightarrow  H^{n+1}(G/P, {\mathcal J}^{n}) \rightarrow 0
\end{equation*} 
From the exact sequence and the isomorphism $R^{n}\text{ind}_{P}^{G} {\mathbb C}|_{G_{\0}}\cong H^{n}(G/P, {\mathcal J}^{0})$, one has an 
injection: 
$$f_{n}:R^{n}\text{ind}_{P}^{G} {\mathbb C}|_{G_{\0}}\hookrightarrow H^{n}(G/P, {\mathcal J}^{n}/{\mathcal J}^{n+1}).$$ 
The statement of part (b) will now follow if we show that $f_{n}$ is an isomorphism for all $n\geq 0$. Using the hypothesis and \cite[Corollary 2.8]{Bru}, one has 
$$\sum_{n\geq 0} (-1)^{n} R^{n}\text{ind}_{P}^{G} {\mathbb C}=\sum_{n\geq 0} (-1)^{n} H^{n}(G/P,{\mathcal J}^{n}/{\mathcal J}^{n+1}).$$ 
Here the sum is taken in the Grothendieck group of $\pi$-graded $G_{\0}$-modules. Using the fact that $f_{n}$ is an injection, the simple modules appearing in $H^{n}(G/P,{\mathcal J}^{n}/{\mathcal J}^{n+1})$ and $R^{n}\text{ind}_{P}^{G} {\mathbb C}$ have the same parity (depending on the parity of $n$, see \cite[Lemma 4.4]{Bru}). 

It follows that 
$$\sum_{n\geq 0} R^{2n}\text{ind}_{P}^{G} {\mathbb C}=\sum_{n\geq 0} H^{2n}(G/P,{\mathcal J}^{2n}/{\mathcal J}^{2n+1}),$$
$$\sum_{n\geq 0} R^{2n+1}\text{ind}_{P}^{G} {\mathbb C}=\sum_{n\geq 0} H^{2n+1}(G/P,{\mathcal J}^{2n+1}/{\mathcal J}^{2n+2}).$$
Hence, 
$$\sum_{n\geq 0} \dim R^{n}\text{ind}_{P}^{G} {\mathbb C}=\sum_{n\geq 0} \dim H^{n}(G/P,{\mathcal J}^{n}/{\mathcal J}^{n+1}).$$
This proves that $f_{n}$ is an isomorphism for all $n$. 
 \end{proof} 

\begin{proposition} \label{P:Gzeroiso2} Let ${\mathfrak g}=\operatorname{Lie }G$ be a classical simple Lie superalgebra and $P$ be a parabolic 
subgroup with $M$ a $P$-module.  Assume that 
$R^{j}\operatorname{ind}_{P_{\0}}^{G_{\0}} [M \otimes \Lambda^{\bullet}(({\mathfrak g}_{\1}/{\mathfrak p}_{\1})^{*})]=0$ for $j> 0$. 
Then 
$$(R^{j}\operatorname{ind}_{P}^{G} M)|_{G_{\0}}\cong R^{j}\operatorname{ind}_{P_{\0}}^{G_{\0}} [M \otimes \Lambda^{\bullet}(({\mathfrak g}_{\1}/{\mathfrak p}_{\1})^{*})]$$
for $j\geq 0$. 
\end{proposition} 

\begin{proof} We use the setting as described in Proposition~\ref{P:Gzeroiso}. For $j>0$, $H^{j}(G/P,{\mathcal J}^{i}/J^{i+1})=0$ for all $i$. It follows that $H^{j}(G/P,{\mathcal J}^{i})=0$ for 
all $i$, thus $R^{j}\operatorname{ind}_{P}^{G} M=0$ for $j>0$. 

Now consider the case when $j=0$. For each $i$, one has the short exact sequence 
$$0\rightarrow {\mathcal J}^{i+1} \rightarrow {\mathcal J}^{i} \rightarrow {\mathcal J}^{i}/{\mathcal J}^{i+1} \rightarrow 0.$$
Applying the long exact sequence in cohomology and using the fact that $H^{1}(G/P,{\mathcal J}^{i})=0$ yields a 
short exact sequence: 
$$0\rightarrow H^{0}(G/P,{\mathcal J}^{i+1}) \rightarrow H^{0}(G/P,{\mathcal J}^{i}) \rightarrow H^{0}(G/P,{\mathcal J}^{i}/{\mathcal J}^{i+1})\rightarrow 0.$$
For each $i$, this short exact sequence splits over $G_{\0}$ and one can deduce that 
$$(R^{0}\operatorname{ind}_{P}^{G} M)|_{G_{\0}}\cong H^{0}(G/P,{\mathcal J}^{0})\cong \oplus_{i} \operatorname{H}^{0}(G/P,{\mathcal J}^{i}/{\mathcal J}^{i+1})\cong R^{0}\operatorname{ind}_{P_{\0}}^{G_{\0}} [M \otimes \Lambda^{\bullet}(({\mathfrak g}_{\1}/{\mathfrak p}_{\1})^{*}].$$

\end{proof} 
The theorem above justifies the statement of \cite[Proposition 6.1.1]{BKN4} when the additional hypothesis is added. The results in \cite[Proposition 6.5.3]{BKN4} can be justified by using Theorem~\ref{P:Gzeroiso2}. 

Let $P$ be a parabolic subgroup with $P \subseteq G$ and let 
\begin{equation} 
p_{G,P}(t)=\sum_{i=0}^{\infty} \dim R^{i}\text{ind}_{P}^{G} {\mathbb C}\ t^{i}.
\end{equation}

The following proposition will be useful in making the transition from computing $R^{\bullet}\operatorname{ind}_{B_{\0}}^{G_{\0}}\ \Lambda^{\bullet}(({\mathfrak g}_{\1}/{\mathfrak b}_{\1})^{*})$ 
to computing $p_{G,B}(t)$ where $B$ is the parabolic defined in Section~\ref{SS:BBWparabolics}. 

\begin{proposition}\label{P:selfext} Let $j\geq 0$. 
\begin{itemize} 
\item[(a)] If ${\mathfrak g}\neq {\mathfrak q}(n)$ and $(R^{j}\operatorname{ind}_{B}^{G} {\mathbb C})|_{G_{\0}}\cong 
{\mathbb C}^{\oplus t}$, then $R^{j}\operatorname{ind}_{B}^{G} {\mathbb C}\cong 
{\mathbb C}^{\oplus t}$ as a $G$-module. 
\item[(b)] If ${\mathfrak g}= {\mathfrak q}(n)$ with $(R^{j}\operatorname{ind}_{B}^{G} {\mathbb C})|_{G_{\0}}\cong {\mathbb C}^{\oplus t}$ and 
$(R^{j}\operatorname{ind}_{B}^{G} {\mathbb C})|_{G_{\0}}\cong R^{j}\operatorname{ind}_{B_{\0}}^{G_{\0}} \Lambda^{k}(({\mathfrak g}_{\1}/{\mathfrak b}_{\1})^{*})$
for some $k\geq 0$, then $R^{j}\operatorname{ind}_{B}^{G} {\mathbb C}\cong {\mathbb C}^{\oplus t}$ as a $G$-module. 
\end{itemize} 
\end{proposition} 

\begin{proof} (a) The statement follows immediately if there are no self-extensions of the trivial module, that is, 
$\text{Ext}^{1}_{({\mathfrak g},{\mathfrak g}_{\0})}({\mathbb C},{\mathbb C})=0$. This space identifies 
with $S^{1}({\mathfrak g}_{\1})^{G_{\0}}$. For all types other than ${\mathfrak g}={\mathfrak q}(n)$ this is always equal to zero 
(cf. \cite[Table 1]{BKN1}). 

(b) Let ${\mathfrak g}={\mathfrak q}(n)$. For $j>0$, the simple modules appearing as $G$-composition 
factors in the $\pi$-$G$-module, $R^{j}\operatorname{ind}_{B}^{G} {\mathbb C}$, all have the same parity.  (i.e., they are either all ${\mathbb C}$ or all $\Pi{\mathbb C}$). 
See \cite[Lemma 4.4.]{Bru}.

Since 
$$\text{Ext}^{1}_{\text{$\pi$-$({\mathfrak g},{\mathfrak g}_{\0})$}}(\Pi{\mathbb C},\Pi{\mathbb C}) \cong \text{Ext}^{1}_{\text{$\pi$-$({\mathfrak g},{\mathfrak g}_{\0})$}}({\mathbb C},{\mathbb C})\cong S^{1}({\mathfrak g}_{\1}^{*})^{\text{$\pi$-$G_{\0}$}}\subseteq  ({\mathfrak g}_{\1})^{\pi}=0,$$ 
it follows by using the hypothesis that $R^{j}\operatorname{ind}_{B}^{G} {\mathbb C}$ as a $\pi$-$G$-module is isomorphic to 
${\mathbb C}^{\oplus t}$ if $j$ is even and $\Pi {\mathbb C}^{\oplus t}$ if $j$ is odd. Hence, as $G$-module (disregarding the grading), $R^{j}\operatorname{ind}_{B}^{G} {\mathbb C}\cong {\mathbb C}^{\oplus t}$. 

\end{proof}

\subsection{Poincare series for exceptional Lie superalgebras} In the following theorem, 
we compute $p_{G,B}(t)$ for exceptional Lie superalgebras. Although $p_{G,B}(t)$ 
is a polynomial of degree 2, the verification extensively uses the representation theory of 
$\mathfrak{sl}_{2}$, $G_{2}$ and $\mathfrak{so}_{7}$ along with the classical Bott-Borel-Weil (BBW) theorem. 

\begin{theorem}\label{T:exceptionalsheaf} Let ${\mathfrak g}=D(2,1,\alpha)$, $G(3)$ or $F(4)$ and 
${\mathfrak b}$ be the parabolic subalgebra described in Table~\ref{T:BBW-roots}. 
Then 
$$p_{G,B}(t)=1+t^{2}=z_{{\mathfrak b},{\mathfrak g}}(t)=p_{W_{\1}}(t^{2}).$$
\end{theorem} 

\begin{proof} The last two equalities follow from Theorem~\ref{T:zcompute}. It remains to 
show that $p_{G,B}(t)=1+t^{2}$. 

First consider ${\mathfrak g}=D(2,1,\alpha)$. One has ${\mathfrak g}_{\0}\cong \mathfrak{sl}_{2}\times 
\mathfrak{sl}_{2}\times \mathfrak{sl}_{2}$ with ${\mathfrak g}_{\1}\cong V\boxtimes V \boxtimes V$ where 
$V$ is the 2-dimensional natural representation of $\mathfrak{sl}_{2}$. Let $G_{\0}=G_{\0,(1)}\times G_{\0,(2)}\times G_{\0,(3)}$ 
denote the product of three copies of $SL_{2}$ with Borel subgroup $B_{\0}=B_{\0,(1)}\times B_{\0,(2)}\times B_{\0,(3)}$ 
(corresponding to the negative roots). For a given one-dimensional $B_{\0}$-module, $\mu=(\mu_{1},\mu_{2},\mu_{3})$, 
one has 
\begin{equation}\label{eq:Kunneth}
R^{n}\text{ind}_{B_{\0}}^{G_{\0}}\ \mu=\bigoplus_{n_{1}+n_{2}+n_{3}=n} 
R^{n_{1}}\text{ind}_{B_{\0,(1)}}^{G_{\0,(1)}}\ \mu_{1}\boxtimes R^{n_{2}}\text{ind}_{B_{\0,(2)}}^{G_{\0,(2)}}\ \mu_{2}\boxtimes 
R^{n_{3}}\text{ind}_{B_{\0,(3)}}^{G_{\0,(3)}}\ \mu_{3}
\end{equation}  
by the K{\"u}nneth Theorem. It follows that if any of the components vanish then $R^{\bullet}\text{ind}_{B_{\0}}^{G_{\0}}\ \mu=0$. 

The weights of ${\Lambda}^{1}(({\mathfrak g}_{\1}/{\mathfrak b}_{\1})^{*})$ are 
$\{(-\epsilon,-\epsilon,-\epsilon), (-\epsilon,-\epsilon,\epsilon), (\epsilon,-\epsilon,-\epsilon)\}$. Let $X(T_{\0})$ be the integral weights of $G_{\0}$ and 
$\overline{C}_{\mathbb Z}$ be the closure of the bottom alcove in $X(T_{\0})$. Moreover, let $X(T_{\0})_{+}$ be the set of dominant 
integral weights. See \cite[p. 571-572]{Jan} for precise definitions. 

By the BBW theorem, since all the weights of ${\Lambda}^{1}(({\mathfrak g}_{\1}/{\mathfrak b}_{\1})^{*})$ are in $\overline{C}_{\mathbb Z}-X(T_{\0})_{+}$, it follows that at least one component in the decomposition 
(\ref{eq:Kunneth}) vanishes, so $R^{\bullet}\text{ind}_{B_{\0}}^{G_{\0}} \Lambda^{1}(({\mathfrak g}_{\1}/{\mathfrak b}_{\1})^{*})=0$. 
Similarly, ${\Lambda}^{3}(({\mathfrak g}_{\1}/{\mathfrak b}_{\1})^{*})$ is one-dimensional spanned by a vector of weight 
$\mu=(\epsilon,-3\epsilon,-\epsilon)$. The last component vanishes in (\ref{eq:Kunneth}), thus 
$R^{\bullet}\text{ind}_{B_{\0}}^{G_{\0}} \Lambda^{3}(({\mathfrak g}_{\1}/{\mathfrak b}_{\1})^{*})=0$. Also, note that 
$\Lambda^{0}(({\mathfrak g}_{\1}/{\mathfrak b}_{\1})^{*})\cong {\mathbb C}$, so $R^{j}\text{ind}_{B_{\0}}^{G_{\0}} \Lambda^{0}(({\mathfrak g}_{\1}/{\mathfrak b}_{\1})^{*}) 
=0$ for $j>0$ and is isomorphic to ${\mathbb C}$ for $j=0$. 

We need to analyze $\Lambda^{2}(({\mathfrak g}_{\1}/{\mathfrak b}_{\1})^{*})$. This will entail using two-dimensional $B_{\0}$-modules. Similar 
methods will be also be employed for the $G(3)$ and $F(4)$-cases. A direct computation shows that as a 
$B_{\0}$-module, $\Lambda^{2}(({\mathfrak g}_{\1}/{\mathfrak b}_{\1})^{*})$ has head isomorphic to $(0,-2\epsilon,0)$ and two-dimensional socle 
$(0,-2\epsilon,-2\epsilon)\oplus (-2\epsilon,-2\epsilon,0)$. Therefore, one has a short exact sequence: 
\begin{equation}\label{eq:ss1}
0\rightarrow (-2\epsilon,-2\epsilon,0) \rightarrow \Lambda^{2}(({\mathfrak g}_{\1}/{\mathfrak b}_{\1})^{*})\rightarrow N \rightarrow 0
\end{equation} 
where $N$ is a two-dimensional $B_{\0}$-module isomorphic as $(B_{\0,(1)}\times B_{\0,(2)})\times B_{\0,(3)}$-module  to $(0,-2\epsilon)\boxtimes N^{\prime}$ 
where $N^{\prime}$ is a two-dimensional $B_{\0,(3)}$-module with socle $-2\epsilon$ and head ${\mathbb C}$. 

As a $B_{\0,(3)}$-module, one has 
\begin{equation}\label{eq:ss2}
0\rightarrow N^{\prime}\rightarrow L(2\epsilon)\rightarrow 2\epsilon \rightarrow 0
\end{equation}
where $L(2\epsilon)$ is the three-dimensional adjoint representation for $G_{\0,(3)}$. Now by the tensor identity, 
$R^{j}\text{ind}_{B_{\0,(3)}}^{G_{\0,(3)}} L(2\epsilon)\cong [R^{j}\text{ind}_{B_{\0,(3)}}^{G_{\0,(3)}} {\mathbb C}] \otimes 
L(2\epsilon).$ 
This is zero for $j>0$. Applying the long exact sequence in cohomology to (\ref{eq:ss2}) and the fact that 
$R^{j}\text{ind}_{B_{\0,(3)}}^{G_{\0,(3)}}\ 2\epsilon=0$ for $j>0$ shows that 
$R^{j}\text{ind}_{B_{\0,(3)}}^{G_{\0,(3)}} N^{\prime}=0$ for $j>0$. It remains to look at the remaining part of the 
long exact sequence: 
$$0\rightarrow \text{ind}_{B_{\0,(3)}}^{G_{\0,(3)}} N^{\prime} 
\rightarrow L(2\epsilon) \rightarrow L(2\epsilon) \rightarrow R^{1}\text{ind}_{B_{\0,(3)}}^{G_{\0,(3)}} N^{\prime} \rightarrow 0.$$
The only dominant weight of $N^{\prime}$ is $0$ so $\text{ind}_{B_{\0,(3)}}^{G_{\0,(3)}} N^{\prime}$ is either $0$ or 
${\mathbb C}$. This proves the arrow from $L(2\epsilon)$ to $L(2\epsilon)$ must be an isomorphism, thus 
$R^{\bullet}\text{ind}_{B_{\0,(3)}}^{G_{\0,(3)}} N^{\prime}=0$. 

Now apply the long exact sequence in cohomology (\ref{eq:ss1}) and use the fact that 
$R^{\bullet}\text{ind}_{B_{\0,(3)}}^{G_{\0,(3)}} N^{\prime}=0$. This yields 
$$R^{j}\text{ind}_{B_{\0}}^{G_{\0}} (-2\epsilon,-2\epsilon,0) \cong 
R^{j}\text{ind}_{B_{\0}}^{G_{\0}} \Lambda^{2}(({\mathfrak g}_{\1}/{\mathfrak b}_{\1})^{*})$$ 
for all $j\geq 0$. Applying the K{\"u}nneth theorem and the BBW theorem shows that 
$R^{j}\text{ind}_{B_{\0}}^{G_{\0}} \Lambda^{2}(({\mathfrak g}_{\1}/{\mathfrak b}_{\1})^{*})=0$ for $j\neq 2$ and 
$R^{2}\text{ind}_{B_{\0}}^{G_{\0}} \Lambda^{2}(({\mathfrak g}_{\1}/{\mathfrak b}_{\1})^{*})\cong {\mathbb C}$. 
Consequently, $p_{G,B}(t)=1+t^{2}$. 

For $G(3)$ and $F(4)$ the calculations are much more lengthy and involved to 
show that $p_{G,B}(t)=1+t^{2}$. First, $G_{\0}\cong G_{\0,(1)}\times G_{\0,(2)}$ has 
two components and for any one-dimensional 
$B_{\0}=B_{\0,(1)}\times B_{\0,(2)}$-module $\mu=(\mu_{1},\mu_{2})$ one has 
\begin{equation} 
R^{n}\text{ind}_{B_{\0}}^{G_{\0}}\ \mu=\bigoplus_{n_{1}+n_{2}=n} 
R^{n_{1}}\text{ind}_{B_{\0,(1)}}^{G_{\0,(1)}}\ \mu_{1}\boxtimes R^{n_{2}}\text{ind}_{B_{\0,(2)}}^{G_{\0,(2)}}\ \mu_{2}
\end{equation} 
In these cases the last component $G_{\0,(2)}$ is isomorphic to $SL_{2}$, so to prove the vanishing, the focus will be more on the 
first component $G_{\0,(1)}$ which is $G_{2}$ (resp. $\mathfrak{so}_{7}$) for $G(3)$ (resp. $F(4)$). 

We will outline the ideas to handle $G(3)$. The ideas are similar for $F(4)$ and involve more verifications. In the case of 
$G(3)$, one has $\dim {\mathfrak g}_{\1}/{\mathfrak b}_{\1}=6$. 
\vskip .25cm 
\noindent 
(1) Show that if $k\neq 0,2$ then $R^{\bullet}\text{ind}_{B_{\0}}^{G_{\0}} \Lambda^{k}(({\mathfrak g}_{\1}/{\mathfrak b}_{\1})^{*})=0$. 
\vskip .25cm 
One of the main ideas to analyze $\Lambda^{k}(({\mathfrak g}_{\1}/{\mathfrak b}_{\1})^{*})$ for $k\neq 0,2$ is to find a filtration of $B_{\0}$-modules whose 
subquotients are either one-dimensional or two-dimensional modules $N_{j}$ such that $R^{\bullet}\text{ind}_{B_{\0}}^{G_{\0}} N_{j}=0$. 
For the one-dimensional modules, one shows that the weights are in $\overline{C}_{\mathbb Z}-X(T_{\0})_{+}$. For the two dimensional 
modules, one uses the argument as given in $D(2,1,\alpha)$ so that these modules are submodules of the adjoint modules 
for a parabolic subgroup in $G_{\0,(1)}$ corresponding to an $SL_{2}$. For example, in $G(3)$, the weights 
for $\Lambda^{1}(({\mathfrak g}_{\1}/{\mathfrak b}_{\1})^{*})$ are 
$$\{(-\omega_{1}+\omega_{2},-\epsilon), (2\omega_{1}-\omega_{2},-\epsilon),  (0,-\epsilon), (\omega_{1}-\omega_{2},-\epsilon),
(-2\omega_{1}+\omega_{2},-\epsilon), (-\omega_{1}, -\epsilon) \}.$$
By the BBW theorem, all the weights except for $(0,-\epsilon)$ and $(-2\omega_{1}+\omega_{2},-\epsilon)=-\alpha_{1}$ yield no cohomology. 
One can see the vectors of these weights form a subquotient with the desired properties. 
\vskip .25cm 
\noindent 
(2) For $k=0$, analyze $R^{j}\text{ind}_{B_{\0}}^{G_{\0}} \Lambda^{0}(({\mathfrak g}_{\1}/{\mathfrak b}_{\1})^{*})\cong 
R^{j}\text{ind}_{B_{\0}}^{G_{\0}} {\mathbb C}$. 
\vskip .25cm 
From Kempf's vanishing theorem, $R^{j}\text{ind}_{B_{\0}}^{G_{\0}} {\mathbb C}=0$ for $j>0$, and one has $\text{ind}_{B_{\0}}^{G_{\0}} {\mathbb C}={\mathbb C}$. 
\vskip .25cm 
\noindent 
(3) For $k=2$, show that $R^{j}\text{ind}_{B_{\0}}^{G_{\0}}\  \Lambda^{2}(({\mathfrak g}_{\1}/{\mathfrak b}_{\1})^{*})=0$ for $j\neq 2$ and 
$R^{2}\text{ind}_{B_{\0}}^{G_{\0}}\  \Lambda^{2}(({\mathfrak g}_{\1}/{\mathfrak b}_{\1})^{*})={\mathbb C}$. 
\vskip .25cm 
A technique that is used in the verification of (3) is the existence of an embedding of $\Lambda^{2}(({\mathfrak g}_{\1}/{\mathfrak b}_{\1})^{*})$ 
into $\Lambda^{2}(L)\boxtimes (-2\epsilon)$ where $L$ is an irreducible $G_{\0,(1)}$-module. For example, 
when ${\mathfrak g}=G(3)$, one has 
\begin{equation}\label{eq:ss3}
0\rightarrow \Lambda^{2}(({\mathfrak g}_{\1}/{\mathfrak b}_{\1})^{*})\rightarrow \Lambda^{2}(L(\omega_{1}))\boxtimes (-2\epsilon) \rightarrow M 
\rightarrow 0.
\end{equation} 
From the tensor identity, one has that $R^{j}\text{ind}_{B_{\0}}^{G_{\0}}\  \Lambda^{2}(L(\omega_{1}))\boxtimes (-2\epsilon)=0$ 
for $j\geq 0$. This allows one to dimension shift via the long exact sequence in cohomology to 
concentrate on calculating $R^{j}\text{ind}_{B_{\0}}^{G_{\0}}\  M$. In the case $M$ is a 6-dimensional module. 
This makes the computations tractable to verify (3). A similar short exact sequence to (\ref{eq:ss3}) exists for $F(4)$ via the spin representation $L(\omega_{3})$ for $\mathfrak{so}_{7}$ 
and the same technique can be utilized in this case. 

\end{proof} 

\subsection{}\label{SS:TypeQ} Consider $R^{n}\text{ind}_{B}^{G} {\mathbb C}\cong R^{n}\text{ind}_{B_{\0}}^{G_{\0}}\Lambda^{\bullet}(({\mathfrak g}_{\1}/{\mathfrak b}_{\1})^{*})$ as a $G_{\0}$-module 
for $n\geq 0$. The weights of $\Lambda^{n}(({\mathfrak g}_{\1}/{\mathfrak b}_{\1})^{*})$ are $-\rho(J)$ where $J\subseteq \Phi_{\1}^{+}$ and $|J|=n$. 
Here $\rho(J)=\sum_{\alpha\in J} \alpha$. Set $\rho_{\1}=\frac{1}{2}\sum_{\alpha\in \Phi_{\1}^{+}}\alpha$ and let 
$$w\cdot \lambda=w(\lambda +\rho_{\1})-\rho_{\1}$$ 
for $w\in W_{\1}$ and $\lambda$ in the  ${\mathbb Q}$-span of $\Phi_{\1}$. This will be referred to as the {\em odd} dot action of $W_{\1}$. 

One has 
\begin{equation}\label{eq:subdivide}
\rho_{\1}-\rho(J)=\frac{1}{2}\sum_{\gamma\in J^{\prime}}\gamma -\frac{1}{2}\sum_{\alpha\in J}\alpha
\end{equation} 
where $J^{\prime}=\Phi_{\1}-J$. Now $w\in W_{\1}$ permutes the set of odd roots $\Phi_{\1}$. Under the condition 
that $\Phi_{\1}^{-}=-\Phi_{\1}^{+}$, it follows that  
$w(\rho_{\1}-\rho(J))=\rho_{\1}-\rho(J_{1})$ where $J_{1}\subseteq \Phi_{\1}^{+}$, and consequently 
\begin{equation} 
w\cdot (-\rho(J))=-\rho(J_{1})
\end{equation} 

Let $G_{\0}$ be a reductive algebraic group, $\Delta$ be the simple roots in $\Phi^{+}_{\0}$ and $W_{\0}$ be the Weyl group for 
the corresponding root system, $\Phi_{\0}$, for $G_{\0}$. Let $\rho^{\Phi_{\0}}:=\rho_{\0}=\frac{1}{2}\sum_{\alpha\in \Phi_{\0}^{+}}\alpha$ and 
denote the {\em even} dot action by $w\circ \lambda=w(\lambda+\rho_{\0})-\rho_{\0}$ where $w\in W_{\0}$ and $\lambda\in X(T_{\0})$. 

Table~\ref{T:sumevenodd} provides the relationship between $\rho_{\1}$ and $\rho_{\0}$. Observe that in the case when 
${\mathfrak g}$ is $A(n|n)$ or $\mathfrak{osp}(2n+1|2n)$, $G_{\0}\cong G_{\0,(1)}\times G_{\0,(2)}$. In these cases $\rho_{\0}$ 
is a sum $\rho_{\0,(1)}+\rho_{\0,(2)}$ where $\rho_{\0,(j)}$ is the half sum of positive roots arising from $G_{\0,(j)}$ where $j=1,2$. 

\renewcommand{\arraystretch}{1.2}
\begin{table}[htp] 
\caption{Sums of even and odd roots}
\begin{center}
\begin{tabular}{cc}\label{T:Poincareseries}
${\mathfrak g}$ &   \\ \hline
${\mathfrak q}(n)$ & $\rho_{\1}=\rho_{\0}$   \\
$\mathfrak{psq}(n)$ &  $\rho_{\1}=\rho_{\0}$   \\    
$A(n|n)$ &  $\rho_{\1}=\rho_{\0}=\rho_{\0,(1)}^{A_{n-1}}+\rho_{\0,(2)}^{A_{n-1}}$\\
$\mathfrak{osp}(2n+1|2n)$ &  $\rho_{\1}=\rho_{\0}=\rho_{\0,(1)}^{B_{n}}+\rho_{\0,(2)}^{C_{n}}$  \\
\end{tabular}
\end{center}
\label{T:sumevenodd}
\end{table}

A key idea to calculate $R^{\bullet}\text{ind}_{B}^{G}{\mathbb C}$ entails connecting the even and odd dot actions on weights of $\Lambda^{\bullet}({\mathfrak u}_{\1})$ 
as shown in the next example. 

\begin{example}\label{E:TypeQ} Let ${\mathfrak g}={\mathfrak q}(n)$ or $\mathfrak{psq}(n)$. There exists a $B_{\0}$-isomorphism 
$\Lambda^{\bullet}({\mathfrak u}_{\0}^{-})\cong \Lambda^{\bullet}(({\mathfrak g}_{\1}/{\mathfrak b}_{\1})^{*})$. 
Furthermore, $\rho_{\1}=\rho_{\0}$ and the even and odd dot actions coincide. 

One can now directly apply \cite[II 6.18, Proposition]{Jan} to conclude that 
$R^{n}\text{ind}_{B_{\0}}^{G_{\0}}\Lambda^{\bullet}(({\mathfrak g}_{\1}/{\mathfrak b}_{\1})^{*})\cong {\mathbb C}^{\oplus t_{n}}$ where 
$t_{n}=|\{w\in \Sigma_{n}:\ l(w)=n\}|$. The contributions in this cohomology group are given by weights in $\Lambda^{n}(({\mathfrak g}_{\1}/{\mathfrak b}_{\1})^{*})$, 
so one can apply Proposition~\ref{P:selfext}(b) to conclude that $p_{G,B}(t)=p_{W_{\1}}(t)$.  This result generalizes the ${\mathfrak q}(2)$ 
example computed by Brundan \cite[Lemma 4.4]{Bru} for all ${\mathfrak q}(n)$, $n\geq 1$.  
\end{example} 

\subsection{Combinatorics with odd roots} We will start by focusing on the cases when 
${\mathfrak g}$ is of type $A(n|n)$ or $\mathfrak{osp}(2n+1|2n)$.  

In this setting $G_{\0}$ is a product of two reductive algebraic groups which is unlike the case for type $Q$. The dot action of 
the group $W_{\1}$ on $\Phi_{\1}$ is more complicated in this setting, yet one still has 
a beautiful connection between $w\cdot 0$ with natural subsets of roots in $\Phi_{\1}^{+}$. 
We will consider the following set of even simple roots for ${\mathfrak g}_{\0}$ given in Table~\ref{T:evenroots} (cf. \cite[ 12.1]{Hum1}).

\begin{center}
\begin{table}[htp] 
\caption{Even Simple Roots}
\begin{tabular}{cccc} 
${\mathfrak g}$ &  $\bar{{\Delta}}_{\0,(1)}$ & $\bar{{\Delta}}_{\0,(2)}$              \\ \hline
$A(n|n)$ & $\{\epsilon_{1}-\epsilon_{2},\dots, \epsilon_{n-1}-\epsilon_{n}\}$ &    $\{\delta_{1}-\delta_{2},\dots, \delta_{n-1}-\delta_{n}\}$    \\   
$\mathfrak{osp}(2n+1|2n)$ & $\{\epsilon_{1}-\epsilon_{2},\dots, \epsilon_{n-1}-\epsilon_{n},\epsilon_{n}\}$ &    $\{\delta_{1}-\delta_{2},\dots, \delta_{n-1}-\delta_{n},2\delta_{n}\}$    \\                                                                   
\end{tabular}
\label{T:evenroots} 
\end{table}
\end{center}

Set $I=\{1,2,\dots,n-1\}$ for $A(n|n)$ (resp. $I=\{1,2,\dots,n\}$ for $\mathfrak{osp}(2n+1|2n)$). 
Let $s_{j,\0,(1)}$ (resp. $s_{j,\0,(2)}$) be the reflection corresponding to the $j$th root in 
${{\Delta}}_{\0,(1)}$ (resp. ${{\Delta}}_{\0,(2)}$).  For $j\in I$, set 
$$s_{j}:=s_{j,\0,(1)}s_{j,\0,(2)}.$$
Then $s_{j}$ is a simple reflection in $W_{\1}$ and $W_{\1}$ is generated by $\{s_{j}:\ j\in I\}$. 

\begin{example}\label{ex:reflections} Let ${\mathfrak g}=A(n|n)$. Observe that 
$s_{j}(\epsilon_{j}-\delta_{j+1})=\delta_{j+1}-\epsilon_{j}\in \Phi^{-}_{\1}$ and 
$s_{j}(\delta_{j}-\epsilon_{j+1})=\epsilon_{j+1}-\delta_{j}\in \Phi^{-}_{\1}$. 
Furthermore, 
$$s_{j}(\Phi_{\1}^{+}-\{\epsilon_{j}-\delta_{j+1},\delta_{j}-\epsilon_{j+1}\})\subset \Phi_{\1}^{+}.$$ 
That is, the only roots in $\Phi_{\1}^{+}$ that are sent to $\Phi_{\1}^{-}$ are 
$\epsilon_{j}-\delta_{j+1}$ and $\delta_{j}-\epsilon_{j+1}$. 
\end{example} 

From direct computations, one can show that Example~\ref{ex:reflections} extends to the other 
algebras listed in Table~\ref{T:evenroots}. There are two sets of roots in $\Phi_{\1}^{+}$, $\bar{\Delta}_{\1,(1)}=\{\beta_{j}: j\in I\}$ and 
$\bar{\Delta}_{\1,(2)}=\{\gamma_{j}: j\in I\}$ with the property that
\begin{itemize} 
\item $s_{j}(\{\beta_{j},\gamma_{j}\})\subset \Phi_{\1}^{-}$ with $s_{j}(\beta_{j})=-\beta_{j}$ and $s_{j}(\gamma_{j})=-\gamma_{j}$ 
\item $s_{j}(\Phi_{\1}^{+}-\{\beta_{j},\gamma_{j}\})\subset \Phi_{\1}^{+}$ 
\end{itemize} 
The following table gives this correspondence. 
\vskip .5cm 
\begin{center}
\begin{tabular}{cccc}
${\mathfrak g}$ &  $\bar{{\Delta}}_{\1,(1)}$ & $\bar{{\Delta}}_{\1,(2)}$              \\ \hline
$A(n|n)$ & $\{\epsilon_{1}-\delta_{2},\dots, \epsilon_{n-1}-\delta_{n}\}$ &    $\{\delta_{1}-\epsilon_{2},\dots, \delta_{n-1}-\epsilon_{n}\}$    \\   
$\mathfrak{osp}(2n+1|2n)$ & $\{\epsilon_{1}-\delta_{2},\dots, \epsilon_{n-1}-\delta_{n},\delta_{n}\}$ &    
$\{\delta_{1}-\epsilon_{2},\dots, \delta_{n-1}-\epsilon_{n},\epsilon_{n}+\delta_{n}\}$    \\                                                                   
\end{tabular}
\end{center}
\vskip .5cm 
For $w\in W_{\1}$ set 
\begin{equation} \label{E:Phiw}
\Phi(w) = -(w\Phi_{\1}^+\cap\Phi_{\1}^-) = w\Phi_{\1}^-\cap\Phi_{\1}^+\subset \Phi_{\1}^+.
\end{equation}

The following results establish some basic facts about $\Phi(w)$.

\begin{proposition}\label{P:combinatorics} Let $w\in W_{\1}$. 
\begin{itemize}
\item[(a)] \label{I:card} $|\Phi(w)|=2\cdot l(w)$. 
\item[(b)] \label{I:wdot0} $w\cdot 0=-\rho(\Phi(w))$.
\item[(c)] \label{I:Phiw} If $w=s_{j_1}\dots s_{j_t}$ is a reduced expression, then
$$\Phi(w)=\{\beta_{j_{1}},s_{j_{1}}\beta_{j_{2}},s_{j_{1}}s_{j_{2}}\beta_{j_{3}},\dots,s_{j_{1}}s_{j_{2}}\dots s_{j_{t-1}}\beta_{j_{t}}\}\cup 
\{\gamma_{j_{1}},s_{j_{1}}\gamma_{j_{2}},s_{j_{1}}s_{j_{2}}\gamma_{j_{3}},\dots,s_{j_{1}}s_{j_{2}}\dots s_{j_{t-1}}\gamma_{j_{t}}\}.
$$
\item[(d)] If $w\cdot 0 = - \rho(J)$ for some 
$J\subset\Phi^+_{\1}$ then $J=\Phi(w)$.
\end{itemize}
\end{proposition} 

\begin{proof} (a) (b) and (c): One proves these statements using induction on $l(w)$. 
When $w=\text{id}$ (i.e., $l(w)=0$), these statements are clear. Now suppose that 
$w\in W_{\1}$ and $w=s_{j}w^{\prime}$ where $l(w)=l(w^{\prime})+1$. One has 
$\beta_{j}\notin \Phi(w^{\prime})$ and $\gamma_{j}\notin \Phi(w^{\prime})$ due to the minimality of 
the expression $w=s_{j}w^{\prime}$. Since $s_{j}$ sends all roots in $\Phi_{\1}^{+}$ other than $\beta_{j}$ and $\gamma_{j}$ to $\Phi_{\1}^{+}$, 
one can express 
\begin{equation}\label{eq:union}
\Phi(w)=s_{j}(\Phi(w^{\prime}))\cup \{\beta_{j},\gamma_{j}\}.
\end{equation} 
This is a disjoint union of sets. This proves (a) and (c). 

For (b), one observes that 
$$w\cdot 0=s_{j}\cdot(w^{\prime}\cdot 0)=s_{j}\cdot(-\rho(\Phi(w^{\prime})))=-s_{j}\rho(\Phi(w^{\prime}))+s_{j}\cdot 0=
-s_{j}\rho(\Phi(w^{\prime}))-\beta_{j}-\gamma_{j}=-\rho(\Phi(w))$$ 
by using (\ref{eq:union}).

(d) We adapt the line of reasoning given in \cite[Lemma 3.1.2(b)]{UGAVIGRE}. Statement (d) will be 
proved by induction on $l(w)$.  If $w=\text{id}$ or equivalently $l(w)=0$ then 
$w\cdot 0=0$, so $J=\varnothing=\Phi(\text{id})$. 

Let $w\in W_{\1}$ with $l(w)>0$. One has 
write $w=s_j w^{\prime}$ with $l(w)=l(w^{\prime})+1$. 
We have $\beta_{j},\gamma_{j}\in \Phi(w)$ and these elements are not in 
$\Phi(w^{\prime})$. 

Let $w\cdot 0 = -\rho(J)$ where $J=\{\sigma_1,\sigma_{2},\dots,\sigma_m\}\in \Phi_{\1}^{+}$.  
Then
$$w'\cdot 0 = s_j \cdot (w\cdot 0)=s_j(w\cdot 0) + s_j \rho - \rho 
= -(s_j \sigma_1+\dots +s_j \sigma_m + \beta_{j}+\gamma_{j}).
$$
There are two cases.
\vskip .25cm 

\noindent{\bf Case 1}: $\sigma_i\neq \beta_{j}$ or $\gamma_{j}$ for all $i$. Without loss of generality we may 
assume that $i=m$ when there is equality. In this case, each of the three sets (that we
will denote by $J^{\prime}$) 
\begin{itemize} 
\item $\{s_j\sigma_1,\dots, s_j \sigma_m,\beta_{j},\gamma_{j}\}$ 
\item $\{s_j\sigma_1,\dots, s_j \sigma_{m-1},\beta_{j}\}$ 
\item $\{s_j\sigma_1,\dots, s_j \sigma_{m-1},\gamma_{j}\}$
\end{itemize}  
yields distinct elements in $\Phi_{\1}^{+}$ whose sum equals $-w^{\prime}\cdot 0$.  Now by induction, $-\rho(J')=\Phi(w')$, 
This is a contradiction because $\beta_{j} \notin\Phi(w')$ and $\gamma_{j} \notin\Phi(w')$.  
\vskip .25cm

\noindent{\bf Case 2}: $\sigma_i=\beta_{j}$ and $\sigma_{k}=\gamma_{j}$ for some $i,k$. We may assume that 
$i=m$ and $k=m-1$. so $w'\cdot 0 = -(s_j\sigma_1+\dots +s_j\sigma_{m-2})$. 
By induction, $\Phi(w')=\{s_j\sigma_1,\dots,s_j\sigma_{m-2}\}$. Consequently, 
$\Phi(w)=\Phi(w')\cup\{\beta_{j},\gamma_{j}\} = \{\sigma_1,\dots,\sigma_m\}$. 
\end{proof} 

\subsection{}\label{SS:TypeABC} In this section we compute $p_{G,B}(t)$ for the algebras listed in Table~\ref{T:evenroots}. 
This will be accomplished in a series of steps. Recall that $G_{\0}=G_{\0,(1)}\times G_{\0,(2)}$. It will be convenient to 
view a weight of $G_{\0,(1)}\times G_{\0,(2)}$ as a pair $(\sigma_{1},\sigma_{2})$ that is expressed as $\sigma_{1}+\sigma_{2}$ when 
considered as a weight of $\Lambda^{\bullet}(({\mathfrak g}_{\1}/{\mathfrak b}_{\1})^{*})$. 

When $\rho_{\bar{0}}=\rho_{\bar{1}}$, the even dot action is compatible with the odd action for $w\in W_{\bar{1}}$. That is, 
$$(w,w)\circ \mu=w\cdot \mu$$ 
for $w\in W_{\1}$ and for any weight $\mu$ (i.e., in the span of the $\epsilon$'s and $\delta$'s). This observation 
about the compatibility of the actions is central to making the computations in the paper.
\vskip .5cm 
\noindent 
(1) {\em If $\sigma=(\sigma_{1},\sigma_{2})$ is a weight of $\Lambda^{\bullet}(({\mathfrak g}_{\1}/{\mathfrak b}_{\1})^{*})$ and 
$R^{n}\operatorname{ind}_{B_{\0}}^{G_{\0}} (\sigma_{1},\sigma_{2}) \neq 0$ then $\sigma=(w_{1}\circ 0,w_{2}\circ 0)$ where $w_{1},w_{2}\in W_{\1}$}. Here 
$\circ$ denotes the dot action of the Weyl group of $G_{\0,(1)}\times G_{\0,(2)}$. 
\vskip .5cm 
\noindent 
First consider ${\mathfrak g}=A(n|n)$. Then $W_{\1}$ can be identified as the diagonal embedding of $\Delta:\Sigma_{m}\hookrightarrow \Sigma_{m}\times \Sigma_{m}$ 
with $w\in W_{\1}$ represented as $(w,w)$.  There exists $\lambda_{j}\in X(T_{\0,(j)})_{+}$ and $w_{j}\in \Sigma_{m}$ for $j=1,2$ such 
that $\sigma_{1}=w_{1}\circ \lambda_{1}$ and $\sigma_{2}=w_{2}\circ \lambda_{2}$. We have 
\begin{equation}\label{eq:generic}
w_{\1}\circ \lambda_{1}+w_{2}\circ \lambda_{2}=-\rho(J)
\end{equation}  
for some $J\subseteq \Phi_{\1}^{+}$. Since $\rho_{\0}=\rho_{\0,(1)}+\rho_{\0,(2)}=\rho_{\1}$, applying $(w_{1}^{-1},w_{1}^{-1})\in W_{\1}$ to this equation yields 
\begin{equation}\label{eq:generic2}
\lambda_{1}+(w_{1}^{-1}w_{2})\circ \lambda_{2}=w_1^{-1}\cdot (-\rho(J))=-\rho(J_{1})
\end{equation}
for some $J_{1}\subseteq \Phi_{\1}^{+}$. 

We claim that the dominance condition on $\lambda_{1}$ forces $\lambda_{1}=0$. One has 
\begin{equation}\label{eq:-rho}
-\rho(J_{1})=\sum_{i>j}m_{i,j}(\epsilon_{i}-\delta_{j})+\sum_{i>j}n_{i,j}(\delta_{i}-\epsilon_{j}).
\end{equation}
with $1\leq i,j \leq n$ and $m_{i,j}, n_{i,j}\geq 0$. In (\ref{eq:generic2}), the term $(w_{1}^{-1}w_{2})\circ \lambda_{2}$ only involves $\delta_{j}$'s. 
The term involving $\epsilon_{1}$ in (\ref{eq:-rho}) is less than or equal to zero, whereas the term involving $\epsilon_{n}$ is 
greater than or equal to zero. Since $\lambda_{1}$ is dominant it follows that $\lambda_{1}=0$. 
Therefore, $w_{1}\circ 0+ w_{2}\circ \lambda_{2}=-\rho(J)$. Apply $w_{2}^{-1}$ to both sides and repeat the argument above to get that $\lambda_{2}=0$. 

Next consider ${\mathfrak g}=\mathfrak{osp}(2n+1|2n)$. Then $W_{\1}=\Delta(\Sigma_{m}\ltimes ({\mathbb Z}_{2})^{m})$. 
Given (\ref{eq:generic}), one can use the same line of reasoning as in the preceding paragraph with a few modifications. 
One needs to add the additional term to (\ref{eq:-rho}): $\sum_{i,j}q_{i,j}(\epsilon_{i}+\delta_{j})+\sum_{j}r_{j}\delta_{j}$ with 
$q_{i,j},r_{j}\leq 0$. The dominance condition for $\mathfrak{so}(2n+1)$ (resp. $\mathfrak{sp}(2n)$) entails that the coefficient 
involving $\epsilon_{n}$ (resp. $\delta_{n}$) is greater than or equal to zero. This allows us to show that $\lambda_{1}=0$ and $\lambda_{2}=0$. 

\vskip .5cm

\noindent 
(2) {\em If $R^{n}\operatorname{ind}_{B_{\0}}^{G_{\0}} (w_{1}\circ 0,w_{2}\circ 0) \neq 0$ for $w_{1},w_{2}\in W_{\1}$ then $w_{1}=w_{2}$}. 
\vskip .5cm 
Suppose that $w_{1}\circ 0+w_{2}\circ 0=-\rho(J)$ for some $J\subseteq \Phi_{\1}^{+}$. Then $(w_{1}^{-1}w_{2})\circ 0=-\rho(J_{1})$ for some 
$J_{1}\subseteq \Phi_{\1}^{+}$. Set $w=w_{1}^{-1}w_{2}$, and note that $-\rho(J_{1})$ consists of a negative sum of roots for 
$G_{\0,(2)}$ (i.e, roots involving $\delta$'s). 

Let $l(w)>0$ and $w=s_{i,\0,(2)}w^{\prime}$ be a reduced expression. 
Then 
$$s_{i,\0,(1)}\circ 0+w^{\prime}\circ 0=s_{i}\cdot (-\rho(J_{1}))=s_{i}(-\rho(J_{1}))+s_{i}\cdot 0$$ 
Now $s_{i,\0,(1)}\cdot 0=-(\epsilon_{i}-\epsilon_{i+1})$ and 
$s_{i}\cdot 0=-(\epsilon_{i}-\epsilon_{i+1})-(\delta_{i}-\delta_{i+1})$. Set $\bar{\alpha}=\delta_{i}-\delta_{i+1}$. 
It follows that 
\begin{eqnarray*} 
w^{\prime}\cdot 0 &=& s_{i}(-\rho(J_{1}))-\bar{\alpha} \\
&=&-\rho(J_{1})+\langle -\rho(J_{1}), \bar{\alpha}^{\vee} \rangle \bar{\alpha}-\bar{\alpha}\\
&=& w\cdot 0 +[\langle -\rho(J_{1}), \bar{\alpha}^{\vee} \rangle-1] \bar{\alpha}.
\end{eqnarray*} 
Therefore, 
\begin{equation}\label{eq:comparisionbar-alpha}
w^{\prime}\cdot 0 - w\cdot 0= -[\langle -\rho(J_{1}), \bar{\alpha}^{\vee} \rangle+1] \bar{\alpha}.
\end{equation}
From the explicit descriptions of the negative roots summing to $w^{\prime}\cdot 0$ and $w\cdot 0$, one 
can conclude that $w^{\prime}\cdot 0 - w\cdot 0=-(w^{\prime})^{-1}(\bar{\alpha})$. The equation (\ref{eq:comparisionbar-alpha}) 
shows that $(w^{\prime})^{-1}(\bar{\alpha})=\bar{\alpha}$ and $\langle -\rho(J_{1}), \bar{\alpha}^{\vee} \rangle=0$. 
Therefore, 
$$0=\langle -\rho(J_{1}), \bar{\alpha}^{\vee} \rangle=0=\langle w\cdot 0, \bar{\alpha}^{\vee} \rangle=
\langle w\rho_{\0}^{(2)}-\rho_{\0}^{(2)}, \bar{\alpha}^{\vee} \rangle.$$ 
Consequently, $\langle w\rho_{\0}^{(2)}, \bar{\alpha}^{\vee} \rangle=1$. On the other hand,  
$$\langle w\rho_{\0}^{(2)}, \bar{\alpha}^{\vee} \rangle=\langle s_{i,\0,(2)}w^{\prime}\rho_{\0}^{(2)}, \bar{\alpha}^{\vee} \rangle=
\langle w^{\prime}\rho_{\0}^{(2)}, -\bar{\alpha}^{\vee} \rangle=\langle \rho_{\0}^{(2)}, -\bar{\alpha}^{\vee} \rangle=-1.$$ 
This is a contradiction, so $l(w)=0$ and $w_{1}=w_{2}$. 
\vskip .5cm 
\noindent 
(3) {\em $\dim \Lambda^{\bullet}(({\mathfrak g}_{\1}/{\mathfrak b}_{\1})^{*})_{(w\circ 0,w\circ0)}=1$ for all $w\in W_{\1}$.}
\vskip .5cm
\noindent 
The statement (3) follows from Proposition~\ref{P:combinatorics}.  
\vskip .5cm 
Let $n\geq 0$. According to the K{\"u}nneth formula 
\begin{equation} 
R^{n}\text{ind}_{B_{\0}}^{G_{\0}}\ (w\circ 0, w\circ 0)=\bigoplus_{n_{1}+n_{2}=n} 
R^{n_{1}}\text{ind}_{B_{\0,(1)}}^{G_{\0,(1)}}\ w\circ 0\ \boxtimes R^{n_{2}}\text{ind}_{B_{\0,(2)}}^{G_{\0,(2)}}\ w\circ 0.
\end{equation}
This shows that 
$$
R^{n}\text{ind}_{B_{\0}}^{G_{\0}}\ (w\cdot 0, w\cdot 0)\cong 
\begin{cases} {\mathbb C} & \text{$n=2l(w)$} \\
0 & \text{otherwise}.
\end{cases} 
$$ 
From (1), (2) and (3), one can conclude that 
$R^{n}\text{ind}_{B_{\0}}^{G_{\0}} \Lambda^{\bullet}(({\mathfrak g}_{\1}/{\mathfrak b}_{\1})^{*})\cong {\mathbb C}^{\oplus t_{n}}$ where 
$t_{n}=|\{w\in W_{\1}:\ \frac{n}{2}=l(w)\}|$ for $n$ even and zero when $n$ is odd. 
Consequently, by Proposition~\ref{P:selfext}(a), one has that 
$p_{G,B}(t)=p_{W_{\1}}(t^{2})$. 

\subsection{Computing Poincar\'e Series via Spectral Sequences} 

Let $P$ be a parabolic subgroup such that $B\subseteq P \subseteq G$. The following result enables one to compute 
$p_{G,B}(t)$ from $p_{G,P}(t)$ and $p_{P,B}(t)$. 

\begin{proposition} \label{P:Poincarespectral} Let $P$ be a parabolic subgroup such that $B\subseteq P \subseteq G$. Suppose that 
\begin{itemize} 
\item[(a)] $R^{2\bullet}\operatorname{ind}_{B}^{P} {\mathbb C}\cong {\mathbb C}^{\oplus s}$ and $R^{2\bullet+1}\operatorname{ind}_{B}^{P}{\mathbb C}=0$; 
\item[(b)] $R^{2\bullet+1}\operatorname{ind}_{P}^{G}{\mathbb C}=0$. 
\end{itemize} 
Then $p_{G,B}(t)=p_{G,P}(t)\cdot p_{P,B}(t)$. 
\end{proposition} 

\begin{proof} There exists a first quadrant spectral sequence
$$E_{2}^{i,j}=R^{i}\text{ind}_{P}^{G}\ R^{j}\text{ind}_{B}^{P} {\mathbb C} \Rightarrow R^{i+j}\text{ind}_{B}^{G} {\mathbb C}.$$
From (a), since the $P$-modules, $R^{\bullet}\operatorname{ind}_{B}^{P} {\mathbb C}$ are either $0$ or a direct sum of trivial modules, one 
can regard these modules as $G$-modules (i.e., the $P$-module structure lifts to $G$).  
Therefore, by the tensor identity,  the $E_{2}$-page can be expressed as a tensor product
$$E_{2}^{i,j}=R^{i}\text{ind}_{P}^{G} {\mathbb C}\otimes  R^{j}\text{ind}_{B}^{P} {\mathbb C}.$$ 
According to (b), $E_{2}^{i,j}=0$ has non-zero terms only if $i$ and $j$ are both even. The differentials in the 
spectral sequence have bidegree $(r,1-r)$. Therefore, the spectral sequence must collapse and yields 
$R^{\bullet}\text{ind}_{B}^{G} {\mathbb C}\cong R^{\bullet}\text{ind}_{P}^{G} {\mathbb C}\otimes  R^{\bullet}\text{ind}_{B}^{P} {\mathbb C}$. 
This proves the statement of the proposition.  
\end{proof}

\subsection{${\mathfrak g}=\mathfrak{osp}(2n|2n)$ for $n\geq 1$}\label{S:ospeven}

We begin by comparing the even and odd roots for ${\mathfrak g}=\mathfrak{osp}(2n|2n)$ through the 
information below. 

\begin{center}
\begin{table}[htp]
\caption{Even Simple Roots}
\begin{tabular}{cccc} 
${\mathfrak g}$ &  $\bar{{\Delta}}_{\0,(1)}$ & $\bar{{\Delta}}_{\0,(2)}$              \\ \hline
$\mathfrak{osp}(2n|2n)$    &  $\{\epsilon_{1}-\epsilon_{2},\dots, \epsilon_{n-1}-\epsilon_{n},\epsilon_{n-1}+\epsilon_{n}\}$ &    $\{\delta_{1}-\delta_{2},\dots, \delta_{n-1}-\delta_{n},2\delta_{n}\}$    
\end{tabular}
\label{T:evenrootsosp} 
\end{table}
\end{center}

In the case when ${\mathfrak g}=\mathfrak{osp}(2n|2n)$ one has $\rho_{\0}\neq \rho_{\1}$. Instead, 
$$\rho_{\1}=\rho^{D_{n}}+2[\epsilon_{1}+\dots+\epsilon_{n}]+\rho^{C_{n}}$$
This necessitates the use of different techniques than the ones used for $A(n|n)$ and $\mathfrak{osp}(2n+1|2n)$ 
(when $\rho_{\0}=\rho_{\1}$). The root system for type $D_{n}$ embeds in the root system for type $C_{n}$ and  
one also has the relationship $\rho^{C_{n}}=\rho^{D_{n}}+[\epsilon_{1}+\dots+\epsilon_{n}]$. 

Consider the subalgebra $\mathfrak{gl}(n|n)$ in $\mathfrak{osp}(2n|2n)$ and the parabolic subalgebra 
generated by $\mathfrak{gl}(n|n)$ and the root vectors of weights $\{-\epsilon_{i}-\delta_{j}:\ 1\leq i,j \leq n\}$, and let $P$ be the corresponding parabolic subgroup scheme. 
From our prior section, $R^{\bullet}\text{ind}_{B}^{P} {\mathbb C}$ is isomorphic to a direct sum of trivial modules and 
$$\sum_{j=0}^{\infty} \dim R^{j}\text{ind}_{B}^{P} {\mathbb C}\ t^{j}=p_{\Sigma_{n}}(t^{2}).$$ 
It suffices to show that 
\begin{equation} \label{eq:Poincarequotient}
\sum_{i=0}^{\infty} \dim R^{i}\text{ind}_{P}^{G} {\mathbb C}\ t^{i}=p_{W_{\1}/\Sigma_{n}}(t^{2})=(1+t^{2})(1+t^{4})\dots (1+t^{2n})
\end{equation}
If this holds, then by Proposition~\ref{P:Poincarespectral}, $p_{G,B}(t)=p_{W_{\1}}(t^{2})$. 

For the case ${\mathfrak g}=\mathfrak{osp}(2n|2n)$, one has $W_{\1}\cong \Sigma_{n}\ltimes (\Z_{2})^{n-1}$. First observe that 
$$R^{n}\text{ind}_{P}^{G} {\mathbb C}|_{G_{\0}}\cong R^{n}\text{ind}_{P_{\0}}^{G_{\0}} \Lambda^{\bullet}(({\mathfrak g}_{\1}/{\mathfrak p}_{\1})^{*}) 
\cong R^{n}\text{ind}_{B_{\0}}^{G_{\0}} \Lambda^{\bullet}(({\mathfrak g}_{\1}/{\mathfrak p}_{\1})^{*}).$$
The weights of $({\mathfrak g}_{\1}/{\mathfrak p}_{\1})^{*}$ are the roots $-\Phi_{\mathfrak p}^{+}:=\{-\epsilon_{i}-\delta_{j}:\ 1\leq i,j \leq n\}$. 
Suppose that $w_{1}\circ \lambda+w_{2}\circ \mu=-\rho(J)$ where $J\subseteq \Phi_{\mathfrak p}^{+}$. Using the argument in 
Section~\ref{SS:TypeABC}(1), we can deduce that 
\begin{equation} \label{eq:lambda}
1\geq \lambda_{1}\geq \lambda_{2} \geq \dots \geq \lambda_{n-1} \geq |\lambda_{n}|
\end{equation} 
and 
\begin{equation} \label{eq:mu}
1\geq \mu_{1}\geq \mu_{2} \geq \dots \geq \mu_{n-1} \geq \mu_{n} \geq 0. 
\end{equation} 
From (\ref{eq:mu}), $\mu=\delta_{1}+\delta_{2}+\dots + \delta_{s}$. Then 
$$w_{2}\circ \mu=w_{2}(\mu+\rho^{C_{n}})-\rho^{C_{n}}.$$ 
If $s\geq 1$, then the first term in $\mu+\rho^{C_{n}}$ is $(n+1)\delta_{1}$ and must be sign changed to obtain a summand involving $-\delta_{j}$'s in 
$-\rho(J)$. However, if this term is sign changed to $-(n+1)$ and possibly permuted, then the corresponding term in $w_{2}\circ \mu$ is at most $-n-2$ which 
less than $-n$. This leads to a contradiction, thus $s=0$ and $\mu=0$. Consider $w_{2}\circ 0=w_{2}(\rho^{C_{n}})-\rho^{C_{n}}$ with 
$\rho^{C_{n}}=n\delta_{1}+(n-1)\delta_{2}+\dots+\delta_{n}=(n,n-1,\dots,1)$. Using (\ref{eq:constraint}) shows 
that $w_{2}$ must fix $n\delta_{1}$. This proves that 
$w_{1}\circ \lambda+w_{2}\circ 0=-\rho(J)$ where $J\subseteq \{-\epsilon_{i}-\delta_{j}:\ 1\leq i\leq n,\ 2\leq j \leq n\}$. 

Next from (\ref{eq:lambda}), $\lambda=\epsilon_{1}+\epsilon_{2}+\dots +\epsilon_{s} \pm \epsilon_{n}$ and consider 
$w_{1}\circ \lambda=w_{1}(\lambda+\rho^{D_{n}})-\rho^{D_{n}}$. Now if $s\geq 1$ then the term $n\epsilon_{1}$ in 
$\lambda+\rho^{D_{n}}$ must change sign so there is an $\epsilon$-coefficient in $w_{1}\circ \lambda$ less than or equal to $-n$. 
However, from the preceding paragraph, in $-\rho(J)$ the coefficient of $\epsilon_{i}$ is greater than $-(n-1)$. Therefore, $s=0$ and 
by the dominance condition, $\lambda=0$. 

We will prove (\ref{eq:Poincarequotient}) by induction $n$. Assume for $n-1$,  
$$\sum_{i=0}^{\infty} \dim R^{i}\text{ind}_{P}^{G} {\mathbb C}\ t^{i}=(1+t^{2})(1+t^{4})\dots (1+t^{2(n-1)})$$ 
given via $2^{(n-1)}$ solutions of $w_{1}\circ 0+ w_{2}\circ 0=-\rho(J)$, $J\subseteq \{\epsilon_{i}-\delta_{j}:\ 
2\leq i \leq n,\ 2\leq j \leq n\}$ with $l(w_{1})=l(w_{2})$. For $\mathfrak{osp}(2|2)$ (i.e., $n-1=1$), this can be verified directly. 

Suppose for ${\mathfrak g}=\mathfrak{osp}(2n|2n)$, one has 
\begin{equation}\label{eq:constraint}
w_{1}\circ 0+ w_{2}\circ 0=-\rho(J) 
\end{equation}
with $J\subseteq \{\epsilon_{i}-\delta_{j}:\ 
1\leq i \leq n,\ 1\leq j \leq n\}$. First consider $w_{2}\circ 0=w_{2}(\rho^{C_{n}})-\rho^{C_{n}}$ with 
$\rho^{C_{n}}=n\delta_{1}+(n-1)\delta_{2}+\dots+\delta_{n}=(n,n-1,\dots,1)$. From (\ref{eq:constraint}) one can deduce that  
that $w_{2}$ must fix $n\delta_{1}$ and either (i) fix $(n-1)\delta_{2}$ or (ii) permute and sign change $(n-1)\delta_{2}$ to 
$-(n-1)\delta_{n}$. In the first case (i),
$$w_{2}\circ 0=(0,0,*,*,\dots,*).$$ 
In the case (i) consider $w_{1}\circ 0$ where 
$\rho^{D_{n}}=(n-1)n\epsilon_{1}+(n-2)\delta_{2}+\dots+\epsilon_{n-1}=(n-1,n-2,\dots,1,0)$. 
Under $w_{1}$, either $(n-1)$ is fixed or gets sign changed to $-(n-1)$ and is permuted in 
the $\epsilon_{n}$-position. The latter is not possible because $w_{2}\circ 0=(0,0,*,*,\dots,*)$ (i.e.,  
only $\epsilon_{n}+\delta_{j}$ can occur for $j=3,\dots,n$). Hence, in case (i), 
$w_{1}\circ 0=(0,*,*,\dots,*)$. The conclusion is that in case (i), we are reduced to the 
$2^{(n-1)}$ solutions of $w_{1}\circ 0+ w_{2}\circ 0=-\rho(J)$ in $\mathfrak{osp}(2(n-1)|2(n-1))$. 

Next we handle case (ii). We can reduce to the solutions of (\ref{eq:constraint}) in $\mathfrak{osp}(2(n-1)|2(n-1))$ by multiplying by an element 
with length $2n$. In case (ii), $w_{2}\circ 0=(0,b_{1},b_{2},\dots,b_{n-2},-n)$ and $w_{1}\circ 0=(a_{1},a_{2},\dots,a_{n-1},-(n-1))$ with 
$$(a_{1},a_{2},\dots,a_{n-1},-(n-1))+(0,b_{1},b_{2},\dots,b_{n-2},-n)=-\rho(J).$$ 
Note that $a_{i},b_{j}\leq 0$. 

Set $\tau_{1}=s_{\epsilon_{n}}s_{n-1,\0,(1)}s_{n-2,\\0,(1)}\dots s_{1,\0,(1)}$ where $s_{\epsilon_{n}}(\epsilon_{j})=\epsilon_{j}$ for 
$j=1,2,\dots,n-1$ and $s_{\epsilon_{n}}(\epsilon_{n})=-\epsilon_{n}$. Even though $\tau_{1}$ is not in the Weyl group for 
type $D_{n}$, it can be shown that $\tau_{1}^{-1}w_{1}\circ 0=w_{1}^{\prime}\circ 0$ for some $w_{1}'\in \Sigma_{n}\ltimes ({\mathbb Z}_{2})^{n-1}$. 
By direct computation,  
$$
\tau_{1}^{-1}w_{1}\circ 0=(0, 1+a_{1},1+a_{2},\dots,1+a_{n-1}). 
$$ 
On the other hand, let $\tau_{2}=s_{n,\0,(2)}s_{n-1,\0,(2)}\dots s_{2,\0,(2)}\in \Sigma_{n}\ltimes ({\mathbb Z}_{2})^{n}$. One can 
verify that 
$$
\tau_{2}^{-1}w_{2}\circ 0=(0, 0, 1+b_{1}, 1+b_{2},\dots,1+b_{n-2}). 
$$

Next we need to show that 
$$\tau_{1}^{-1}w_{1}\circ 0+\tau_{2}^{-1}w_{2}\circ 0=
(0, 1+a_{1},1+a_{2},\dots,1+a_{n-1})+
(0, 0, 1+b_{1}, 1+b_{2},\dots,1+b_{n-2})=-\rho(J_{2}).$$ 
for some $J_{2}\subseteq \{\epsilon_{i}-\delta_{j}:\ 2\leq i \leq n,\ 2\leq j \leq n\}$.
From our assumption,  
$$(a_{1},a_{2},\dots,a_{n-1},-(n-1))+(0,b_{1},b_{2},\dots,b_{n-2},-n)=-\rho(J).$$ 
This implies that $\{\epsilon_{i}+\delta_{n}:\ 1\leq i \leq n\}\cup \{\epsilon_{n}+\delta_{j}:\ 2\leq j \leq n-1\}\subseteq J$, 
and 
$$(a_{1}+1,a_{2}+1,\dots,a_{n-1}+1,0)+(0,b_{1}+1,b_{2}+1,\dots,b_{n-2}+1,0)=-\rho(J_{1}).$$ 
for some $J_{1}\subseteq \{\epsilon_{i}-\delta_{j}:\ 1\leq i \leq n-1,\ 1\leq j \leq n-1\}$. This claim now 
follows by applying a permutation of the coordinates. We can now conclude by the induction hypothesis that 
$$\tau_{1}^{-1}w_{1}\circ 0+\tau_{2}^{-1}w_{2}\circ 0=\tilde{w}_{1}\circ 0+\tilde{w}_{2}\circ 0=-\rho(J_{2}).$$
for unique $\tilde{w}_{1},\tilde{w}_{2}\in \Sigma_{n-1}\ltimes ({\mathbb Z}_{2})^{n-2}$. Moreover, one can verify that 
$l(\tau_{j}\tilde{w}_{j})=n+l(\tilde{w}_{j})$ for $j=1,2$. Consequently, in case (ii), we are reduced to the 
$2^{(n-1)}$ solutions of $w_{1}\circ 0+ w_{2}\circ 0=-\rho(J)$ in $\mathfrak{osp}(2(n-1)|2(n-1))$ by multiplying by 
$\tau^{-1}=\tau_{1}^{-1}\tau_{2}^{-1}$ whose total length is $2n$. This proves (\ref{eq:Poincarequotient}). 

\subsection{${\mathfrak g}=\mathfrak{osp}(2(n+1)|2n)$ for $n\geq 1$} 

Consider the embedding $\mathfrak{osp}(2n|2n)\hookrightarrow \mathfrak{osp}(2(n+1)|2n)$, and let 
${\mathfrak p}$ be the parabolic subalgebra generated by $\mathfrak{osp}(2n|2n)$ and the root vectors 
with weights of the form $-\epsilon_{1}\pm \delta_{j}$ for $j=1,2,\dots,n$. Let $P$ be the parabolic subgroup scheme with $\text{Lie }P={\mathfrak p}$. 
One has $B\subseteq P \subseteq G$. 

In this case, we have $W_{\1}\cong \Sigma_{n}\ltimes ({\mathbb Z}_{2})^{n}$ for ${\mathfrak g}=\mathfrak{osp}(2(n+1)|2n)$. 
In order to show that $p_{G,B}(t)=p_{W_{\1}}(t^{2})$, we use Proposition~\ref{P:Poincarespectral} to 
reduce our computation to proving that $p_{G,P}(t)=1+t^{2n}$. Here we are using information about the Poincar\'e series 
for $\mathfrak{osp}(2n|2n)$. 

The weights of $({\mathfrak g}_{\1}/{\mathfrak p}_{\1})^{*}$ are $-\epsilon_{1}\pm \delta_{j}$ for $j=1,2,\dots,n$. 
Suppose we have a weight of the form $w_{1}\circ \lambda+w_{2}\circ \mu =-\rho(J)$ where $J\subseteq 
\{-\epsilon_{1}\pm \delta_{j}:\ j=1,2,\dots,n\}$. Then $w_{1}\circ \lambda=-k\epsilon_{1}=-k\omega_{1}$. Therefore, 
$\langle w_{1}\circ \lambda,\alpha^{\vee} \rangle =0$ for $\alpha\in \Delta_{\0,(1)}$. It follows that 
$\langle \lambda+\rho_{\0}^{(1)},w_{1}^{-1}\alpha^{\vee} \rangle =1$. Consequently, $w_{1}^{-1}\alpha \geq 0$ and 
$w_{1}^{-1}\alpha\in  \Delta_{\0,(1)}$. We have $w_{1}^{-1}\{\alpha_{2},\dots,\alpha_{n+1}\}\subseteq \{\alpha_{2},\dots,\alpha_{n+1}\}$, thus 
$\langle \lambda,\alpha_{j}\rangle =0$ for $j=2,3,\dots,n+1$. 

Now consider $w_{1}\circ s\omega_{1}=-k\omega_{1}$ or equivalently, 
\begin{equation*} 
w_{1}\circ s\epsilon_{1}=-k\epsilon_{1}.
\end{equation*}
A direct computation using the dot action for $D_{n+1}$ shows that there are two solutions: 
(i) $w_{1}=1$, $s=0$, $k=0$ and (ii) $s=0$, $k=2n$ and $l(w_{1})=2n$. This proves the assertion. 

\subsection{} \label{SS:TypeAosp} We now extend our computation for $p_{G,B}(t)$ when ${\mathfrak g}=A(p|q)$ and when $\mathfrak g=\mathfrak{osp}(p|q)$. 
We consider the following embeddings of Lie superalgebras ${\mathfrak g}^{\prime}\subseteq {\mathfrak g}$ and 
set of positive roots $\Phi_{\1,P}^{+}$. Set $n\leq m-1$. 
\renewcommand{\arraystretch}{1.2}
\begin{table}[htp]
\caption{Embeddings}
\begin{center}
\begin{tabular}{ccc}\label{T:Embeddings}
${\mathfrak g}^{\prime}$ & ${\mathfrak g}$  &  $\Phi_{\1,P}^{+}$ \ \ \ \ \ \ \ [$n\leq m-1$]  \\ \hline
$A(n|m-1)$ & $A(n|m)$  &   $\{\epsilon_{i}-\delta_{m}:\  1\leq i \leq n \}$  \\
$\mathfrak{osp}(2n+1|2(m-1))$ & $\mathfrak{osp}(2n+1|2m)$ & $\{-\epsilon_{i}+\delta_{1}:\ 2\leq i \leq n  \}\ \cup \{\epsilon_{i}+\delta_{1}:\ 1\leq i \leq n \}\cup \{\delta_{1}\}$   \\ 
$\mathfrak{osp}(2(m-1)+1|2n)$ & $\mathfrak{osp}(2m+1|2n)$ & $\{\epsilon_{1}-\delta_{i}:\ 1\leq i \leq n \}\cup 
 \{\epsilon_{1}+\delta_{i}:\ 1\leq i \leq n \}$       \\
$\mathfrak{osp}(2n|2(m-1))$ & $\mathfrak{osp}(2n|2m)$ &  $\{-\epsilon_{i}+\delta_{1}:\ 1\leq i\leq n \}\cup \{\epsilon_{i}+\delta_{1}:\ 1\leq i \leq n\} $   \\
$\mathfrak{osp}(2(m-1)|2n)$   & $\mathfrak{osp}(2m|2n)$ &  $\{\epsilon_{1}-\delta_{i}:\ 2\leq i \leq n\}\cup \{\epsilon_{1}+\delta_{i}:\ 1\leq i \leq  n\}$                \\
\end{tabular}
\end{center}
\label{T:sums}
\end{table}

Let ${\mathfrak p}$ be the subalgebra generated by ${\mathfrak b}$ and ${\mathfrak g}^{\prime}$ and $P$ be the corresponding parabolic subgroup scheme 
with ${\mathfrak p}=\text{Lie }P$. 
\begin{theorem} Let ${\mathfrak g}^{\prime}=\operatorname{Lie }G^{\prime}$ and ${\mathfrak g}=\operatorname{Lie }G$ be as in Table~\ref{T:Embeddings}. Then 
$p_{G,B}(t)=p_{G^{\prime},B^{\prime}}(t)$. 
\end{theorem} 

\begin{proof} One has $B \subseteq  P  \subseteq G$. We prove the theorem by induction on $m$. One has $R^{\bullet}\text{ind}_{B}^{P} {\mathbb C}$ as a $G^{\prime}$-module identifies with 
$R^{\bullet}\text{ind}_{B^{\prime}}^{G^{\prime}} {\mathbb C}$ (cf. \cite[I. 6.14(1)]{Jan}). The cohomology in odd degree vanishes and cohomology in  even degree is isomorphic to a direct sum of trivial modules. 
Therefore, by Proposition~\ref{P:Poincarespectral} it suffices to show that $p_{G,P}(t)=1$ to prove that $p_{G,B}(t)=p_{G^{\prime},B^{\prime}}(t)$. 

One has 
$$R^{j}\text{ind}_{P}^{G}{\mathbb C}|_{G_{\0}}\cong R^{j}\text{ind}_{P_{\0}}^{G_{\0}} \Lambda^{\bullet}(({\mathfrak g}_{\1}/{\mathfrak p}_{\1})^{*}) \cong 
R^{j}\text{ind}_{B_{\0}}^{G_{\0}} \Lambda^{\bullet}(({\mathfrak g}_{\1}/{\mathfrak p}_{\1})^{*}).$$
The last isomorphism follows by using the spectral sequence relating the composition of induction functors \cite[I. 4.5(c) Proposition]{Jan} and the fact that $R^{t}\text{ind}_{B_{\0}}^{P_{\0}}{\mathbb C}=0$ for $t>0$. 

The weights of ${\mathfrak g}_{\1}/{\mathfrak p}_{\1}$ coincide with $\Phi_{\1,P}^{+}$. 
Let $\sigma$ be a weight of $\Lambda^{\bullet}(({\mathfrak g}_{\1}/{\mathfrak p}_{\1})^{*})$ with 
$R^{j}\text{ind}_{B_{\0}}^{G_{\0}} \sigma\neq 0$ for some $j\geq 0$. 

If ${\mathfrak g}=A(n|m)$ then by the argument given in Section~\ref{SS:TypeABC}(1), $\sigma=w_{1}\circ \lambda_{1}+w_{2}\circ 0$. 
It follows that $w_{2}\circ 0=c\ \delta_{m}$ where $c\geq 0$. This can only happen if $c=0$, which implies that $\sigma=0$. This proves the statement 
of the theorem for $A(n|m)$. 

In the other cases, the arguments given in Sections~\ref{SS:TypeABC}(1), and ~\ref{S:ospeven} show that $\sigma=w_{1}\circ 0+w_{2}\circ 0$. 
Consider the second case in Table~\ref{T:sums}. Then $w_{2}\circ 0=-c\delta_{1}=-c\omega_{1}$. In the root system $C_{m}$ this means that 
$c=0$ or $c=2m$. However, $c\leq 2n\leq 2(m-1)<2m$ which implies that $c=0$, thus $\sigma=0$. The other three cases in the table are handled with 
a similar argument. 
\end{proof}

\subsection{The computation of $R^{\bullet}\text{ind}_{B}^{G}{\mathbb C}$ and $p_{G,B}(t)$} The following theorem relates the sheaf theoretic Poincar\'e polynomial 
with the Poincar\'e polynomial for $W_{\1}$ when ${\mathfrak g}$ is not of type $P$. 

\begin{theorem} \label{T:Poincareseriesequal} Let ${\mathfrak g}$ be a classical simple Lie superalgebra with ${\mathfrak g}=\operatorname{Lie }G$. 
Assume that ${\mathfrak g}$ is not isomorphic to $P(n)$. 
Let $B$ be the parabolic subgroup such that ${\mathfrak b}=\operatorname{Lie }B$ where ${\mathfrak b}$ is the parabolic subalgebra 
defined in Table~\ref{T:BBW-roots}. Then 
\begin{itemize} 
\item[(a)] $R^{\bullet}\operatorname{ind}_{B}^{G}{\mathbb C}$ is a direct sum of trivial modules. 
\item[(b)] The number of trivial modules in $R^{n}\operatorname{ind}_{B}^{G}{\mathbb C}$ is given by 
$$p_{G,B}(t)=z_{{\mathfrak b},{\mathfrak g}}(t)=p_{W_{\1}}(s)$$ 
where $s$ is the parameter defined in Table~\ref{T:Poincareseries}. 
\end{itemize} 
\end{theorem} 

\begin{proof} Parts (a) and (b) were proved for the various classical Lie superalgebras in the following way. First, it was established that 
$R^{\bullet}\operatorname{ind}_{B_{\0}}^{G_{\0}} \Lambda^{\bullet}(({\mathfrak g}_{\1}/{\mathfrak b}_{\1})^{*})$ is a direct sum of trivial modules. 
Now from Propositions ~\ref{P:Gzeroiso} and ~\ref{P:selfext}, it follows that $R^{\bullet}\operatorname{ind}_{B}^{G}{\mathbb C}$ is a direct 
sum of trivial modules. Part (b) was verified along the way via the calculation of $R^{\bullet}\operatorname{ind}_{B_{\0}}^{G_{\0}} \Lambda^{\bullet}(({\mathfrak g}_{\1}/{\mathfrak b}_{\1})^{*})$. 

For the statements of (a) and (b) the cases when ${\mathfrak g}=D(2,1,\alpha)$, $G(3)$ and $F(4)$ were proved in Theorem~\ref{T:exceptionalsheaf}. 
For type $Q$ the statements was verified in Example~\ref{E:TypeQ}, and for the type $A$ families and orthosymplectic Lie superalgebras in 
Section~\ref{SS:TypeAosp}. 

\end{proof}

\subsection{} The preceding theorem motivates the following definition. 

\begin{definition} Let $G$ be an algebraic supergroup where ${\mathfrak g}=\operatorname{Lie }G$ is a classical simple Lie superalgebra 
and $B$ be a parabolic subgroup with ${\mathfrak b}=\operatorname{Lie }B$ such that 
\begin{itemize} 
\item[(a)] ${\mathfrak b}={\mathfrak b}_{\0}\oplus {\mathfrak b}_{\1}$ where ${\mathfrak b}_{\1}\cong {\mathfrak f}_{\1} 
\oplus {\mathfrak u}_{\1}$ where ${\mathfrak b}_{\0}$ is a Borel subalgebra for ${\mathfrak g}_{\0}$. 
\item[(b)] There exists a finite reflection group $W_{\1}$ such that as graded vector spaces, 
$$\operatorname{H}^{\bullet}({\mathfrak b},{\mathfrak b}_{\0},{\mathbb C})\cong 
\operatorname{H}^{\bullet}({\mathfrak g},{\mathfrak g}_{\0},{\mathbb C})\otimes {\mathbb C}[W_{\1}]_{\bullet}.$$ 
\end{itemize} 
Then ${\mathfrak b}$ is called an {\em BBW parabolic subalgebra} if and only if 
$$p_{G,B}(t)=z_{{\mathfrak b},{\mathfrak g}}(t)=p_{W_{\1}}(s)$$ 
where $s=t^{r}$ for some $r\geq 1$. 
\end{definition}

\section{Results for the Lie superalgebra ${\mathfrak p}(n)$} 

\subsection{} In this section we will present results for the Lie superalgebra ${\mathfrak p}(n)$ and explain how the theory 
differs from the other classical simple Lie superalgebras. Let  ${\mathfrak g}$ be the Lie superalgebra ${\mathfrak p}(n)$ where $n\geq 2$. 
This Lie superalgebra embeds into $\mathfrak{gl}(n|n)$ as $2n \times 2n$ matrices of the form
\begin{equation}\label{E:Pmatrix}\left( 
\begin{array}{c|c}
A&B\\\hline
C&-A^t
\end{array}
\right),
\end{equation}
where $A, B$ and $C$ are $n \times n$ matrices over ${\mathbb C}$ with $A\in \mathfrak{sl}_{n}({\mathbb C})$, $B$ symmetric, and $C$ skew-symmetric.  

Let $V$ be the $n$-dimensional natural representation for $\mathfrak{sl}_{n}({\mathbb C})$ with weights $\epsilon_{j}$, $j=1,2,\dots,n$. 
One has 
$${\mathfrak g}_{\0 } \cong \mathfrak{sl}_{n}({\mathbb C}) \text{ and } {\mathfrak g}_{\1} \cong S^2(V) \oplus \Lambda^{2}(V^{*}).$$ 
The weights of ${\mathfrak g}_{\1}$ are given by 
$$\Phi_{\1}=\{\epsilon_{i}+\epsilon_{j}:\ 1\leq i\leq j \leq n\}\cup \{ -\epsilon_{i}-\epsilon_{j}:\ 1\leq i < j \leq n\}.$$
The Lie superalgebra $\widetilde{{\mathfrak p}}(n)$, which is an enlargement of ${\mathfrak p}(n)$, is constructed by taking ${\mathfrak g}_{\0 } \cong \mathfrak{gl}_{n}({\mathbb C})$. 

\subsection{Cohomology and Hilbert Series}\label{SS:BBWp(n)} For the sake of convenience, we will redefine the detecting subalgebra ${\mathfrak f}$ as follows. 
The vector space ${\mathfrak f}_{\1}$ is the span of the root vectors in ${\mathfrak g}_{\1}$ whose weights are of the form 
$$\Phi_{{\mathfrak f}_{\1}}=\begin{cases} 
\{\pm (\epsilon_{1+j}+\epsilon_{2l-j})\} & \text{for $j=0,1,\dots, l-1$,  $n=2l$} \\
\{\pm (\epsilon_{1+j}+\epsilon_{2l+1-j}), 2\epsilon_{l+1} \}  & \text{for $j=0,1,\dots, l-1$, $n=2l+1$}. \\
\end{cases} 
$$ 
Set ${\mathfrak f}_{\0}=[{\mathfrak f}_{\1},
{\mathfrak f}_{\1}]$ and ${\mathfrak f}={\mathfrak f}_{\0}\oplus {\mathfrak f}_{\1}$. 

In both cases when $n$ is even or odd, $H$ is a torus of dimension $l$ and $N/N_{\0}\cong \Sigma_{l}\ltimes ({\mathbb Z}_{2})^{l}$. One can define a parabolic 
subalgebra ${\mathfrak b}$ as follows. We have 
$$\Phi_{\1}=\{\epsilon_{i}+\epsilon_{j}:\ 1\leq i,j \leq n\}\cup \{-\epsilon_{i}-\epsilon_{j}:\ 1\leq i<j \leq n\}.$$ 
Set 
$$\Phi_{\1}^{-}= \{\epsilon_{i}+\epsilon_{j}:\ n+1< i+j \}\cup \{-\epsilon_{i}-\epsilon_{j}:\ i<j,\ i+j<n+1\}$$ 
and ${\mathfrak b}$ be the parabolic subalgebra generated by the root vectors with roots in 
$\Phi_{\0}^{-}\cup \Phi_{{\mathfrak f}_{\1}} \cup \Phi_{\1}^{-}$ and ${\mathfrak t}_{\0}$. 
The defining hyperplanes for the parabolic are given by 
$$
{\mathcal H}=\begin{cases}  \sum_{i=1}^l x_i (E_i - E_{2l+1-i}),\  x_1>x_2>\cdots > x_l >0   & \text{$n=2l$} \\
\sum_{i=1}^l x_i (E_i - E_{2l+2-i})$,  $x_1>x_2>\cdots > x_l >0  & \text{$n=2l+1$}.         \\       
\end{cases}   
$$

The computation of $\operatorname{H}^{\bullet}({\mathfrak b},{\mathfrak b}_{\0},
{\mathbb C})=S^{\bullet}({\mathfrak f}_{\1}^{*})^{T_{\0}}$ is given in the table below. 

\renewcommand{\arraystretch}{1.2}
\begin{table}[htp]
\caption{Cohomology and Hilbert Series}
\begin{center}
\begin{tabular}{ccc}
${\mathfrak g}$ &  $W_{\1}$ & $\operatorname{H}^{\bullet}({\mathfrak b},{\mathfrak b}_{\0},{\mathbb C})$ \\ \hline
$\mathfrak{p}(n)$, $n=2l$ & $ \Sigma_{l}\ltimes ({\mathbb Z}_{2})^{l}$ &  ${\mathbb C}[x_{1}y_{1},x_{2}y_{2},\dots,x_{l}y_{l},x_{1}x_{2}\dots x_{l},y_{1}y_{2}\dots y_{l}]$       \\                                                   
$\mathfrak{p}(n)$, $n=2l+1$ & $ \Sigma_{l}\ltimes ({\mathbb Z}_{2})^{l}$ &  ${\mathbb C}[x_{1}y_{1},x_{2}y_{2},\dots,x_{l}y_{l},x_{1}^{2}x_{2}^{2}\dots x_{l}^{2}x_{l+1}]$        \\                                                   
\end{tabular}
\end{center}
\label{T:dimensions}
\end{table}

The goal for the remainder of this section is to compute $z_{{\mathfrak b},{\mathfrak g}}(t)$ when $n$ is even and when $n$ is odd. 

In the case when 
$n=2l$ is even, set 
$$S={\mathbb C}[x_{1}y_{1},x_{2}y_{2},\dots,x_{l}y_{l},x_{1}x_{2}\dots x_{l},y_{1}y_{2}\dots y_{l}]$$ 
and 
$$T={\mathbb C}[f_{1},f_{2},\dots,f_{l-1},x_{1}x_{2}\dots x_{l},y_{1}y_{2}\dots y_{l}]$$ 
where $f_{j}$ is the $jth$ symmetric polynomial in $\{x_{1}y_{1},x_{2}y_{2},\dots,x_{l}y_{l}\}$. 
Then as in the case for $\mathfrak{sl}(l | l)$, $S$ is free $T$-module of rank $|\Sigma_{l}|$. Furthermore, if $p_{S}(t)$ (resp. $p_{T}(t)$) are the Poincar\'e polynomials of 
$S$ (resp. $T$) then 
\begin{equation}
p_{\Sigma_l}(t^{2})=p_{S}(t)/p_{T}(t)=p_{{\mathfrak b}}(t)/p_{T}(t).
\end{equation} 
Now we use the fact that $T$ is a polynomial algebra generated in degrees $2,4,\dots, 2l-2$, $l$ and $l$. 
Therefore, 
\begin{equation} 
p_{\mathfrak b}(t)=p_{\Sigma_{l}}(t^2)\cdot p_{T}(t)=\frac{(1-t^{2l})}{(1-t^{2})^{l}(1-t^{l})^{2}}. 
\end{equation} 

From \cite[Table 1]{BKN1}
$\text{H}^{\bullet}({\mathfrak g},{\mathfrak g}_{\0},{\mathbb C})$ is a polynomial algebra generated in degrees 
$4,8,\dots,4(l-1)$, $l$, and $n$. Consequently, 
\begin{equation}\label{eq:z-neven}
z_{{\mathfrak b},{\mathfrak g}}(t)=\frac{(1-t^{4})(1-t^{8})\dots (1-t^{4(l-1)})(1-t^{n})(1+t^{l})}{(1-t^{2})^{l}}=p_{W_{\1}'}(t^2)p_{\Z_2}(t^l)
\end{equation} 
where $W_{\1}'=\Sigma_l\ltimes(\Z_2)^{l-1} $.

In the case when $n=2l+1$ is odd, $\operatorname{H}^{\bullet}({\mathfrak b},{\mathfrak b}_{\0},{\mathbb C})$ 
is a polynomial algebra with $l$ generators in degree $2$ and one generator in degree $n=2l+1$. On the other hand, 
$\operatorname{H}^{\bullet}({\mathfrak g},{\mathfrak g}_{\0},{\mathbb C})$ is a polynomial algebra with generators in 
degrees $4,8,\dots,4l$ and $n$ (cf. \cite[Table 1]{BKN1}). Therefore, for $n$ odd, 
\begin{equation}\label{eq:z-nodd}
z_{{\mathfrak b},{\mathfrak g}}(t)=\frac{(1-t^{4})(1-t^{8})\dots (1-t^{4l})}{(1-t^{2})^{l}}=p_{W_{\1}}(t^2).
\end{equation} 
Note that after cancellation by the factors $(1-t^{2})^{l}$ in (\ref{eq:z-neven}) and (\ref{eq:z-nodd}), one obtains that 
$$z_{{\mathfrak b},{\mathfrak g}}(1)=2^{l}\cdot (l)!=|W_{\1}|.$$

\subsection{${\mathfrak p}(2)$ and ${\mathfrak p}(3)$} First let ${\mathfrak g}={\mathfrak p}(2)$. 
Then $\Phi_{\1}=\{2\epsilon_{2}\}=\{-\alpha\}$ where $\alpha$ is the positive root in ${\mathfrak g}_{\0}=\mathfrak{sl}_{2}$. 
Therefore, one sees that 
$$
R^{j}\text{ind}_{B_{\0}}^{G_{\0}}\Lambda^{\bullet}(({\mathfrak g}_{\1}/{\mathfrak b}_{\1})^{*})=\begin{cases} 
{\mathbb C} & \text{$j=0,1$}\\ 
0  & \text{else}. \\
\end{cases} 
$$
It follows that 
\begin{equation}
z_{{\mathfrak b},{\mathfrak g}}(t)=\frac{(1-t^{2})(1+t)}{(1-t^{2})}=1+t=p_{W_{\1}}(t)=p_{G,B}(t)
\end{equation} 
and ${\mathfrak b}$ is a BBW parabolic subalgebra. 

Next, let ${\mathfrak g}={\mathfrak p}(3)$. It will be convenient to use the root basis and the fundamental weight basis 
for our calculations for $\Phi_{\0}=A_{2}$. One has 
$$\Phi_{\1}^{-}=\{2\epsilon_{3}, \epsilon_{2}+\epsilon_{3}, -\epsilon_{1}-\epsilon_{2}\}=\{2\omega_{2}, -\omega_{1}, -\omega_{2}\}.$$
Now $-\omega_{1}, -\omega_{2}\in \overline{C}_{\mathbb Z}$, and 
$$s_{\alpha_{2}}\cdot (-2\omega_{2})=-\omega_{1}\in \overline{C}_{\mathbb Z}-X(T_{\0})_{+}.$$ 
Therefore, $R^{\bullet}\text{ind}_{B_{\0}}^{G_{\0}} \Lambda^{1}(({\mathfrak g}_{\1}/{\mathfrak b}_{\1})^{*})=0$. 
Similarly, 
$$s_{\alpha_{1}}\cdot (\omega_{2}-3\omega_{1})=\omega_{1}-\omega_{2}\in \overline{C}_{\mathbb Z}-X(T_{\0})_{+},$$ 
thus, $R^{\bullet}\text{ind}_{B_{\0}}^{G_{\0}} \Lambda^{3}(({\mathfrak g}_{\1}/{\mathfrak b}_{\1})^{*})=0$. 

The weights of $\Lambda^{2}(({\mathfrak g}_{\1}/{\mathfrak b}_{\1})^{*})$ are $\{2\omega_{2}-\omega_{1}, -3\omega_{2}, -\omega_{1}-\omega_{2}\}$. 
The weight $-2\omega_{2}-\omega_{1}$ is conjugate to $-\omega_{2}$ by $s_{\alpha_{1}}s_{\alpha_{2}}$, and $-\omega_{1}-\omega_{2}
\in \overline{C}_{\mathbb Z}-X(T_{\0})_{+}$. So these weights do not contribute to give any cohomology. On the other hand, 
\begin{equation} 
(s_{\alpha_{2}}s_{\alpha_{1}})\cdot 0=-3\omega_{2}.  
\end{equation} 
Consequently, $R^{\bullet}\text{ind}_{B_{\0}}^{G_{\0}} \Lambda^{2}(({\mathfrak g}_{\1}/{\mathfrak b}_{\1})^{*})={\mathbb C}$. 

In summary, one has 
\begin{equation}
z_{{\mathfrak b},{\mathfrak g}}(t)=1+t^{2}=p_{W_{\1}}(t^{2})=p_{G,B}(t)
\end{equation} 
and ${\mathfrak b}$ is again a BBW parabolic subalgebra. 

\subsection{${\mathfrak p}(4)$} Next consider the Lie superalgebra ${\mathfrak g}=\mathfrak{p}(4)$. One has 
\begin{eqnarray*} 
\Phi_{\1}^{-}&=&\{2\epsilon_{4},2\epsilon_{3},\epsilon_{3}+\epsilon_{4},\epsilon_{2}+\epsilon_{4},-\epsilon_{1}-\epsilon_{2},-\epsilon_{1}-\epsilon_{2}\}\\
&=&\{-2\omega_{3},-2\omega_{2}+2\omega_{3},-\omega_{2},-\omega_{1}+\omega_{2}-\omega_{3},-\omega_{2},-\omega_{1}+\omega_{2}-\omega_{3}\}. 
\end{eqnarray*} 
It is useful to express the elements in $\Phi_{\1}^{-}$ in terms of fundamental weights of ${\mathfrak g}_{\0}=\mathfrak{sl}_{n}$. Note that $-\omega_{2}$ and 
$-\omega_{1}+\omega_{2}-\omega_{3}$ occur with multiplicity two. 

The weights of $\Lambda^{1}(({\mathfrak g}_{\1}/{\mathfrak b}_{\1})^{*})$ are precisely the ones in $\Phi_{\1}^{-}$. All of these weights are conjugate to a weight in 
$\overline{C}_{\mathbb Z}-X(T_{\0})_{+}$, thus $R^{\bullet}\text{ind}_{B_{\0}}^{G_{\0}} \Lambda^{1}(({\mathfrak g}_{\1}/{\mathfrak b}_{\1})^{*})=0$. 
Next observe that $\Lambda^{6}(({\mathfrak g}_{\1}/{\mathfrak b}_{\1})^{*})$ is one-dimensional and spanned by a vector of weight $-2\rho_{\0}$. 
Therefore, $\Lambda^{5}(({\mathfrak g}_{\1}/{\mathfrak b}_{\1})^{*})\cong \Lambda^{1}(({\mathfrak g}_{\1}/{\mathfrak b}_{\1})^{*})^{*}\otimes (-2\rho_{\0})$. 
Let $\mu=-\lambda-2\rho_{\0}$ be a weight of $\Lambda^{5}(({\mathfrak g}_{\1}/{\mathfrak b}_{\1})^{*})$ where $\lambda$ is a weight of  $\Lambda^{1}(({\mathfrak g}_{\1}/{\mathfrak b}_{\1})^{*})$. 
Then 
$$w_{0}\cdot \mu=w_{0}(-\lambda)=-w_{0}\lambda.$$ 
The possible weights of the form $w_{0}\cdot \mu$ are $\{-2\omega_{1},2\omega_{1}-2\omega_{2},-\omega_{2},-\omega_{1}+\omega_{2}+\omega_{3}\}$ 
which are all conjugate to a weight in $\overline{C}_{\mathbb Z}-X(T_{\0})_{+}$. Consequently, $R^{\bullet}\text{ind}_{B_{\0}}^{G_{\0}} \Lambda^{5}(({\mathfrak g}_{\1}/{\mathfrak b}_{\1})^{*})=0$. 

The distinct weights of $\Lambda^{3}(({\mathfrak g}_{\1}/{\mathfrak b}_{\1})^{*})$ are 
$$\{-3\omega_{2},-\omega_{1}-\omega_{2}-\omega_{3},-\omega_{1}-3\omega_{3},-2\omega_{2}-2\omega_{3}, -2\omega_{1}+2\omega_{2}--4\omega_{3},-\omega_{1}-2\omega_{2}+\omega_{1},-2\omega_{1}
-2\omega_{1}+\omega_{2}-2\omega_{3}\}.$$
A lengthy verification shows that all of the weights above are conjugate to a weight in $\overline{C}_{\mathbb Z}-X(T_{\0})_{+}$, thus 
$R^{\bullet}\text{ind}_{B_{\0}}^{G_{\0}} \Lambda^{3}(({\mathfrak g}_{\1}/{\mathfrak b}_{\1})^{*})=0$. 

The distinct weights in $\Lambda^{2}(({\mathfrak g}_{\1}/{\mathfrak b}_{\1})^{*})$ that are conjugate to a weight in $\overline{C}_{\mathbb Z}-X(T_{\0})_{+}$ are 
$$\{-2\omega_{2},-\omega_{2}-2\omega_{3},-\omega_{1}+\omega_{2}-3\omega_{3},-\omega_{1}-\omega_{2}+\omega_{3},-\omega_{1}-\omega_{3}\}.$$ 
For the other two weights: $-3\omega_{2}+2\omega_{3}$ (multiplicity 2), and $-2\omega_{1}+2\omega_{2}-2\omega_{3}$ (multiplicity 1), one has 
\begin{equation} 
(s_{\alpha_{1}}s_{\alpha_{2}})\cdot (-\omega_{2}-2\omega_{3})=0,
\end{equation} 
\begin{equation} 
(s_{\alpha_{1}}s_{\alpha_{3}})\cdot (-2\omega_{1}+2\omega_{2}-2\omega_{3})=0.
\end{equation} 
Consequently, $R^{j}\text{ind}_{B_{\0}}^{G_{\0}} \Lambda^{2}(({\mathfrak g}_{\1}/{\mathfrak b}_{\1})^{*})=0$ for $j\neq 2$ and 
$R^{2}\text{ind}_{B_{\0}}^{G_{\0}} \Lambda^{2}(({\mathfrak g}_{\1}/{\mathfrak b}_{\1})^{*})\cong {\mathbb C}^{\oplus 3}$. 
By using duality this also holds for $\Lambda^{4}(({\mathfrak g}_{\1}/{\mathfrak b}_{\1})^{*})$. 

Finally, $w_{0}(-2\rho_{\0})=0$ and $l(w_{0})=6$, thus 
$R^{j}\text{ind}_{B_{\0}}^{G_{\0}} \Lambda^{6}(({\mathfrak g}_{\1}/{\mathfrak b}_{\1})^{*})=0$ for $j\neq 6$ and 
$R^{6}\text{ind}_{B_{\0}}^{G_{\0}} \Lambda^{6}(({\mathfrak g}_{\1}/{\mathfrak b}_{\1})^{*})\cong {\mathbb C}$. By gathering all this information, one can now 
conclude that 
$$p_{G,B}(t)=1+3t^{2}+3t^{4}+t^{6}=(1+t^{2})^{3}=z_{\mathfrak b,\mathfrak g}(t).$$

For ${\mathfrak g}={\mathfrak p}(4)$, one has $W_{\1}=\Sigma_{2}\ltimes ({\mathbb Z}_{2})^{2}$. The Poincar\'e polynomial 
$$p_{W_{\1}}(t)=\frac{(1-t^{2})(1-t^{4})}{(1-t)^{2}}=(1+t)(1+t+t^{2}+t^{3}).$$ 
From this, it is clear that $p_{G,B}(t)\neq p_{W_{\1}}(t^{r})$ for any $r\geq 1$, and ${\mathfrak b}$ is not a BBW parabolic. 

\subsection{} Given our computations for ${\mathfrak g}=\mathfrak{p}(n)$, we conclude this section with two open questions about the parabolic subalgebra ${\mathfrak b}$. 
\vskip .25cm 
\noindent
(5.5.1) Does $p_{G,B}(t)=z_{{\mathfrak b},{\mathfrak g}}(t)$? 
\vskip .25cm 
\noindent 
(5.5.2) Is there a natural subset of elements in $\Sigma_{n}$ that describes the grading on ${\mathbb C}[W_{\1}]_{\bullet}$ given by 
$z_{{\mathfrak b},{\mathfrak g}}(t)$?


\section{Comparing cohomology and supports for $({\mathfrak g},{\mathfrak g}_{\0})$, $({\mathfrak b},{\mathfrak b}_{\0})$ and $({\mathfrak f},{\mathfrak f}_{\0})$} 

{\em In the section assume that ${\mathfrak g}$ is a classical Lie superalgebra, ${\mathfrak b}$ is the  BBW parabolic subalgebra and ${\mathfrak f}$ is the 
detecting subalgebra as defined in Section~\ref{SS:parabolicdef} .} 

\subsection{} By using the finite generation of the cohomology ring $\operatorname{H}^{\bullet}({\mathfrak b},{\mathfrak b}_{\0},{\mathbb C})$, 
one can define two types of support varieties. Let ${\mathcal V}_{({\mathfrak b},{\mathfrak b}_{\0})}(M)$ be the variety associated to the annihilator of $\operatorname{H}^{\bullet}({\mathfrak b},{\mathfrak b}_{\0},{\mathbb C})$ 
on $\text{Ext}^{\bullet}_{({\mathfrak b},{\mathfrak b}_{\0})}(M,M)$. One has an injection of $\operatorname{H}^{\bullet}({\mathfrak g},{\mathfrak g}_{\0},{\mathbb C})\hookrightarrow \operatorname{H}^{\bullet}({\mathfrak b},{\mathfrak b}_{\0},{\mathbb C})$ such that 
$\operatorname{H}^{\bullet}({\mathfrak b},{\mathfrak b}_{\0},{\mathbb C})$ is finitely generated over $\operatorname{H}^{\bullet}({\mathfrak g},{\mathfrak g}_{\0},{\mathbb C})$. Set 
$\widehat{{\mathcal V}}_{({\mathfrak b},{\mathfrak b}_{\0})}(M)$ to be the variety associated to the annihilator of $\operatorname{H}^{\bullet}({\mathfrak g},{\mathfrak g}_{\0},{\mathbb C})$ on $\text{Ext}^{\bullet}_{({\mathfrak b},{\mathfrak b}_{\0})}(M,M)$. 

The following theorem compares the support varieties for cohomology in  
$({\mathfrak b},{\mathfrak b}_{\0})$, $({\mathfrak t},{\mathfrak t}_{\0})$ and $({\mathfrak f},{\mathfrak f}_{\0})$.

\begin{theorem} \label{T:btsupports} Let $M$ be a finite-dimensional ${\mathfrak b}$-module. 
\begin{itemize} 
\item[(a)] ${\mathcal V}_{({\mathfrak b},{\mathfrak b}_{\0})}(M)\cong {\mathcal V}_{({\mathfrak t},{\mathfrak t}_{\0})}(M)$. 
\item[(b)] $\widehat{{\mathcal V}}_{({\mathfrak b},{\mathfrak b}_{\0})}(M)\cong {\mathcal V}_{({\mathfrak t},{\mathfrak t}_{\0})}(M)/N$. 
\item[(c)] ${\mathcal V}_{({\mathfrak t},{\mathfrak t}_{\0})}(M)\cong {\mathcal V}_{({\mathfrak f},{\mathfrak f}_{\0})}(M)/T_{\0}$. 
\item[(d)] ${\mathcal V}_{({\mathfrak t},{\mathfrak t}_{\0})}(M)/N\cong {\mathcal V}_{({\mathfrak f},{\mathfrak f}_{\0})}(M)/N$. 
\end{itemize}
\end{theorem} 

\begin{proof} (a) First observe that  by Theorem~\ref{t:b-tcoho}(b), the restriction map 
$\operatorname{H}^{\bullet}({\mathfrak b},{\mathfrak b}_{\0},{\mathbb C})\rightarrow 
\operatorname{H}^{\bullet}({\mathfrak t},{\mathfrak t}_{\0},{\mathbb C})$ is an isomorphism. 
Therefore, ${\mathcal V}_{({\mathfrak t},{\mathfrak t}_{\0})}(M)\subseteq {\mathcal V}_{({\mathfrak b},{\mathfrak b}_{\0})}(M)$ 
(cf. argument in \cite[Section 6.1]{BKN1}).  

Let $M=\oplus_{\lambda\in {\mathfrak t}_{\0}^{*}} M_{\lambda}$ be a weight space decomposition 
of $M$. Note that each $M_{\lambda}$ is a ${\mathfrak t}$-module. Next observe one can construct a ${\mathfrak b}$-stable filtration of $M$: 
$$M:=M_{0}\supseteq M_{1}\supseteq M_{2} \supseteq \dots \supseteq M_{s} \supseteq \{0\}$$ 
such that $M_{i}/M_{i+1}\cong M_{\lambda_{i}}$ for some $\lambda_{i}\in {\mathfrak t}_{\0}^{*}$. 

The filtration above provides a short exact sequence $0\rightarrow M_{s} \rightarrow M \rightarrow M/M_{s} \rightarrow 0$. One can then use the long exact sequence in cohomology to show that 
$${\mathcal V}_{({\mathfrak b},{\mathfrak b}_{\0})}(M,M^{\prime})\subseteq {\mathcal V}_{({\mathfrak b},{\mathfrak b}_{\0})}(M_{s},M^{\prime})\cup {\mathcal V}_{({\mathfrak b},{\mathfrak b}_{\0})}(M/M_{s},M^{\prime})$$ 
for all finite-dimensional ${\mathfrak b}$-modules $N$. Specializing $M=M^{\prime}$, one obtains 
$${\mathcal V}_{({\mathfrak b},{\mathfrak b}_{\0})}(M)\subseteq {\mathcal V}_{({\mathfrak b},{\mathfrak b}_{\0})}(M_{s})\cup {\mathcal V}_{({\mathfrak b},{\mathfrak b}_{\0})}(M/M_{s}).$$ 
Applying this procedure inductively yields 
\begin{equation} 
{\mathcal V}_{({\mathfrak b},{\mathfrak b}_{\0})}(M)\subseteq \bigcup_{\lambda\in t_{\0}^{*}} {\mathcal V}_{({\mathfrak b},{\mathfrak b}_{\0})}(M_{\lambda}).
\end{equation} 
Here $M_{\lambda}$ is regarded as ${\mathfrak b}$-module with trivial ${\mathfrak u}$-action. 

Next apply the LHS spectral sequence for $M_{\lambda}$: 
$$E_{2}^{i,j}=\operatorname{Ext}^{i}_{({\mathfrak t},{\mathfrak t}_{\0})}({\mathbb C},\operatorname{Ext}^{j}_{({\mathfrak u},{\mathfrak u}_{\0})}({\mathbb C},{\mathbb C})\otimes 
M_{\lambda}^{*}\otimes M_{\lambda})\Rightarrow \operatorname{Ext}^{i+j}_{({\mathfrak b},{\mathfrak b}_{\0})}(M_{\lambda},M_{\lambda}).$$ 
By using the identification of 
$$R:=\operatorname{H}^{\bullet}({\mathfrak b},{\mathfrak b}_{\0},{\mathbb C})\cong \operatorname{H}^{\bullet}({\mathfrak t},{\mathfrak t}_{\0},{\mathbb C})\cong S^{\bullet}({\mathfrak t}^{*}_{\1})^{T_{\0}}$$
one has that $R$ acts on the rows of $E_{2}$ and the abutment. It follows that 
$${\mathcal V}_{({\mathfrak b},{\mathfrak b}_{\0})}(M_{\lambda})\subseteq {\mathcal V}_{({\mathfrak t},{\mathfrak t}_{\0})}(M_{\lambda}).$$ 
Since ${\mathcal V}_{({\mathfrak t},{\mathfrak t}_{\0})}(M)=\cup_{\lambda\in {\mathfrak t}_{\0}^{*}} {\mathcal V}_{({\mathfrak t},{\mathfrak t}_{\0})}(M_{\lambda})$, 
one has ${\mathcal V}_{({\mathfrak b},{\mathfrak b}_{\0})}(M)\subseteq {\mathcal V}_{({\mathfrak t},{\mathfrak t}_{\0})}(M)$. 

(b) The result can be obtained using the argument given in (a) and replacing (i) $\operatorname{H}^{\bullet}({\mathfrak b},{\mathfrak b}_{\0},{\mathbb C})\rightarrow 
\operatorname{H}^{\bullet}({\mathfrak t},{\mathfrak t}_{\0},{\mathbb C})$ by $\operatorname{H}^{\bullet}({\mathfrak g},{\mathfrak g}_{\0}, {\mathbb C}) \rightarrow 
\operatorname{H}^{\bullet}({\mathfrak t},{\mathfrak t}_{\0},{\mathbb C})^{N}$, (ii) ${\mathcal V}_{({\mathfrak t},{\mathfrak t}_{\0})}(-)$ by 
${\mathcal V}_{({\mathfrak t},{\mathfrak t}_{\0})}(-)/N$, and (iii) ${\mathcal V}_{({\mathfrak b},{\mathfrak b}_{\0})}(-)$ by $\widehat{{\mathcal V}}_{({\mathfrak b},{\mathfrak b}_{\0})}(-)$. 

(c) We have ${\mathfrak f}\unlhd {\mathfrak t}$, so one can apply the Lyndon-Hochschild-Serre spectral sequence for relative cohomology 
$$E_{2}^{i,j}=\operatorname{H}^{i}({\mathfrak t}/{\mathfrak f},{\mathfrak t}_{\0}/{\mathfrak f}_{\0},\operatorname{H}^{j}({\mathfrak f},{\mathfrak f}_{\0},M^{\prime}))
\Rightarrow \operatorname{H}^{i+j}({\mathfrak t},{\mathfrak t}_{\0},M^{\prime})
$$ 
for any ${\mathfrak t}$-module $M^{\prime}$. The spectral sequence collapses (${\mathfrak t}/{\mathfrak t}_{\0}\cong {\mathfrak f}/{\mathfrak f}_{\0}$) and yields: 
\begin{equation} 
\operatorname{H}^{\bullet}({\mathfrak t},{\mathfrak t}_{\0},M^{\prime})\cong \operatorname{H}^{\bullet}({\mathfrak f},{\mathfrak f}_{\0},M^{\prime})^{T_{\0}}.
\end{equation}
This proves that the restriction map: $\operatorname{H}^{\bullet}({\mathfrak t},{\mathfrak t}_{\0},M^{\prime})\hookrightarrow \operatorname{H}^{\bullet}({\mathfrak f},{\mathfrak f}_{\0},M^{\prime})$ 
is an injective map, so by \cite[Theorem 4.4.1]{LNZ}, ${\mathcal V}_{({\mathfrak t},{\mathfrak t}_{\0})}(M)\cong {\mathcal V}_{({\mathfrak f},{\mathfrak f}_{\0})}(M)/T_{\0}$. 


(d) One can obtain this part by using (c) and taking quotients with $N$. \end{proof} 


\subsection{Geometric Induction and Spectral Sequences} Let $G$ (resp. $B$) be the supergroup (scheme) such that $\text{Lie }G={\mathfrak g}$ 
(resp. $\text{Lie }B={\mathfrak b}$).  If $M$ is a $G$-module (resp. $B$-module) then one can consider $M$ as a ${\mathfrak g}$-module (resp. ${\mathfrak b}$-module)
by differentiation. The following results provides a spectral sequence that relates the relative cohomology for ${\mathfrak g}$ and ${\mathfrak b}$ via the 
higher right derived functors of $\text{ind}_{B}^{G}(-)$. 

\begin{proposition}\label{P:spectralseq} Let $M_{1}$ be a $G$-module and $M_{2}$ be a $B$-module. Then there exists a first quadrant spectral sequence. 
$$E_{2}^{i,j}=\operatorname{Ext}^{i}_{({\mathfrak g},{\mathfrak g}_{\0})} (M_{1},R^{j}\operatorname{ind}_{B}^{G} M_{2})\Rightarrow 
\operatorname{Ext}_{({\mathfrak b},{\mathfrak b}_{\0})}^{i+j}(M_{1},M_{2}). $$ 
\end{proposition} 

\begin{proof} The spectral sequence is constructed via a composition of functors. Let ${\mathcal F}_{1}(-)=\text{Hom}_{({\mathfrak g},{\mathfrak g}_{\0})}(M,-)$ and 
${\mathcal F}_{2}(-)=\text{ind}_{B}^{G}(-)$. We are regarding ${\mathcal F}_{1}$ (resp. ${\mathcal F}_{2}$) on the relative category ${\mathcal C}_{({\mathfrak g},{\mathfrak g}_{\0})}$ 
(resp. ${\mathcal C}_{({\mathfrak b},{\mathfrak b}_{\0})}$) where the injective objects are relatively projective over $U({\mathfrak g}_{\0})$ (resp. $U({\mathfrak b}_{\0})$). 

The functors ${\mathcal F}_{1}$ and ${\mathcal F}_{2}$ are left exact. Furthermore, an injective object in ${\mathcal C}_{({\mathfrak b},{\mathfrak b}_{\0})}$ is 
a direct summand of $\text{ind}_{B_{\0}}^{B} N$ for some $B_{\0}$-module $N$. Observe that 
$${\mathcal F}_{2}(\text{ind}_{B_{\0}}^{B}N)\cong \text{ind}_{B}^{G}[\text{ind}_{B_{\0}}^{B} N]\cong \text{ind}_{B_{\0}}^{G}N=\text{ind}_{G_{\0}}^{G}[\text{ind}_{B_{\0}}^{G_{\0}} N].$$ 
Therefore, ${\mathcal F}_{2}(\text{ind}_{B_{\0}}^{B}N)$ is an injective module in ${\mathcal C}_{({\mathfrak g},{\mathfrak g}_{\0})}$. It follows 
that injective objects in  ${\mathcal C}_{({\mathfrak b},{\mathfrak b}_{\0})}$ are taken to objects acyclic for ${\mathcal F}_{1}$. 
Finally, observe that 
$${\mathcal F}_{1}\circ {\mathcal F}_{2}(-)=\text{Hom}_{({\mathfrak g},{\mathfrak g}_{\0})}(M,\text{ind}_{B}^{G}(-))\cong 
\text{Hom}_{({\mathfrak b},{\mathfrak b}_{\0})}(M,-).$$
The existence of the spectral sequence now follows by \cite[I. 4.1 Proposition]{Jan}. 
\end{proof}

\subsection{Restricting relative $U({\mathfrak g}_{\0})$-injectives to $U({\mathfrak b})$} We can use the 
spectral sequence to investigate what happens when an relative injective $U({\mathfrak g}_{\0})$-module restricts to ${\mathfrak b}$. 

\begin{theorem} Let $I$ be a ${\mathfrak g}$-module that is a relatively injective $U({\mathfrak g}_{\0})$-module and $M$ be any finite-dimensional ${\mathfrak g}$-module. Then 
\begin{itemize} 
\item[(a)] $\operatorname{Ext}^{j}_{({\mathfrak b},{\mathfrak b}_{\0})}(M,I)\cong \operatorname{Hom}_{({\mathfrak g},{\mathfrak g}_{\0})}(M,
[R^{j}\operatorname{ind}_{B}^{G}{\mathbb C}] \otimes I)$
\item[(b)] $\operatorname{Ext}^{j}_{({\mathfrak b},{\mathfrak b}_{\0})}(M,I)=0$ for $j>\dim G_{\0}/B_{\0}$. 
\end{itemize} 
\end{theorem} 

\begin{proof} One can apply the spectral sequence given in Proposition~\ref{P:spectralseq}:
\begin{equation} 
E_{2}^{i,j}=\text{Ext}^{i}_{({\mathfrak g},{\mathfrak g}_{\0})}(M,[R^{j}\operatorname{ind}_{B}^{G}{\mathbb C}] \otimes I)\Rightarrow 
\operatorname{Ext}^{i+j}_{({\mathfrak b},{\mathfrak b}_{\0})}(M,I).
\end{equation}

Since $I$ is injective the spectral sequence collapses and yields (a). For part (b), one has $R^{j}\operatorname{ind}_{B}^{G}{\mathbb C}=0$
for $j\geq \dim G_{\0}/B_{\0}$ by Proposition~\ref{P:Gzeroiso}. 
\end{proof}

The result above shows that $I$ restricted to ${\mathfrak b}$ need not be a relatively injective $U({\mathfrak b}_{\0})$-module. 
However the result does show that if $I$ is a relatively injective $U({\mathfrak g}_{\0})$-module then 
$\{0\}={\mathcal V}_{({\mathfrak g},{\mathfrak g}_{\0})}(I)={\mathcal V}_{({\mathfrak b},{\mathfrak b}_{\0})}(I)$. 

\subsection{Collapsing of the Spectral Sequence} The next result shows that the spectral sequence given in Proposition~\ref{P:spectralseq} collapses when 
$M={\mathbb C}$ and ${\mathfrak b}$ is a BBW parabolic subalgebra. 

\begin{theorem} \label{T:sscollapse} Let ${\mathfrak g}$ be a classical simple Lie superalgebra with ${\mathfrak g}\neq P(n)$. Then 
the following spectral sequence collapses:  
\begin{equation} 
E_{2}^{i,j}=\operatorname{Ext}^{i}_{({\mathfrak g},{\mathfrak g}_{\0})}({\mathbb C},[R^{j}\operatorname{ind}_{B}^{G}{\mathbb C}])\Rightarrow 
\operatorname{Ext}^{i+j}_{({\mathfrak b},{\mathfrak b}_{\0})}({\mathbb C},{\mathbb C}). 
\end{equation} 
\end{theorem} 

\begin{proof} It suffices to show that 
\begin{equation} \label{e:sscollapsecondition} 
\sum_{i+j=n} \dim E_{2}^{i,j}=\dim \operatorname{H}^{n}({\mathfrak b},{\mathfrak b}_{\0},{\mathbb C})
\end{equation} 
for all $n\geq 0$. This will insure that the differentials $d_{r}$ are zero for $r\geq 2$. 

Since $R^{j}\text{ind}_{B}^{G}{\mathbb C}\cong {\mathbb C}^{\oplus m_{j}}$ by Theorem~\ref{T:Poincareseriesequal}(a), one has 
\begin{equation} \label{e:e2} 
\bigoplus_{i+j=n} E_{2}^{i,j}\cong \bigoplus_{i+j=n} \operatorname{H}^{i}({\mathfrak g},{\mathfrak g}_{\0},{\mathbb C})\otimes R^{j}\text{ind}_{B}^{G}{\mathbb C}
 \end{equation} 
 for all $n\geq 0$. 

Now by Theorem~\ref{T:Poincareseriesequal}(b), $p_{\mathfrak b}(t)=p_{\mathfrak g}(t)\cdot p_{G,B}(t)$. Therefore, by comparing coefficients of $t^{n}$, one can conclude 
that (\ref{e:sscollapsecondition}) holds. 

\end{proof}

\subsection{} We can now give conditions via the collapsing of the spectral sequence in Proposition~\ref{P:spectralseq} for $M_{1}\cong {\mathbb C}$ and 
$M_{2}\cong {\mathbb C}$ to insure that $\widehat{\mathcal V}_{({\mathfrak b},{\mathfrak b}_{\0})}(M) \cong {\mathcal V}_{({\mathfrak g},{\mathfrak g}_{\0})}(M)$.

\begin{theorem} \label{T:gbsupports} Let $M$ a finite-dimensional 
${\mathfrak g}$-module. Suppose that 
\begin{itemize} 
\item[(a)] $R^{j}\operatorname{ind}_{B}^{G} {\mathbb C}\cong {\mathbb C}^{\oplus m_{j}}$ for $j>0$.
\item[(b)] The spectral sequence 
\begin{equation} 
E_{2}^{i,j}=\operatorname{Ext}^{i}_{({\mathfrak g},{\mathfrak g}_{\0})}({\mathbb C},[R^{j}\operatorname{ind}_{B}^{G}{\mathbb C}])\Rightarrow 
\operatorname{Ext}^{i+j}_{({\mathfrak b},{\mathfrak b}_{\0})}({\mathbb C},{\mathbb C})
\end{equation}
collapses and yields an isomorphism of 
$R=\operatorname{H}^{\bullet}({\mathfrak g},{\mathfrak g}_{\0},{\mathbb C})=S^{\bullet}({\mathfrak g}_{\1}^{*})^{G_{\0}}$-modules. 
\end{itemize} 
Then $\operatorname{res}^*:\widehat{{\mathcal V}}_{({\mathfrak b},{\mathfrak b}_{\0})}(M) \rightarrow {\mathcal V}_{({\mathfrak g},{\mathfrak g}_{\0})}(M)$ 
is an isomorphism.
\end{theorem} 

\begin{proof} Let $M$ be a finite-dimensional ${\mathfrak g}$-module. By assumption, $R^{j}\operatorname{ind}_{B}^{G} {\mathbb C}\cong {\mathbb C}^{\oplus m_{j}}$ for $j>0$. 
Using the tensor identity, $R^{j}\operatorname{ind}_{B}^{G} M\cong [R^{j}\operatorname{ind}_{B}^{G} {\mathbb C}]\otimes M$, one has two spectral sequences: 
\begin{equation} \label{ss:C}
E_{2}^{i,j}=\operatorname{Ext}^{i}_{({\mathfrak g},{\mathfrak g}_{\0})}({\mathbb C}, {\mathbb C}^{\oplus m_{j}})  \Rightarrow 
\operatorname{Ext}^{i+j}_{({\mathfrak b},{\mathfrak b}_{\0})}({\mathbb C},{\mathbb C}),
\end{equation}
\begin{equation} \label{ss:M}
\bar{E}_{2}^{i,j}=\operatorname{Ext}^{i}_{({\mathfrak g},{\mathfrak g}_{\0})}({\mathbb C}, [M^{*}\otimes M]^{\oplus m_{j}})  \Rightarrow 
\operatorname{Ext}^{i+j}_{({\mathfrak b},{\mathfrak b}_{\0})}(M,M).
\end{equation}
The spectral sequence (\ref{ss:C}) acts on (\ref{ss:M}) in the following way. There exists a natural map of ${\mathbb C}$-algebras 
$\rho:\operatorname{H}^{\bullet}({\mathfrak b},{\mathfrak b}_{\0},{\mathbb C})\rightarrow \operatorname{Ext}^{\bullet}_{({\mathfrak b},{\mathfrak b}_{\0})}(M,M)$ 
that is defined by taking an extension class in $\operatorname{H}^{\bullet}({\mathfrak b},{\mathfrak b}_{\0},{\mathbb C})$ and tensoring the class by $M$. 
Set $\widehat{J}_{{\mathfrak b},M}=\text{Ann}_{R} \operatorname{Ext}^{\bullet}_{({\mathfrak b},{\mathfrak b}_{\0})}(M,M)$. Then one has an 
injective ring homomorphism
\begin{equation} 
\rho:\operatorname{H}^{\bullet}({\mathfrak b},{\mathfrak b}_{\0},{\mathbb C})/ \widehat{J}_{{\mathfrak b},M}\hookrightarrow \operatorname{Ext}^{\bullet}_{({\mathfrak b},{\mathfrak b}_{\0})}(M,M).
 \end{equation} 
For $j\geq 0$, there also exist maps on the direct sum of algebras: 
\begin{equation} 
\rho_{j}: \operatorname{H}^{\bullet}({\mathfrak g},{\mathfrak g}_{\0},{\mathbb C})^{\oplus m_{j}}\rightarrow \operatorname{Ext}^{\bullet}_{({\mathfrak g},{\mathfrak g}_{\0})}(M,M)^{\oplus m_{j}}
\end{equation} 
with 
\begin{equation} 
\rho_{j}: [R/J_{M}]^{\oplus m_{j}}\hookrightarrow \operatorname{Ext}^{\bullet}_{({\mathfrak g},{\mathfrak g}_{\0})}(M,M)^{\oplus m_{j}}.
\end{equation} 
Furthermore, there is a compatibility of differentials: 
\begin{equation}
\rho_{j}(d_{r}(x))=\bar{d}_{r}(\rho_{j}(x)).
\end{equation} 
Since (\ref{ss:C}) collapses, $d_{r}(x)=0$ for $r\geq 2$, thus $\bar{d}_{r}(\rho_{j}(x))=0$ for $r\geq 2$, $j\geq 0$. Therefore, 
the differentials on $[R/J_{M}]^{\oplus m_{j}}$ in (\ref{ss:M}) are zero, and 
$\operatorname{Ext}^{\bullet}_{({\mathfrak b},{\mathfrak b}_{\0})}(M,M)$ contains a copy of the module 
$\oplus_{j\geq 0} [R/J_{M}]^{\oplus m_{j}}$. 

Now suppose that $y\in R$ annihilates $\operatorname{Ext}^{\bullet}_{({\mathfrak b},{\mathfrak b}_{\0})}(M,M)$. Then 
$y$ annihilates $R/J_{M}$ so $y\in J_{M}$. Consequently, $\text{Ann}_{R} \operatorname{Ext}^{\bullet}_{({\mathfrak b},{\mathfrak b}_{\0})}(M,M) \subseteq J_{M}$, 
and ${\mathcal V}_{({\mathfrak g},{\mathfrak g}_{\0})}(M)\subseteq \widehat{{\mathcal V}}_{({\mathfrak b},{\mathfrak b}_{\0})}(M)$. The other inclusion 
holds by looking at the action of $R$ on the spectral sequence (\ref{ss:M}) [e.g., if $R$ annihilates $\bar{E}_{2}$, then it annihilates the abutment]. 
Hence, ${\mathcal V}_{({\mathfrak g},{\mathfrak g}_{\0})}(M)= \widehat{{\mathcal V}}_{({\mathfrak b},{\mathfrak b}_{\0})} (M)$. 
\end{proof} 

\subsection{Proof of Theorem~\ref{T:isosupports} } For ${\mathfrak g}={\mathfrak p}(n)$, the first isomorphism in Theorem~\ref{T:isosupports}(b) can be deduced from 
\cite[Theorem 5.1.1(a)]{LNZ} since $P(n)$ is type I. Now assume that ${\mathfrak g}\neq P(n)$, then the first isomorphism in Theorem~\ref{T:isosupports}(b) follows from Theorem~\ref{T:isosupportsfe}. 
Therefore, it suffices to prove that $\operatorname{res}^{*}:{\mathcal V}_{(\f,\f_{\0})}(M)/N \rightarrow {\mathcal V}_{(\g,\g_{\0})}(M)$ is an isomorphism. 
From Theorem~\ref{T:gbsupports} $\operatorname{res}^{*}:\widehat{{\mathcal V}}_{({\mathfrak b},{\mathfrak b}_{\0})}(M) \rightarrow {\mathcal V}_{(\g,\g_{\0})}(M)$ is an isomorphism. 
The statement of the theorem now follows by applying Theorem~\ref{T:btsupports}(b)(d).

\section{Tables for BBW parabolics and Poincar\'e series} 

\subsection{BBW Parabolics}\label{SS:defineBBW} The following tables provide a reference for the construction of BBW parabolic 
subalgebras. In these tables, the roots for the detecting subalgebras and the BBW parabolics are given as well 
as the defining hyperplanes. One should note that although ${\Phi}_{{\mathfrak f}_{\1}} = \emptyset$ for ${\mathfrak g} = {\mathfrak q}(n), \mathfrak{psq}(n)$ in  Table \ref{T:detectingsub}, the algebra ${\mathfrak f}_{\1}$ is not trivial, and it equals the odd part of the Cartan subalgebra of ${\mathfrak g}$. Also, for convenience, the obvious restrictions for the indexes are not included in some cases. For example, in Table \ref{T:BBW-roots}, the restrictions of the indexes for ${\Phi}_{{\mathfrak f}_{\1}}$ when ${\mathfrak g} = \mathfrak{osp}(2m|2n)$ should be $i < j$, $1 \leq i, k \leq m$, $1 \leq j, \ell \leq n$, but we just write  $i < j$. 

\bigskip 
\renewcommand{\arraystretch}{1.2}
\begin{table}[htp]
\caption{Roots for the detecting subalgebras}
\begin{center}
\begin{tabular}{cc}
${\mathfrak g}$ &  ${\Phi}_{{\mathfrak f}_{\1}}$ \\ \hline
$\mathfrak{gl}(m|n), \mathfrak{sl}(m|n)$ [$m\leq n$] &     $\{ \pm\ ( \epsilon_i - \delta_i ) \; | \; 1 \leq i \leq m \}$           \\   
$\mathfrak{osp}(2m|2n)$ &          $\{  \pm\ (  \epsilon_i - \delta_i )\; | \; 1 \leq i \leq \min (m,n) \}$                                                                    \\  
 $\mathfrak{osp}(2m+1|2n)$  [$m\geq n$]  &         $\{  \pm\ (  \epsilon_i - \delta_i )\; | \; 1 \leq i \leq n \}$                                          \\  
  $\mathfrak{osp}(2m+1|2n)$  [$m <  n$]  &     $\{  \pm\ (  \epsilon_i - \delta_i ) \; | \; 1 \leq i \leq m \}$          \\                                                                     
${\mathfrak q}(n), \mathfrak{psq}(n)$ &  $\emptyset$\\
$ D(2,1,\alpha)$   &  $\{\pm\ (\epsilon, -\epsilon, \epsilon)\}$ \\
$G(3)$  &  $\{\pm\ (\omega_{1},-\epsilon)\}$ \\
$F(4)$  &  $\{\pm\ (\omega_{3},-\epsilon)\}$ \\
\end{tabular}
\end{center}
\label{T:detectingsub}
\end{table}

\renewcommand{\arraystretch}{1.2}
\begin{table}[htp]
\caption{Hyperplanes of BBW parabolics}
\begin{center}
\begin{tabular}{cc}
${\mathfrak g}$ &  $\mathcal H$ \\ \hline
$\mathfrak{gl}(m|n), \mathfrak{sl}(m|n)$ [$m\leq n$] &  $\sum_{i=1}^n x_i (E_i + D_i)$,  $x_1>x_2>\cdots > x_n$    \\   
$\mathfrak{osp}(2m|2n)$ &  $\sum_{i=1}^{r} x_i (E_i + D_i)$,  $x_1>x_2>\cdots > x_r > 0$ [$r = \max (m,n)$]\\  
 $\mathfrak{osp}(2m+1|2n)$ & $ \sum_{i=1}^{r} x_i (E_i + D_i)$,  $x_1>x_2>\cdots > x_r > 0$ [$r = \max (m,n)$] \\                                                                     
${\mathfrak q}(n), \mathfrak{psq}(n)$ &  $\sum_{i=1}^n x_i E_i $,  $x_1>x_2>\cdots > x_n$ \\
$ D(2,1,\alpha)$   &    $x_1E_1 +(x_1+x_3)E_2 + x_3 E_3$,  $x_1 > x_3 > 0$                                   \\
$G(3)$  &                 $x_1L_1 +x_2 L_2 + x_1 E$,  $2x_1 > x_2 > x_1>0$                                         \\
$F(4)$  &                  $x_1L_1 +x_2 L_2 + x_3L_3 + x_3 E$,  $2x_1 > x_3> x_2 > x_1>0$                                \\
\end{tabular}
\end{center}
\label{T:BBW-hyperplanes}
\end{table}

\renewcommand{\arraystretch}{1.2}
\begin{table}[htp]
\caption{Roots of BBW parabolics}
\begin{center}
\begin{tabular}{cc}
${\mathfrak g}$ &  ${\Phi}_{\1}^{-}$ \\ \hline
$\mathfrak{gl}(m|n), \mathfrak{sl}(m|n)$ [$m\leq n$] &  $\{- \epsilon_i + \delta_j, - \delta_i + \epsilon_j \; | \; i < j \} $    \\   
$\mathfrak{osp}(2m|2n)$ &  $\{- \epsilon_i + \delta_j, - \delta_i + \epsilon_j, - \epsilon_k - \delta_{\ell}, \; | \; i < j \}$\\  
 $\mathfrak{osp}(2m+1|2n)$  [$m\geq n$]  & $ \{ - \epsilon_i + \delta_j, -\delta_i + \epsilon_j, - \epsilon_k - \delta_{\ell},   - \delta_{t} \; | \; i < j\}$ \\  
  $\mathfrak{osp}(2m+1|2n)$  [$m <  n$]  &  $ \{ - \epsilon_i + \delta_j, -\delta_i + \epsilon_j, - \epsilon_k - \delta_{\ell},   - \delta_{t} \; | \; i < j, t \leq m\}$ \\                                                                     
${\mathfrak q}(n), \mathfrak{psq}(n)$ & $ \{ - \epsilon_i + \epsilon_j  | \; i < j \}$ \\
$ D(2,1,\alpha)$   &  $\{(-\epsilon,-\epsilon,-\epsilon), (-\epsilon,-\epsilon,\epsilon), (\epsilon,-\epsilon,-\epsilon)\}$ \\
$G(3)$  &  $\{(-\omega_{1}+\omega_{2},-\epsilon), (2\omega_{1}-\omega_{2},-\epsilon),  (0,-\epsilon), (\omega_{1}-\omega_{2},-\epsilon),$ \\
              &  $(-2\omega_{1}+\omega_{2},-\epsilon), (-\omega_{1}, -\epsilon) \}$ \\
$F(4)$  &  $\{(\omega_{2}-\omega_{3},-\epsilon), (\omega_{1}-\omega_{2}+\omega_{3},-\epsilon),  (\omega_{1}-\omega_{3},-\epsilon),$ \\ 
              &  $(-\omega_{2}+\omega_{3},-\epsilon), (-\omega_{1}+\omega_{2}-\omega_{3}, -\epsilon),(-\omega_{1}+\omega_{3},-\epsilon),(-\omega_{3},-\epsilon) \}$ \\                                            
\end{tabular}
\end{center}
\label{T:BBW-roots}
\end{table}

\newpage

\subsection{Poincar\'e Series} For the parabolic subalgebras ${\mathfrak b}$ given in Section~\ref{SS:defineBBW}, the table below provides gives 
a description of the cohomology $\operatorname{H}^{\bullet}({\mathfrak b},{\mathfrak b}_{\0},{\mathbb C})$ and relationship between 
$z_{{\mathfrak b},{\mathfrak g}}(t)$ with $p_{W_{\1}}(t)$. 

\renewcommand{\arraystretch}{1.2}
\begin{table}[htp]
\caption{Cohomology and Hilbert Series}
\begin{center}
\begin{tabular}{cccc}\label{T:Poincareseries}
${\mathfrak g}$ &  $W_{\1}$ &$\operatorname{H}^{\bullet}({\mathfrak b},{\mathfrak b}_{\0},{\mathbb C})$ & $z_{{\mathfrak b},{\mathfrak g}}(t)$ \\ \hline
$\mathfrak{gl}(m|n)$ [$m\geq n$] &  $\Sigma_{n}$   & ${\mathbb C}[x_{1}y_{1},x_{2}y_{2},\dots,x_{n}y_{n}]         $ & $p_{W_{\1}}(t^{2})$      \\
$\mathfrak{sl}(m|n)$ [$m> n$] &  $\Sigma_{n}$   & ${\mathbb C}[x_{1}y_{1},x_{2}y_{2},\dots,x_{n}y_{n}]         $ & $p_{W_{\1}}(t^{2})$      \\
$\mathfrak{sl}(n|n)$ &  $\Sigma_{n}$   & ${\mathbb C}[x_{1}y_{1},x_{2}y_{2},\dots,x_{n}y_{n},x_{1}x_{2}\dots x_{n},y_{1}y_{2}\dots y_{n}]$  & $p_{W_{\1}}(t^{2})$        \\                                                   
$\mathfrak{psl}(n|n)$ & $\Sigma_{n}$  & ${\mathbb C}[x_{1}y_{1},x_{2}y_{2},\dots,x_{n}y_{n},x_{1}x_{2}\dots x_{n},y_{1}y_{2}\dots y_{n}]$  & $p_{W_{\1}}(t^{2})$        \\                                                                            
${\mathfrak q}(n)$ & $\Sigma_{n}$ &  ${\mathbb C}[z_{1},z_{2},\dots, z_{n}]$         & $p_{W_{\1}}(t)$  \\
$\mathfrak{psq}(n)$ &  $\Sigma_{n}$ & ${\mathbb C}[z_{1},z_{2},\dots, z_{n},z_{1}z_{2}\dots z_{n}]/(z_{1}+z_{2}+\dots+ z_{n})$          & $p_{W_{\1}}(t)$  \\                                  
$\mathfrak{osp}(2m+1|2n)$ & $\Sigma_{r}\ltimes ({\mathbb Z}_{2})^{r}$ & ${\mathbb C}[x_{1}y_{1},x_{2}y_{2},\dots,x_{r}y_{r}] $  [$r=\text{min}(m,n)$]      & $p_{W_{\1}}(t^{2})$  \\
$\mathfrak{osp}(2m|2n)$ [$m>n$] & $\Sigma_{n}\ltimes ({\mathbb Z}_{2})^{n}$ & ${\mathbb C}[x_{1}y_{1},x_{2}y_{2},\dots,x_{n}y_{n}] $      & $p_{W_{\1}}(t^{2})$ \\
$\mathfrak{osp}(2m|2n)$ [$m\leq n$] & $\Sigma_{m}\ltimes ({\mathbb Z}_{2})^{m-1}$ & ${\mathbb C}[x_{1}y_{1},x_{2}y_{2},\dots,x_{m}y_{m}] $        & $p_{W_{\1}}(t^{2})$   \\
$D(2,1,\alpha)$                                 & $\Sigma_{2}$ &  ${\mathbb C}[xy]$                             & $p_{W_{\1}}(t^{2})$        \\                                                   
$G(3)$                                             & $\Sigma_{2}$ &  ${\mathbb C}[xy]$                             & $p_{W_{\1}}(t^{2})$        \\                                                      
$F(4)$                                             & $\Sigma_{2}$ &  ${\mathbb C}[xy]$                             & $p_{W_{\1}}(t^{2})$        \\                                                      
\end{tabular}
\end{center}
\label{T:dimensions}
\end{table}

\end{document}